\pdfoutput=1
\documentclass[preprint,svgnames,dvipsnames]{elsarticle}

\usepackage[OT1]{fontenc}
\usepackage{amsmath,amsfonts, amssymb,amsthm}
\usepackage[numbers]{natbib}
\usepackage[colorlinks,citecolor=blue,urlcolor=blue]{hyperref}
\usepackage[left=3cm,right=3cm,top=3cm,bottom=3cm]{geometry}
\usepackage{algorithm}
\usepackage{algorithmicx}
\usepackage{algpseudocode}
\usepackage{bm}
\usepackage{graphicx}
\usepackage{dsfont}
\usepackage{tabularx}
\usepackage{stmaryrd}
\usepackage[svgnames,dvipsnames]{xcolor}
\usepackage{listings}

\newcommand{\overI}{\overline{\mathcal{I}}}
\newcommand{\underI}{\underline{\mathcal{I}}}
\newcommand{\overunderI}{\underline{\overline{\mathcal{I}}}}
\newcommand{\E}[2]{\mathbb{E}_{ #1 } \left[ #2 \right]}
\newcommand{\EE}{\mathbb{E}}
\newcommand{\1}{\mathds{1}}

\newtheorem{thm}{Theorem}
\newtheorem{propo}{Proposition}
\newtheorem{lme}{Lemma}
\newtheorem{dfi}{Definition}
\newtheorem{rmk}{Remark}
\newtheorem{example}{Example}

\makeatletter
\def\ps@pprintTitle{%
\let\@oddhead\@empty
\let\@evenhead\@empty
\def\@oddfoot{}%
\let\@evenfoot\@oddfoot}
\makeatother

\definecolor{backcolour}{rgb}{0.95,0.95,0.92}
\begin{document}

\begin{frontmatter}

%%%%%%%%%%%%%%%%%%%%%%%%%%%%%%%%
% Title & runtitle
\title{Quantile-constrained Wasserstein projections for robust interpretability of numerical and machine learning models}

%%%%%%%%%%%%%%%%%%%%%%%%%%%%%%%%
% Authors information

\author[a,b,c,e]{Marouane Il Idrissi}
\author[a,b,d]{Nicolas Bousquet}
\author[c]{Fabrice Gamboa}
\author[a,b,c]{Bertrand Iooss}
\author[c]{Jean-Michel Loubes}

\address[a]{EDF Lab Chatou, 6 Quai Watier, 78401 Chatou, France}
\address[b]{SINCLAIR AI Lab., Saclay, France}
\address[c]{Institut de Mathématiques de Toulouse, 31062 Toulouse, France}
\address[d]{Sorbonne Université, LPSM, 4 place Jussieu, Paris, France}
\address[e]{Corresponding Author - Email: marouane.il-idrissi@edf.fr}

%%%%%%%%%%%%%%%%%%%%%%%%%%%%%%%%%%%%%%%%%%%%%%%%%
% ABSTRACT
% Maximum length: 200 words
\begin{abstract}
Robustness studies of black-box models is recognized as a necessary task for numerical models based on structural equations and predictive models learned from data. These studies must assess the model's robustness to possible misspecification of regarding its inputs (e.g., covariate shift). The study of black-box models, through the prism of uncertainty quantification (UQ), is often based on sensitivity analysis involving a probabilistic structure imposed on the inputs, while ML models are solely constructed from observed data. Our work aim at unifying the UQ and ML interpretability approaches, by providing relevant and easy-to-use tools for both paradigms. To provide a generic and understandable framework for robustness studies, we define perturbations of input information relying on quantile constraints and projections with respect to the Wasserstein distance between probability measures, while preserving their dependence structure. We show that this perturbation problem can be analytically solved. Ensuring regularity constraints by means of isotonic polynomial approximations leads to smoother perturbations, which can be more suitable in practice. Numerical experiments on real case studies, from the UQ and ML fields, highlight the computational feasibility of such studies and provide local and global insights on the robustness of black-box models to input perturbations.
\end{abstract}

%%%%%%%%%%%%%%%%%%%%%%%%%%%%%%%%
% Keywords


%%%%%%%%%%%%%%%%%%%%%%%%%%%%%%%%
% TOC
%\tableofcontents

\end{frontmatter}

%%%%%%%%%%%%%%%%%%%%%%%%%%%%%%%%%%%%%%%%%%%%%%%%%
% INTRODUCTION
\section{Introduction}\label{sec:intro}
Multiple engineering fields require models for prediction and phenomenological understanding. Machine learning (ML) and uncertainty quantification (UQ) of numerical models are two essential approaches to developing and manipulating such models. Because they require, for their enlightened use, an adequate understanding of their characteristics, they share fundamental similarities. These two frameworks feed off each other through the duality of sensitivity analyses (SA), a fundamental methodological corpus in UQ, and ML interpretability methods, as explained by \cite{RAZAVI2021104954,iooken22}. In particular, recent advances in explainable ML leverage tools from SA to produce meaningful interpretations of black-box models \citep{Fel2021, Benard2022}, and novel SA estimation schemes are heavily based on the construction of suitable ML models \citep{Broto2020, Benard_shaff_2022}. Both SA and ML interpretability especially rely on the definition, estimation, and manipulation of diagnostics related to the characteristics of a model and how its behavior depends on its inputs \citep{DaVeiga2021,molnarbook2018,samek_explainable_2019}. Formally speaking, let a model $f$ be defined as a mapping between \emph{inputs} $X \in \mathcal{X}$ and \emph{outputs} $Y \in \mathcal{Y}$ where $(\mathcal{X},\mathcal{Y})$ are two metric spaces: 
$$Y = f(X).$$
In a ML context, $f$ is defined as a  \emph{predictive model} (e.g., penalized linear regression, neural network) linking a feature (or covariate) instance $X$ to a prediction $Y$ \citep{Hastie2009}. In the UQ framework, a so-called \emph{computer model} $f$ represents the numerical implementation of a hypothetical-deductive link (e.g., by systems of ordinary differential equations, by finite element methods) between $X$ and $Y$ \citep{Smith_book_2014}. 
  
In both fields,  the input $X\in\mathcal{X}$ is generally assumed to be random, inducing a general framework for handling uncertainties about the latter. Let $P$ be the distribution of $X$. In the ML context, $P$ is defined implicitly by an empirical measure: given a set of observations $x^{(1)}, \dots, x^{(n)} \in \mathcal{X}$,
\begin{equation}
    P = \frac{1}{n}\sum_{i=1}^n \delta_{x^{(i)}}
    \label{eq:empMeas}
\end{equation}
where $\delta$ denotes the Dirac measure. On the other hand, in the UQ setting, $P$ is often explicitly chosen based on observations of $X$, expert assessment (domain knowledge), or stochastic inversion from observations of $Y$ \citep{Sullivan2017}. The diagnostics mentioned above correspond to estimations of key interpretable features of $Y$, or \emph{quantities of interest} (QoI). In the SA literature, such a QoI is often referred as the \emph{score}  \citep{Rubinstein1989}, while ML researchers rather speak about \emph{predictive performance}  \citep{Paananen2019}.  Recall that SA aims to rank the dimensions of $X$ according to their influence on this QoI \citep{cac03}. For instance, in local SA, it is usually computed by way of a differential operator with respect to (w.r.t.) the dimensions of $X$ \citep{DaVeiga2021}. In global SA, it can be typically chosen as the output's variance or its quantiles \citep{Fel2021,MAUMEDESCHAMPS2018122}. More general objects characterizing the distribution of $Y$ (e.g., a kernel embedding \citep{Barr2022}) can also be of interest, possibly at the cost of a less immediate interpretability. In ML interpretability, local methods focus on a particular prediction instance, letting the QoI be the identity function \citep{Stevens2020} or a local linear decomposition of $f$ \citep{Visani2022}. Global ML interpretability often relies on QoI defined as performance metrics (e.g., accuracy, loss value) of  $f$ computed over a training dataset \citep{Gromping2015,Covert2020,Iooss2022}. 

Accordingly, the use of these diagnostics to allow different levels of interpretability is subject to the same robustness problem: they must remain relevant when $P$ suffers from \emph{misspecification}. It would improve the confidence in both the usage and the insights that an ML model offers \citep{Barredo2020}. In this framework it is fundamentally connected to problems of domain adaptation and transfer learning \citep{Bruzzone2010,Lemberger2020} (e.g., when the data used for the design of $f$ suffer from selection bias w.r.t. operational data), including in particular the robustness to \emph{covariate shift} \citep{SHIMODAIRA2000227,Tripuraneni2021} (see Figures \ref{fig:covariate.shift1} and \ref{fig:covariate.shift2} for an illustration). More generally, whether in UQ or ML, this misspecification is due to the epistemic uncertainty affecting the knowledge of the properties of $P$ (e.g., support, geometry, topology) due to the finiteness of the available information (e.g., data, expertise, boundary conditions) \citep{Hullermeier2021,Sullivan2017}. 

\begin{figure}[b!]
 \centering
    \includegraphics[width=0.75\linewidth]{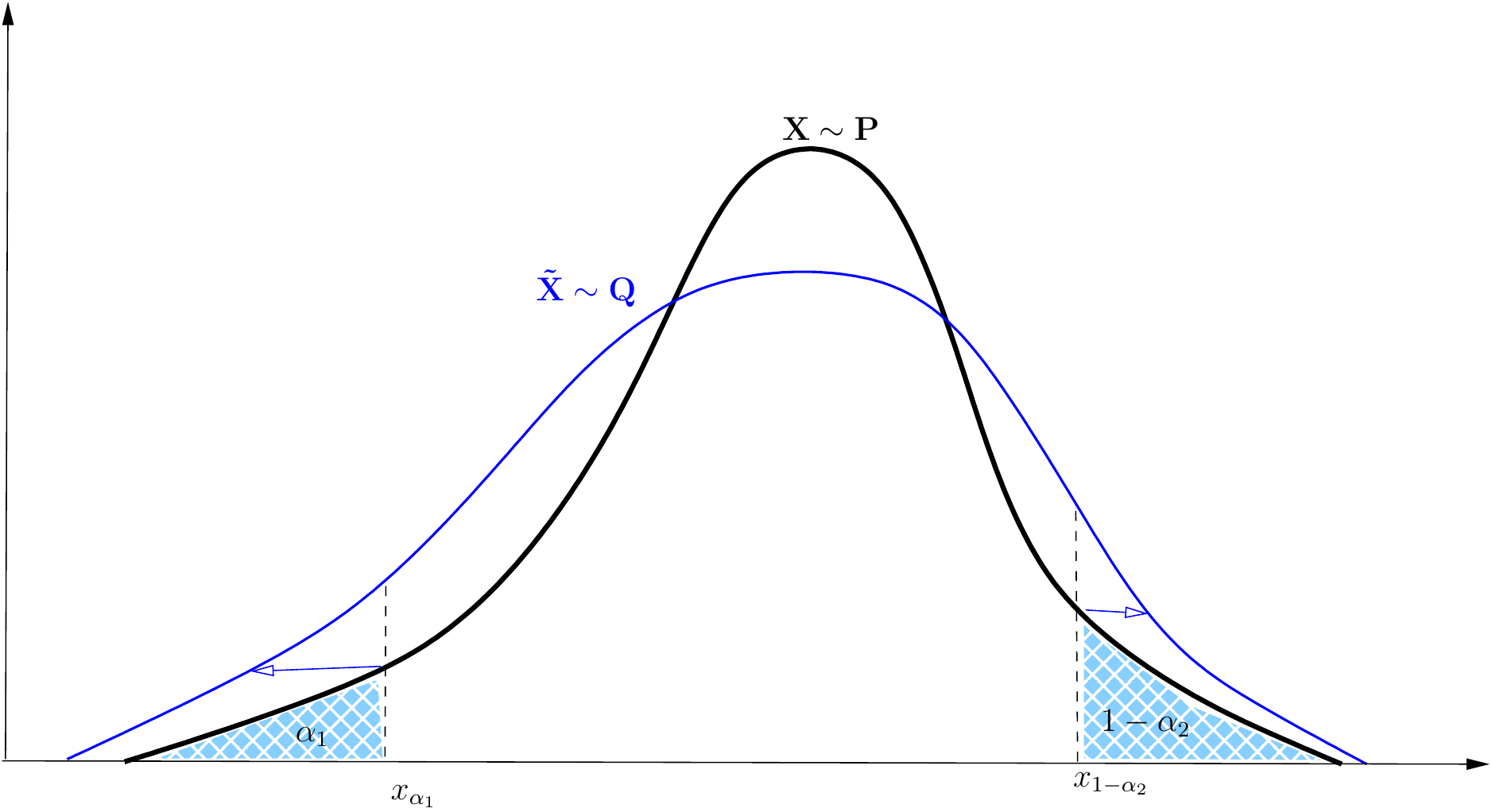}
    \caption{One-dimensional illustration of covariate shift in the UQ framework. A postulated density distribution for $X\sim P$ (solid black line) is dilated by modifying the tail order of low and high quantiles defining an application domain (blue arrows), resulting in a new distribution $Q$. }
    \label{fig:covariate.shift1}
\end{figure}

\begin{figure}[t!]
 \centering
    \includegraphics[width=0.75\linewidth]{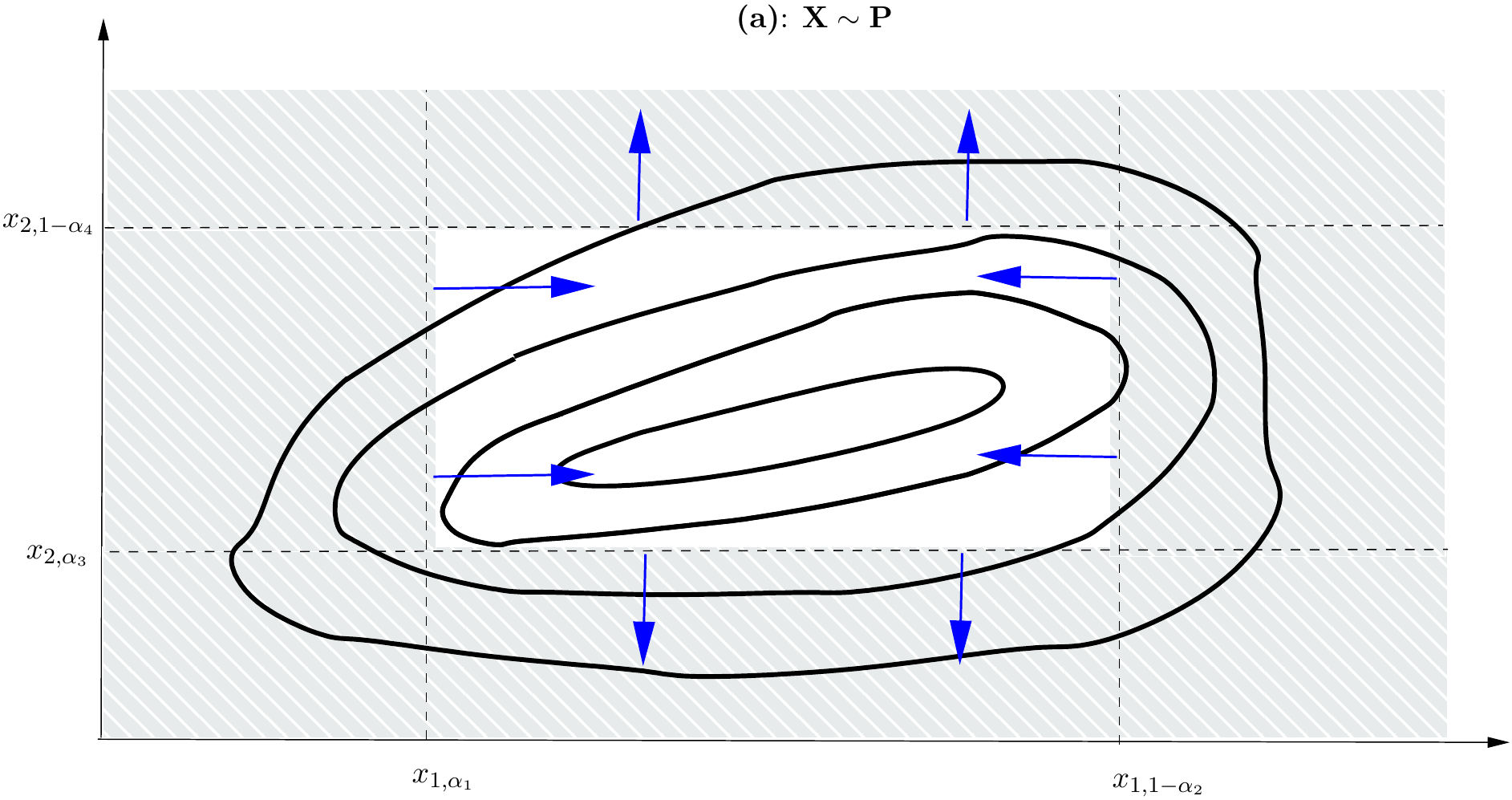}
    \includegraphics[width=0.75\linewidth]{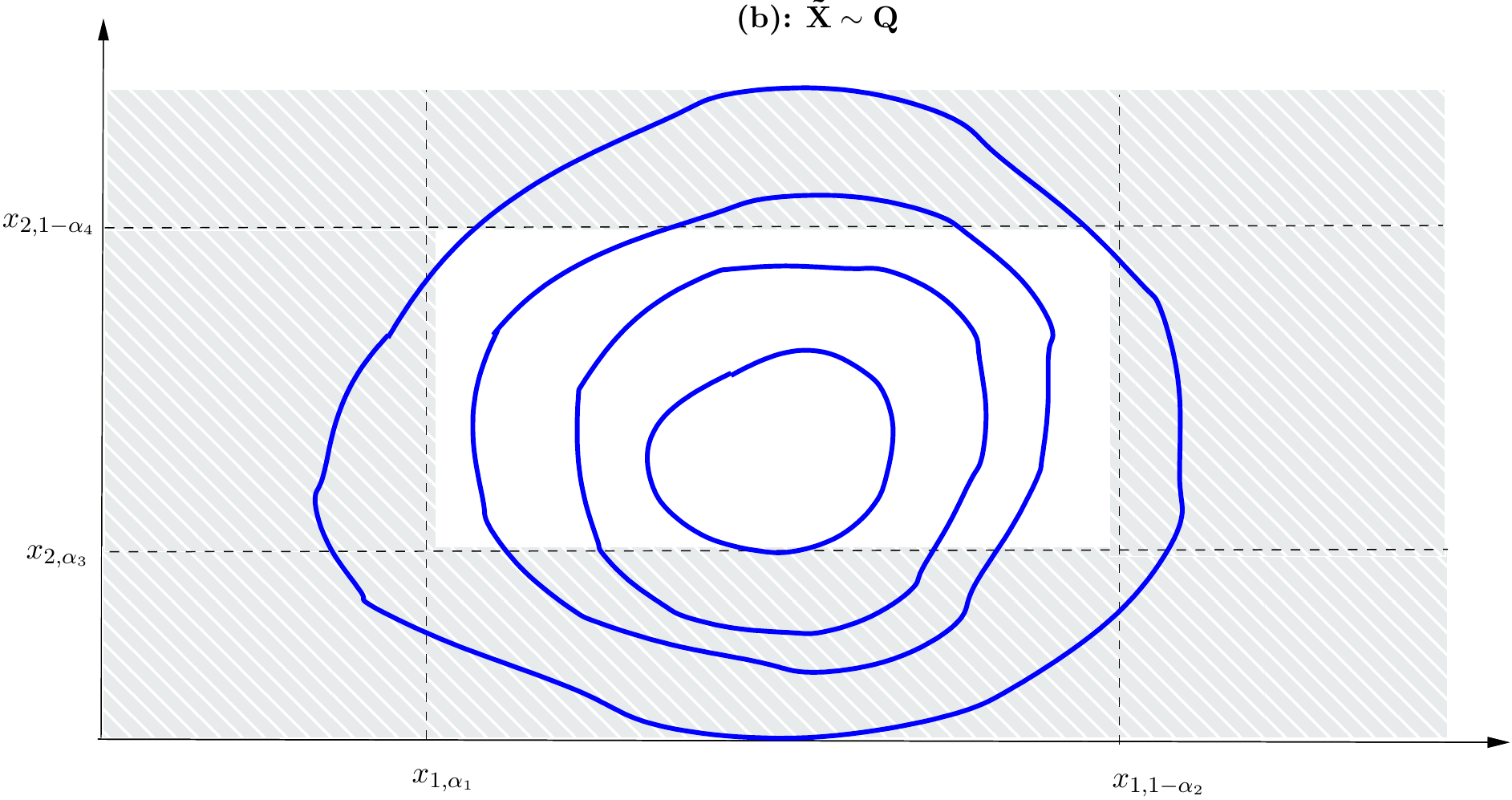}
    \caption{Two-dimensional illustration of covariate shift in the UQ framework. Graph {\bf (a)} displays the isodensity curves of a joint distribution $P$ on $X=(X_1,X_2)$. Graph {\bf (b)} displays the result of a simultaneous quantile shrink for $X_1$ and quantile dilatation for $X_2$ while preserving the dependence structure.  }
    \label{fig:covariate.shift2}
\end{figure}

This epistemic uncertainty about $P$ thus impacts the generalization capability of $f$. It is essential to define a practical \emph{application domain} \citep{ROY20112131}, or usage domain of $f$, allowing to extrapolate beyond a validity domain based on selected situations (e.g., a domain including unobserved data points in ML, or unsure structural mechanisms in numerical models). Generalization bounds in ML usually depend on notions of capacities (e.g., the Vapnik–Chervonenkis dimension \cite{Vapnik1971}). They rely on a sufficient number of observations to obtain generalization, i.e., linked to epistemic uncertainty reduction. Hence, generalization power and epistemic uncertainties are intimately linked \citep{VallePerez2020}. Therefore, robustness studies of the application domain's uncertainty are particularly important for both fields. It motivated the search for adversarial situations in ML \citep{AthalyeEIK18} and the use of imprecise probabilities in numerical model exploration \citep{BEER20134}. 

The need to conduct robustness studies raises the question of how to reasonably vary $P$ in a \emph{perturbation class} ${\cal{C}}\subset \mathcal{P}(\mathcal{X})$, where  $\mathcal{P}(\mathcal{X})$ is the set of probability measures supported on $\mathcal{X}$, by modifying the mass given to different parts of $\mathcal{X}$ in an interpretable way. Such variations allow checking the consistency of the information delivered by these diagnostics. Here, \textit{reasonableness} means two things: a perturbed distribution $Q$ must remain close to the original $P$ in some formalized sense, and the computations required for the robustness analyses must have a moderate cost. 

Several authors in both fields have conducted this type of study. In the ML setting, finding adversarial situations or detecting bias affecting a QoI can stand on moving the original data (e.g., by optimal transport \citep{Pydi2020,pmlr-v97-gordaliza19a}), resulting in an empirical shift of $P$. Researchers in UQ distinguish between \emph{distributional sensitivity analysis} (DSA) \citep{Allaire2010DistributionalSA,Narayan2011} which aims at providing a robustness diagnosis on a QoI used in SA (or connected similar studies, e.g., in financial risk analysis \citep{Cont2008RobustnessAS}), and other robustness approaches whose goal is to study the sensitivity of a given QoI to a gradual modification of $P$ in ${\cal{C}}$. For instance \cite{Lemaitre2015}, \cite{Gauchy2021} and \cite{Iooss2021} proposed such indices for exceedance probabilities, quantiles or superquantiles of the distribution of $Y$ under various perturbation schemes. 

Many choices have been proposed for perturbation classes. They are often indexed by parameters $\theta$ and denoted $\mathcal{C}(\theta)$. For instance, in the SA setting, they can be defined as contamination classes for Bayesian SA \citep{Gustafson1995,Roos2015}, as classes of generalized moments \citep{Lemaitre2015,Bachoc2020} or as classes defined by Fisher-Rao derivatives \citep{Kurtek2015,Gauchy2021}, among others. A unified view of these approaches is to define the perturbed distribution $Q$ as the solution to the projection problem
\begin{align}
    \begin{split}
        Q = \underset{G \in \mathcal{P}(\mathcal{X})}{\textrm{argmin }} & \quad \mathcal{D}(P,G) \\
    \textrm{s.t. }& \quad G \in \mathcal{C}(\theta).
    \label{eq:pert_prob}
    \end{split}
\end{align}
where $\mathcal{D}$ is a discrepancy between probability measures and $\mathcal{C}(\theta)$ denotes a fixed perturbation class. The Kullback-Leibler (KL) divergence is often chosen as a suitable discrepancy (e.g. in \cite{Lemaitre2015,Bachoc2020}), leading to entropic projections \cite{Csiszar1975}. For instance in those developments, whenever $\mathcal{X}\subset \mathbb{R}^d$ ($d\geq 1$), ${\cal{C}}(\theta)$ can be defined as the distributions belonging to $\mathcal{P}(\mathcal{X})$ with a deviated mean $\eta$, different from $\E{P}{X}$. Reusing the term proposed by \cite{Bachoc2020}, $\theta = \|\eta-\EE[X]\|$ can be called a \emph{perturbation intensity}, where $\theta$ belongs to an ordered set $\Theta$. Hence, for a fixed perturbation intensity $\theta$, the solution $Q$ of (\ref{eq:pert_prob}) is the optimally perturbed distribution of $P$ w.r.t. $\mathcal{C}(\theta)$. 

Although rational, moment-based perturbations of $P$ presents significant issues that limit the conclusions of interpretability analyses, and the subsequent unification of UQ and ML methodologies. First, the choices of $\theta$, $\mathcal{C}(\theta)$, and $\Theta$ remain subjective and submitted to strong assumptions. For instance, in \cite{Lemaitre2015}, some generalized moment of $Q$ are required to exist, and the moment deviation intensity must subjectively chosen by the practitioner. When $\mbox{dim}(X)$ is large, such assessments seem challenging to make. In \cite{Gauchy2021}, $\Theta$ comprises the values taken by the derivative of the Fisher metric, requiring regularity properties on $P$, which is often assumed to be in a restrictive parametric family, for computational reasons. In general, determining the boundaries of $\Theta$ remains challenging.    

Second, choosing the KL divergence or the derivative of the Fr\'echet metric as in \cite{Hart2019} as a discrepancy imposes strong restrictions on the space of reachable probability distributions. It result in a search on a restricted part of $\mathcal{P}(\mathcal{X})$ in (\ref{eq:pert_prob}). For instance, using the KL divergence implies that $Q$ should be absolutely continuous w.r.t. $P$, which does not allow to consider continuous perturbed distributions $Q$ given an empirical initial distribution $P$. Finally, the behavior of these discrepancies (purely related to information geometry) remains uneasy to explain to non-specialists.

This article, therefore, aims to address both of these issues, improving the connections between the interpretability analyses conducted in UQ and ML settings, and their robustness. More precisely, our contributions are twofold:
\begin{description}
\item[(a)] We propose a meaningful and generic approach to perturb distributions through marginal quantile constraints, selecting the 2-Wasserstein distance for ${\cal{D}}$. These choices allow to solve the problems mentioned above. A key point of the proposed methodology is that no strong regularity assumptions must be assumed on $f$ and it does not rely on $f$ having a particular structure (e.g., belonging to a particular family of predictive models or based on specific physical equations). This effectively unifies UQ and ML approaches.
\item[(b)] We demonstrate the computational tractability of this methodology by implementing it on different types of QoI, on both numerical and predictive models, and studying some robustness indicators to perturbations in the particular case where $f$ is considered to be a black box.
\end{description}

This article is organized as follows. In Section~\ref{sec:distribPert} we first introduce useful notations and definitions. Section~\ref{sec:qcPert} develops and motivates the choice of quantiles as an interpretable and generic basis for defining meaningful perturbations and defines perturbation schemes relevant for different robust interpretability studies. Section~\ref{sec:wassProj} presents the general framework of probability measure projection using the $2$-Wasserstein distance and proposes analytical solutions and numerical optimization schemes for solving the input perturbation problem through the help of isotonic polynomials. Section~\ref{sec:robust_meas} showcases our method on two use-cases from both the ML and UQ fields, from which local and global robustness insights are highlighted. A discussion section ends this article, opening avenues for improvement. All proofs of technical results are postponed to a  dedicated appendix.

%%%%%%%%%%%%%%%%%%%%%%%%%%%%%%%%%%%%%%%%%%%%%%%%%
% Section 2 : General framework of distributional perturbation
\section{Notations and main definitions}\label{sec:distribPert}
Let us introduce some useful notations and definitions. Let $d$ and $p$ be two positive integers. Let $\mathcal{X}$ be a subset of $\mathbb{R}^d$, and $\mathcal{P}_p(\mathcal{X})$ be the subset of $\mathcal{P}(\mathcal{X})$ of measures with finite $p$-th moment. The initial probability measure $P$ is either defined through an explicit distribution, or empirically, as in (\ref{eq:empMeas}). We will denote by $Q\in\mathcal{P}(\mathcal{X})$ the optimally perturbed distribution of $P$. 

Furthermore, let us denote $\Omega_X\subset\mathcal{X}$ the application domain. It is the subset of $\mathcal{X}$ where $f$ is intended to be used for predictions \citep{ROY20112131}. Both in ML and UQ, given a set  ${\bf x_n}=\{x_1,\ldots,x_n\}$ of training, validating or testing examples, the convex hull of ${\bf x_n}$  or a broader span of ${\bf x_n}$ are common candidates to define $\Omega_X$ \citep{Yousefzadeh2021}. More generally, $\Omega_X$ is an extrapolation domain where $f$ is assumed to generalize well (e.g., a paving of a compact subspace of $\mathcal{X}$ selected by tree-based classification \citep{Hooker2004b}, confidence measures or cross-validation schemes \citep{Hooker2012,Neyshabur,Wang2021,pmlr-v139-krueger21a}). In ML specifically, including out-of-distribution data in $\Omega_X$ remains an open problem \citep{9710159,Wang2021,Shen2021}. 

To echo the classical assumptions in ML and UQ, we assume that $\Omega_X$ is the union of compact subsets of $\mathcal{X}$. These subsets can be defined under some uncertainties, typically on their bounds. In a robustness analysis perspective, assuming that the dependence structure is maintained, the uncertainties on $\Omega_X$ can be interpreted as variations on the values (or thresholds) of the extreme quantiles of marginal distributions. Figure~\ref{fig:sch_opeDomain} illustrates a typical situation for a univariate marginal of $X$. 

\begin{figure}[b!]
 \centering
    \includegraphics[width=\linewidth, trim= 0.75cm 0cm 0cm 0cm]{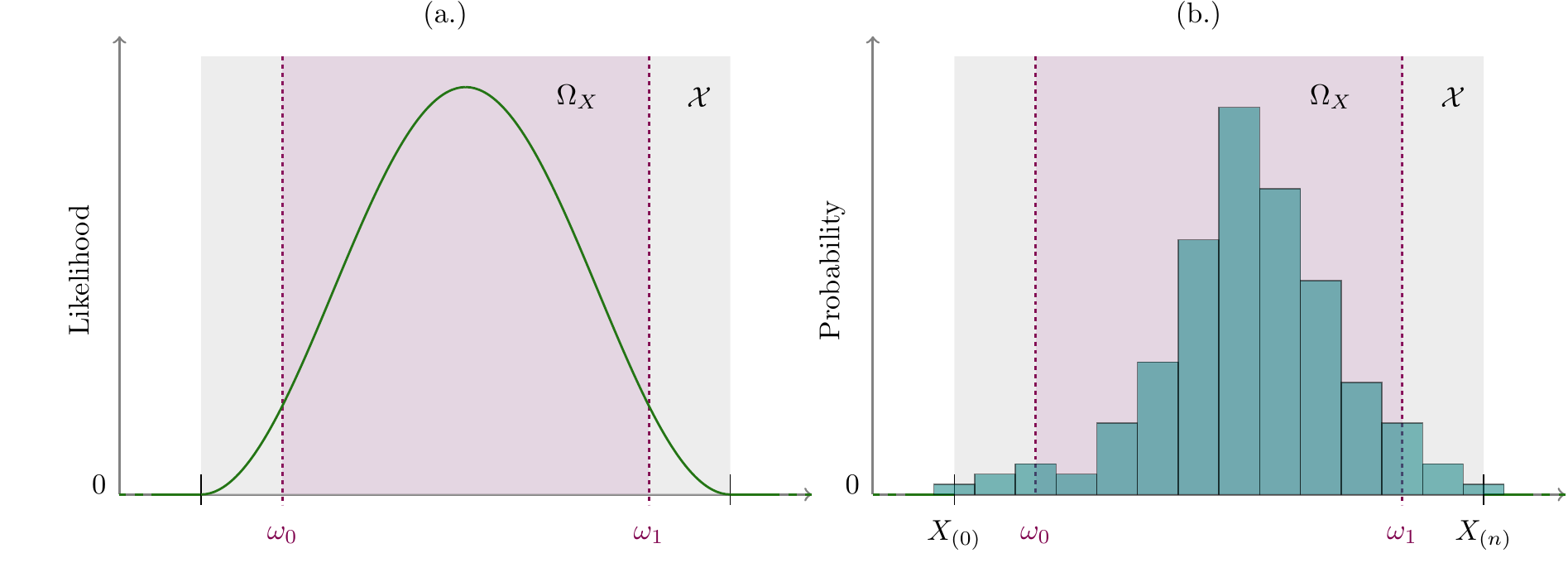}
    \caption{Application domain $\Omega_X \subset \mathcal{X}$ in the UQ (a.) and ML (b.) settings. In the UQ setting, $\mathcal{X}$ is the support of the explicitly chosen density (in green). In the ML setting, $\mathcal{X}$ is the interval between the minimum and maximum observed values. In both cases, $\Omega_X$ is included in $\mathcal{X}$, although it is not mandatory. }
    \label{fig:sch_opeDomain}
\end{figure}

The upcoming developments deal with perturbations on the marginal distributions of $P$. The following definitions recall classical results on cumulative distribution functions (cdf) and generalized quantile functions (gqf) of univariate probability measures. Denote $\mathcal{F}$ the space of univariate distribution functions:
\begin{align}
    \begin{split}
        \mathcal{F} = \Bigl\{ F &: \mathbb{R} \rightarrow [0,1] \mid F \text{ is right-continuous, non-decreasing} \\
        &\text{such that }\underset{x \rightarrow \infty}{\text{lim}} F(x) = 1 \ \text{and} 
         \underset{x \rightarrow -\infty}{\text{lim}} F(x) = 0\Bigr\}.
        \label{eq:CDF_space}
    \end{split}
\end{align}
For any $P \in \mathcal{P}(\mathbb{R})$, $\mathcal{F}$ contains the cdf of $P$, defined as:
\begin{equation*}
    F_P(t) = \int_{(-\infty, t]} dP = P\bigl((-\infty, t] \bigr).
\end{equation*}
The gqf of any probability measure $P$ can be formally defined as follows \citep{Resnick1987, Dufour1995, Kallenberg2021}.
\begin{dfi}[Generalized quantile function]
Let $P \in \mathcal{P}(\mathbb{R})$ with cdf $F_P$.
\begin{description}
\item[(i)] The generalized quantile function (gqf) of $P$ is the unique left-continuous, non-decreasing generalized inverse of $F_P$, defined, for every $a \in (0,1)$, as:
\begin{align}
    \begin{split}
         F^\leftarrow_P(a)  &= \sup~\{ t \in \mathbb{R} \mid F_P(t) < a \}, \\
                            &= \inf~\{ t \in \mathbb{R} \mid F_P(t) \geq a \}.
         \label{eq:quantile_func}
    \end{split}
\end{align}
\item[(ii)] The unique right-continuous non-decreasing generalized inverse $F^{\rightarrow}_P$ of $F_P$, almost-everywhere equal to $F^{\leftarrow}_P$, is defined, for every $a \in (0,1)$,  as:
\begin{align}
    \begin{split}
        F^\rightarrow_P(a)  &= \sup~\{ t \in \mathbb{R} \mid F_P(t) \leq a \}, \\
                            &= \inf~\{ t \in \mathbb{R} \mid F_P(t) > a \}, \\
                            &= F^{\leftarrow}_P \left( a^+ \right)
    \end{split}
\end{align}
where $F^{\leftarrow}_P(a^+)=\lim\limits_{x\to a^+} F^{\leftarrow}_P(x)$.
\end{description}
\label{def:genQuant}
\end{dfi}
It is important to note that when $F_P$ admits an inverse $F^{-1}_P$ in the traditional sense (e.g., it is continuous, strictly increasing), then the following equality holds:
$$F^{-1}_P = F^\leftarrow_P = F^\rightarrow_P,$$
which echoes and generalizes the traditional definition of a quantile function as the inverse of a cdf.

In the remainder of this work, any probability measure is assumed to have a finite 2-nd order moment; hence, their gqf is assumed to be square integrable. Accordingly, the perturbation  $\theta$ considered in the projection problem (\ref{eq:pert_prob}) corresponds to quantile constraints placed on one-dimensional marginal variables. Therefore it is assumed that $\Theta \subset \mathbb{R}$.

%%%%%%%%%%%%%%%%%%%%%%%%%%%%%%%%%%%%%%%%%%%%%%%%%%%%%%%%%%%%%%%%%%%%%%%%%%%%%%
% SECTION 3 Quantile Perturbations
\section{Quantile perturbations}\label{sec:qcPert}
In this section, we define a particular set of perturbation classes. They are characterized as a set of constraints on the (generalized) quantiles of the marginal components of the inputs (or features) variables $X$ of $f$. Before presenting a formal definition of such classes, we argue their relevance for practical studies.

%%%%%%%%%%%%%%%%%%%%%%%%%%%%%%%%%%%%%%%%%%%%%%%%%%%%%%%%%%%%%%%%%%
% 3.1 Motivations
\subsection{Motivations}\label{sec:motivations}

First, \emph{generalized quantiles always exist}, unlike generalized moments considered by \cite{Lemaitre2015}. It is also what motivated \cite{Bachoc2020}, who proposed to define the perturbation range $\Theta$ in (\ref{eq:pert_prob}) as a compact set bounded by empirical quantiles. More precisely, recalling Definition~\ref{def:genQuant}, a gqf is the left-generalized inverse of a cdf. Moreover, they allow to uniquely characterize probability measures in $\mathcal{P}(\mathbb{R})$.
\begin{rmk}\label{prop:uniqueQuantFunc}
The unique link between left-generalized inverses of functions in $\mathcal{F}$ and probability measures can be formalized as follows. Let $F^{\leftarrow}$ be a function in the set:
\begin{align}
    \begin{split}
        \mathcal{F}^{\leftarrow} = \Bigl\{&F^{\leftarrow} : (0,1) \rightarrow \mathbb{R} \mid F^{\leftarrow} \text{ is left-continuous and non-decreasing} \Bigr\}.
        \label{eq:quantFunc_space}
     \end{split}
\end{align}
Then there exists a unique probability measure in $\mathcal{P}(\mathbb{R})$ with cdf $F$, and such that $F^{\leftarrow}$ is its gqf. This classical result seems to be widely known in the literature. For completeness, a proof sketch is provided in the appendices.
\end{rmk}
\noindent Since generalized inverses of functions in $\mathcal{F}$ always exist, perturbing marginal quantiles in (\ref{eq:pert_prob}) do not require additional assumptions either on the initial probability measure $P$ or on the shape of the target probability measure $Q$. Hence, it allows for generic, well-defined perturbations, which is key for merging both SA and ML interpretability.

Second, \emph{constraints placed on (generalized) quantiles can be interpretable}. Considering $b_i\in\mathbb{R}$ as a representative magnitude of a real component $X_i$ of $X$, then $\mathbb{P}(X_i<b_i)=\alpha_i\in(0,1)$ if and only if $b_i$ minimizes the expected pinball cost function \citep{NEURIPS2021_5b168fdb}:
\begin{eqnarray*}
\ell_i(b) & = & \EE_{X_i}\left[\vert X_i-b \vert\left\{\1_{\{X_i< b\}} + \frac{1-\alpha_i}{\alpha_i} \1_{\{X_i\geq b\}}\right\} \right].
\end{eqnarray*}
This means that $\alpha_i$ can be interpreted as the ratio $c_1/(c_1+c_2)$ where $c_1\vert x_i-b\vert\1_{\{X_i< b\}}$ and $c_2\vert x_i-b\vert\1_{\{X_i\geq b\}}$ are relative first-order costs associated to estimating $x_i$ by $b$. When $X_i$ is an influential input, it seems relevant to produce rules to modify these costs (then the value of $\alpha_i$), shifting the distribution of $X_i$. In particular, this information-theoretic interpretation of quantiles promotes their variation in sensitivity studies conducted in economic or financial contexts (see, e.g., \cite{pesenti_bettini_millossovich_tsanakas_2021}). Moreover, Bayesian statisticians have long recognized this representation as one of the most appropriate formal approaches to incorporating expert knowledge into statistical model inference \citep{Gelfand1995,Mikkola2021}.

Third, in many applied problems, \emph{(generalized) quantile specifications are often key to studying the influence of input variables on a decision-making output}. In both the UQ and ML framework, inputs $X$ can themselves be partially derived from calculations from upstream learning or numerical models (e.g., for multi-physics problems). $P$ is then often calibrated by quantile matching, which may introduce uncertainties on some of its marginal features \citep{pmlr-v97-song19a}. Numerous applications dealing with economic stress tests or risk mitigation against natural hazards use quantiles as influential inputs of decision-helping models. For instance, in the drought risk studies in \cite{Ecoto2021}, the association between soil wetness, climatic, seismic, and socioeconomic variables (e.g., city-level description) is often carried out using marginal quantiles that play the role of features for cost predictive models. Input variations of daily value-at-risk percentiles, computed from legacy data, were recently required by the European Banking Authority for generating macroeconomic scenarios used for EU-wide stress tests \citep{EBA2020}. Reverse SA studies for financial risk management, such as those conducted in \citep{Pesenti2021ReverseSA}, are primarily based on moving values-at-risk, which are quantiles.  

Let us end this subsection with two motivating examples. They offer two additional concrete illustrations of using quantiles for influence analysis. They also illustrate two different quantile perturbation schemes: quantile shifting and application domain dilatation. These schemes are formally introduced in Section~\ref{sec:qPert_shiftDila}. 

\begin{example}[Economic stress test]
Inspired by \citep{Bloom2009}, assume that an economic shock happens in an abstract country. Prospective analyses announce a $\$200$ drop in the population median wage. Before the shock, the population wage distribution $P$ is known (or observed), thanks to some annual census data. This distribution has a median wage of $\$2000$. The new population wage distribution is unknown due to the lack of recent data. The economists would like to know if they can be confident in their predictive macro-economic model $f$ w.r.t. this sudden change.
A way to answer this problem would be assessing the behavior of the model $f$ on a distribution $Q$, such that:
$$F^{\leftarrow}_Q(0.5)= 1800.$$
\end{example}

\begin{example}[River water level] This example is inspired from \cite{Iooss2011} and more deeply studied in Section~\ref{sec:waterLevel}. The safety of an industrial site located near a river is studied through the prediction of the water level $Y=f(X)$ where $f$ is a numerical hydrodynamic model, and $X$ gathers the physical features of the river. A key dimension of $X$ is the Strickler roughness coefficient for the upstream water level \cite{Fu2017}, which is modeled as a truncated Gaussian distribution on $\Omega_X=[20,50]$. However, this application domain is tainted with epistemic uncertainties on the actual nature of the riverbed (e.g.,  more or less subject to shrubby vegetation). The practical use of $f$ would require assessing its predictive power under a wider interval $\Omega_X=[5,65]$. A way to express this prospective study is to assess the model's behavior on a distribution $Q$, such that:
$$F^{\rightarrow}_Q(0)=5, \quad F^{\leftarrow}_Q(1) = 65.$$
\label{ex:riverWaterLevel}
\end{example}

%%%%%%%%%%%%%%%%%%%%%%%%%%%%%%%%%%%%%%%%%%%%%%%%%%%%%%%%%%%%%%%%%%
% 3.2 Formalization
\subsection{A formal definition of quantile perturbation classes}\label{sec:qPert_def}

Focusing on a univariate input $X \sim P \in \mathcal{P}(\mathbb{R})$, let us first recall the formal definition of a quantile. For  $P \in \mathcal{P}(\mathbb{R})$ and $X\sim P$, for $\alpha \in [0,1]$, an $\alpha$-quantile of $P$ is a number $p_{\alpha} \in \mathbb{R}$ such that:
\begin{eqnarray*}
{P}\left( \left\{ X < p_{\alpha}\right\} \right) \leq \alpha & \text{and} &  {P}\left( \left\{ X \leq p_{\alpha}\right\} \right)\geq \alpha.
\label{alpha.quantile}
\end{eqnarray*}
In certain cases, an $\alpha$-quantile is not unique. For instance, assuming that $P$ is purely atomic (e.g., an empirical measure) and that its cdf $F_P$ takes the constant value $\alpha$ on an open interval $(t_0,t_1)$ (i.e., it is the case if $t_0$ and $t_1$ are both atoms of an empirical probability measure), then any $t \in (t_0,t_1)$ is an $\alpha$-quantile. By convention, the gqf of $P$ defined by (\ref{eq:quantile_func}) is the smallest of the $\alpha$-quantiles of $P$ (in this case, $F^{\leftarrow}_P(\alpha)=t_0$). 

As a result, given a chosen $b\in\mathcal{X}$, defining a perturbed version $Q$  of $P$ through the equality constraints $F^{\leftarrow}_Q(\alpha)=b$ seems somewhat arbitrary. It would implicitly imply constraining the smallest $\alpha$-quantile value of $Q$. Arguably, the value of $b$ should be constrained to be a part of the set of all possible $\alpha$-quantiles:
\begin{align*}
    b \in \{q_{\alpha} \in \mathcal{X} \mid & Q( \left\{X < q_{\alpha} \right\}) \leq \alpha, \  Q(\left\{X\leq q_{\alpha} \right\})\geq \alpha \},
\end{align*}
or written equivalently
\begin{equation}
    F^{\leftarrow}_Q(\alpha) \geq b \geq F^{\leftarrow}_Q(\alpha^+) = F^{\rightarrow}_Q(\alpha).
    \label{eq:qConstr_true}
\end{equation}
In the case where $F_Q$ is invertible, it becomes a traditional equality constraint: any $\alpha$-quantile is uniquely defined (i.e., $F^{\leftarrow}_Q(\alpha) =F^{\rightarrow}_Q(\alpha)$). Constraints of the form (\ref{eq:qConstr_true}), which are referred to as \emph{quantile constraints}. They are the basis to define quantile perturbation classes.

\begin{dfi}[Quantile perturbation class]
Let $K \in \mathbb{N}^*$ be the cardinality of a collection of quantile constraints defined by $\alpha = (\alpha_1, \dots, \alpha_K)^\top \in [0,1]^K$ and $b = (b_1, \dots, b_K)^\top \in \mathbb{R}^K$. The quantile perturbation class $\mathcal{Q} \subset \mathcal{P}(\mathbb{R})$ is the set of probability measures defined as:
$$\mathcal{Q} = \left\{Q \in \mathcal{P}(\mathbb{R}) \mid F^{\leftarrow}_Q(\alpha_i) \leq b_i \leq F^{\rightarrow}_Q(\alpha_i), \quad i=1,\dots,K \right\}.$$
\label{def:qClass}
\end{dfi}
Under simple conditions on the values $\alpha$ and $b$, quantile perturbation classes are non-empty.
\begin{lme}
Let $\mathcal{Q}$ be a perturbation class defined over $\alpha \in [0,1]^K$ and $b \in \mathbb{R}^K$, which are assumed to be ordered, without loss of generality. If
\begin{equation}
    0 \leq \alpha_1 < \dots < \alpha_K \leq 1 , \quad \text{ and } \quad b_1 < \dots < b_K,
    \label{eq:condNonempty}
\end{equation}
then $\mathcal{Q}$ is non-empty.
\label{propo:nonemptyQ}
\end{lme}

Since ${F}^{\leftarrow}_Q$ can oftentimes be discontinuous, smoothing restrictions can be envisioned. It entails restricting $\mathcal{Q}$ to probability measures whose quantile functions are \emph{smooth} (e.g., continuous, derivable). Smooth quantile perturbation classes can be introduced as follows.
\begin{dfi}[Smooth quantile perturbation class]
Let $K \in \mathbb{N}^*$ and let $\alpha = (\alpha_1, \dots, \alpha_K)^\top \in [0,1]^K$, $b = (b_1, \dots, b_K)^\top \in \mathbb{R}^K$. Additionally, let $\mathcal{V} \subseteq \mathcal{F}^{\leftarrow}$ be a given set of smooth non-decreasing functions. The smooth quantile perturbation class $\mathcal{Q}_{\mathcal{V}} \subset \mathcal{P}(\mathbb{R})$ is the set of probability measures defined as:
$$\mathcal{Q}_{\mathcal{V}} = \left\{Q \in \mathcal{P}(\mathbb{R}) \mid F^{\leftarrow}_Q \in \mathcal{V}, \quad F^{\leftarrow}_Q(\alpha_i) \leq b_i \leq F^{\rightarrow}_Q(\alpha_i), \quad i=1,\dots,K \right\}.$$
\label{def:smoothQClass}
\end{dfi}
These classes are further discussed and illustrated in Section~\ref{sec:qcwProj_smooth}. Note that smooth perturbation classes generalize  perturbation classes since $\mathcal{Q} = \mathcal{Q}_{\mathcal{F}^{\leftarrow}}$. Therefore, without loss of generality, perturbation classes are denoted $\mathcal{Q}_{\mathcal{V}}$ in the following. In the next few paragraphs, echoing the illustrative examples in Section~\ref{sec:motivations}, we formalize two types of quantile constraints: quantile shifts and application domain dilatations.

%%%%%%%%%%%%%%%%%%%%%%%%%%%%%%%%%%%%%%%%%%%%%%%%%
% 3.3 Types of quantile perturbations
\subsection{Two key quantile perturbations}\label{sec:qPert_shiftDila}
\subsubsection{Quantile shift}\label{sec:qPert_shift}
The \emph{quantile shift perturbation} aims at increasing as well as decreasing some of the quantiles of the initial distribution. Formally, given a quantile level $\alpha \in [0,1]$, and an initial $\alpha$-quantile $p_{\alpha} = F^{\leftarrow}_P(\alpha)$, the quantile shift entails defining values for $b$ in (\ref{eq:qConstr_true}) ranging over a compact interval $[\eta_0, \eta_1] \subseteq \Omega_X$ such that $\eta_0 < p_{\alpha} < \eta_1$. The next lemma formalizes  an expression for $b$ as a function of a perturbation intensity $\theta$ standardized on $\Theta=[-1,1]$. 
\begin{lme}
    Let $\Theta=[-1,1]$ and denote $\bm\eta=(\eta_{0},\eta_{1})$ with $\eta_0 < p_{\alpha} < \eta_1$. For $\theta\in\Theta$, let,
    $$
    b_{\alpha} (\bm\eta, \theta)=\begin{cases}
    p_{\alpha} (1+\theta) - \theta \eta_0 & \text{if } -1\leq \theta <0, \\
    p_{\alpha} & \text{if } \theta=0, \\
    p_{\alpha} (1-\theta) + \theta \eta_1 & \text{if } 0<\theta\leq 1.
    \end{cases}$$
    Then, for $Q_{\theta}\in\mathcal{P}(\mathbb{R})$ such that
    $$F^{\leftarrow}_{Q_{\theta}}(\alpha) \geq b_{\alpha}(\bm\eta, \theta) \geq  F^{\rightarrow}_{Q_{\theta}}(\alpha),$$
    one has that, $\forall \theta\in\Theta$:
    \begin{eqnarray*}
        -1 \leq \theta < 0 & \Leftrightarrow & \eta_0 \leq F^{\leftarrow}_{Q_{\theta}}(\alpha) < p_{\alpha}, \\
        \theta=0 & \Leftrightarrow &  F^{\leftarrow}_{Q_{\theta}}(\alpha) = p_{\alpha}, \\
        0<\theta \leq 1 & \Leftrightarrow &  p_{\alpha} < F^{\leftarrow}_{Q_{\theta}}(\alpha) \leq \eta_1.
    \end{eqnarray*}
    \label{lemma:quantile.shift.lemma}
\end{lme}

We refer to Figure~\ref{fig:sch_pert} (a.) for an illustration of this perturbation scheme, and to Section~\ref{sec:AFE_pert} for a real-world application. The quantile shift perturbation class can, for a given initial quantile level $\alpha \in [0,1]$, and valid interval bounds $\bm\eta =(\eta_0, \eta_1)$, $\eta_0<p_{\alpha}<\eta_1$, be formally defined as the collection of perturbation classes
\begin{equation}
    \mathcal{T}(\bm\eta, \theta) = \{Q \in \mathcal{P}(\mathbb{R}) \mid F^{\leftarrow}_Q(\alpha) \leq b_{\alpha}(\bm\eta, \theta) \leq F^{\rightarrow}_Q(\alpha)\}
    \label{eq:quantShift_pert}
\end{equation}
indexed by the intensity $\theta \in [-1,1]$. Each choice of $\theta$ induces a perturbation class for which the projection problem in (\ref{eq:pert_prob}) must be solved.

\begin{figure*}[t!]
    \centering
    \includegraphics[width=\linewidth, trim=0.25cm 0cm 0.25cm 0cm]{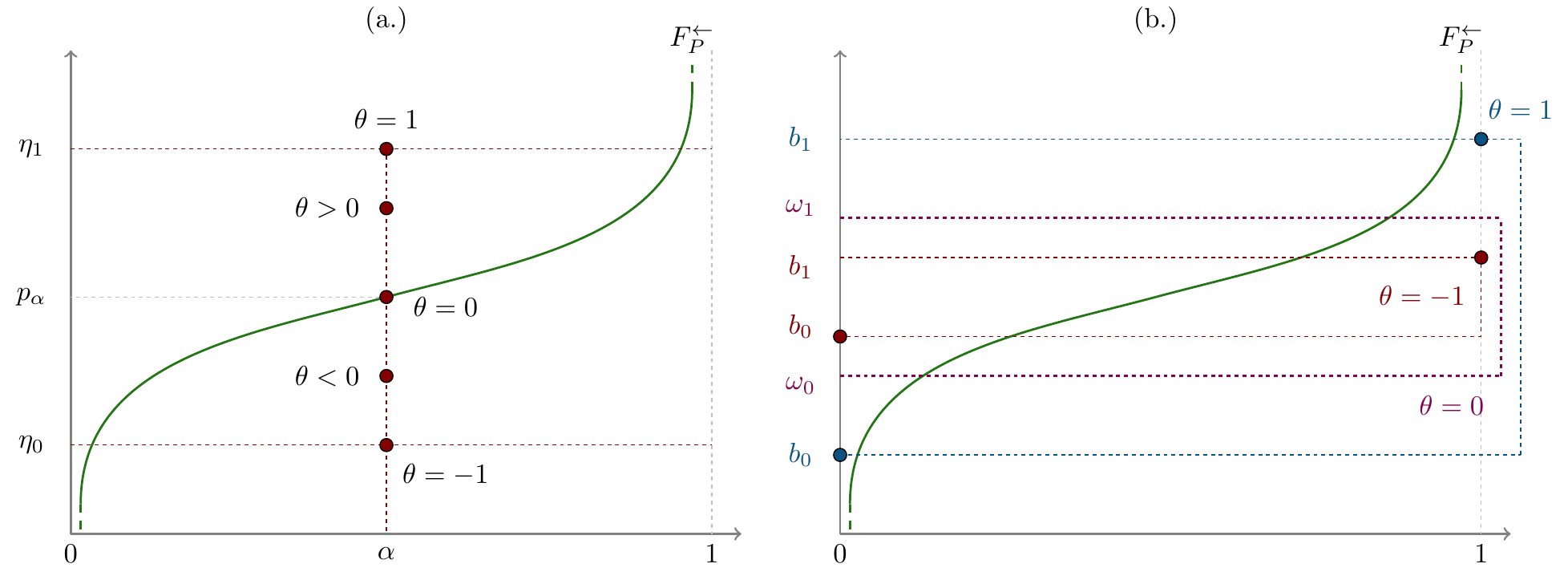}
    \caption{Quantile shift (a.) and application domain dilatation (b.) perturbation schemes. The initial quantile function is displayed in green. On the left, red points indicate different quantile shifting constraints between $\eta_0$ and $\eta_1$, leading to different intensity values $\theta$. On the right, the application domain's width (in magenta) is up to doubled (blue points) or down to halved (red points), according to an intensity parameter $\theta$ evolving in a range $\Theta=[-1, 1]$.}
    \label{fig:sch_pert}
\end{figure*}

%%%%%%%%%%%%%%%%%%%%%%%%%%%%%%%%%%%%%%%%%%%%%%%%%
% Operating domain dilatation
\subsubsection{Application domain dilatation}\label{sec:qPert_dila}
Application domain dilatation consists in perturbing the bounds of the application domain of an input. For a univariate $X\sim P$ with $\Omega_X=[\omega_0, \omega_1]$, the dilatation process amounts to widening or narrowing the width (or diameter $\mbox{diam}(\Omega_X)$) of $\Omega_X$. It is similar to perturbing extreme quantile levels ($\alpha \in \{0,1\}$) of $P$ while preserving the midpoint of $\Omega_X$.  The dilatation is characterized by a parameter $\eta > 1$ controlling the rescaling magnitude of $\Omega_X$ while preserving its midpoint. In other words, one aims at finding a distribution $Q$ with support $\text{Supp}(Q) = [b_0, b_1]$ for $b_0, b_1 \in \mathcal{X}$, $b_0<b_1$, where the midpoint of $[b_0, b_1]$ is equal to the midpoint of $\Omega_X$, but such that $\mbox{diam}(Q) := \mbox{diam}(\mbox{Supp}(Q))$ is rescaled compared to $\mbox{diam}(\Omega_X)$. Similarly to quantile shift, the next lemma formalizes expressions for these two bounds as a function of a perturbation intensity $\theta$ standardized on $\Theta=[-1,1]$.   

\begin{lme}
Let $\Theta=[-1,1]$ and $\eta>1$.  For $\theta\in\Theta$, let $Q_{\theta}\in\mathcal{P}(\mathbb{R})$ such that
\begin{eqnarray*}
F^{\leftarrow}_{Q_{\theta}}(0)\ = \ b_{0} (\eta, \theta)= \begin{cases}
    \frac{1}{2}\left(\omega_0 \bigl(2-\theta(\eta^{-1} -1)\bigr) + \theta\omega_1(\eta^{-1}-1) \right) & \text{if } -1 \leq \theta < 0, \\
    \omega_0 & \text{if } \theta = 0, \\
    \frac{1}{2}\left( \omega_0\bigl(2+\theta(\eta -1)\bigr) - \theta\omega_1(\eta-1)\right) & \text{if } 0<\theta \leq 1,
    \end{cases} \\
F^{\leftarrow}_{Q_{\theta}}(1)\ = \ b_{1} (\eta, \theta)=  \begin{cases}
    \frac{1}{2}\left( \omega_1\bigl(2-\theta(\eta^{-1}-1)\bigr) + \theta\omega_0(\eta^{-1}-1)\right) & \text{if } -1 \leq \theta < 0, \\
    \omega_1 & \text{if } \theta = 0, \\
    \frac{1}{2}\left( \omega_1\bigl(2+ \theta(\eta -1)\bigr) - \theta\omega_0(\eta-1)\right) & \text{if } 0<\theta \leq 1.
    \end{cases}  
\end{eqnarray*}
Then, $\forall (\theta,\eta)\in\Theta\times[1,\infty)$,  
\begin{eqnarray*}
b_{0} (\eta, \theta)+b_{1} (\eta, \theta) & = & \omega_0 + \omega_1 \ \ \ \text{\it (midpoints equality)}
\end{eqnarray*}
and
\begin{eqnarray*}
-1\leq \theta<0 & \Leftrightarrow &  \frac{1}{\eta}\mbox{diam}(\Omega_X) \leq \mbox{diam}(Q_{\theta})<\mbox{diam}(\Omega_X), \\
\theta =0 & \Leftrightarrow & \mbox{diam}(Q_{\theta}) = \mbox{diam}(\Omega_X), \\
0< \theta \leq 1 &  \Leftrightarrow & \mbox{diam}(\Omega_X) < \mbox{diam}(Q_{\theta}) < \eta \mbox{diam}(\Omega_X).
\end{eqnarray*}
\label{lemma:dilatation.lemma}
\end{lme}
We refer to Figure~\ref{fig:sch_pert} (b.) for an illustration of this perturbation scheme. The initial application domain is displayed in magenta and is subject to a dilatation of parameter $\eta=2$. The red constraints halve its width, and the blue constraints double it. One can additionally check that in both cases, the midpoint of the original validity domain is preserved. Section~\ref{sec:waterLevel_pert} showcases the usage of application domain dilatation perturbation in practice.

Given a perturbation class $\mathcal{Q}_{\mathcal{V}}$ defined over some modeling constraints, a collection $\mathcal{C}_{\mathcal{V}}(\theta)$ of perturbation classes driven by $\theta$ can be easily defined:
\begin{eqnarray*}
\mathcal{C}_{\mathcal{V}}(\theta) & = & \mathcal{Q}_{\mathcal{V}} \cap \mathcal{T}(\eta, \theta)
\end{eqnarray*}
where
$\mathcal{T}(\eta, \theta) = \left\{ Q \in \mathcal{P}(\mathbb{R}) \mid F^{\leftarrow}_Q(0) = b_0(\eta, \theta), F^{\leftarrow}_Q(1) = b_1(\eta, \theta) \right\}$ in the case of application domain dilatation perturbations, or $\mathcal{T}(\bm\eta, \theta)$ as defined in (\ref{eq:quantShift_pert}) in the case of a quantile shift.

Many perturbation settings can be defined by combining quantile shifts and domain dilatations. However, for the sake of simplicity, quantile shifts and domain dilatations are studied independently in Section~\ref{sec:robust_meas}.

%%%%%%%%%%%%%%%%%%%%%%%%%%%%%%%%%%%%%%%%%%%%%%%%%%%%%%%%%%%%%%%
% 3.4 Multiple perturbations
\subsection{Perturbing multiple inputs}\label{sec:qPert_copula}

In the previous sections, perturbation classes have been defined \emph{marginally}, i.e., on uni-dimensional probability measures. Whenever multiple inputs (or features) are involved, an independence assumption is often assumed, allowing to perturb marginal distributions independently, as proposed by \cite{Lemaitre2015}. While this hypothesis facilitate the interpretation and use of models, it is questionable in practice. Therefore, one of the main challenges in ML interpretability and SA is to account for the potential dependence structure between the inputs (or features) \citep{PLISCHKE20191046}. Dependencies can often provide useful information which must be preserved through a simultaneous perturbation (e.g., to maintain the plausibility of perturbed information and avoid creating meaningless patterns \citep{Benou2020}). In ML, perturbation problems dealing with data privacy are thus particularly concerned by the preservation of dependencies \citep{Liu2006RandomPM,PAUL2021102954}. In UQ and ML frameworks, marginal quantile perturbations on $P$ must obey this requirement: the dependence structure between the initial and perturbed distributions must remain the same. 

Dependencies between random variables can be modeled using copula-based representations \citep{Nelsen2006}. Let us recall the definition of a copula (or \emph{dependence function}) $C_P:[0,1]^d\to [0,1]$. Let the random $d$-dimensional vector of inputs $X = (X_1, \dots, X_d)^\top \sim P$, where $P \in \mathcal{P}(\mathcal{X})$ and $\mathcal{X} \subseteq \mathbb{R}^d$. Assume that the marginal cdfs $F_{P_i}$, $i=1,\dots,d$ are continuous. Denote $U_{1}, \dots, U_{d}$ the random variables defined as:
$$U_i = F_{P_i}(X_i)$$
and denote $U =(U_1, \dots, U_d)^\top \sim U_P$. For any ${\bf u}=(u_1, \dots, u_d) \in [0,1]^d$, denote $H_{\bf u} = \vartimes_{i=1}^d [0,u_i]$. The copula of $X$, denoted $C_p$ is defined as:
\begin{align*}
    C_P(u) &= Pr\left(U_1\leq u_1, \dots, U_d \leq u_d \right) \\
    &= \int_{H_{\bf u}} dU_P
\end{align*}
We refer to the proof of Remark~\ref{prop:copula-results} for a proper definition of copula for empirical measures. In the following remark, we show that by means of particular monotone transportation maps inspired by optimal transportation theory, the marginal inputs can be optimally perturbed while preserving their copula. Moreover, this result is suitable for both ML and UQ applications since it also holds whenever $P$ is an empirical measure related to an observed set of data points.
\begin{rmk}
Let $\mathcal{X} \subseteq \mathbb{R}^d$, for $d$ a positive integer, and $P \in \mathcal{P}(\mathcal{X})$. Let $Q_i$ be the solution of the optimal projection problem (\ref{eq:pert_prob}) with  $\mathcal{C}=\mathcal{C_V}$, for every marginal distribution $P_i$ of $P$, $i=1,\dots,d$, and where $\mathcal{V}   \subseteq  \otimes_{j=1}^d \mathcal{F}^{\leftarrow}_j$. Let the random vectors
$$X \sim P, \quad \widetilde{X} :=T(X)$$
where 
\begin{eqnarray}
 T : &\mathcal{X}   &\rightarrow \quad \quad  \mathcal{X} \nonumber \\
         &\begin{pmatrix}x_1\\ \vdots\\ x_d \end{pmatrix} &\mapsto  \begin{pmatrix} T_1(x_1)\\\vdots \\ T_d(x_d)\end{pmatrix} \label{eq:pertMap}
\end{eqnarray}
where
$$T_j = \left(F^{\leftarrow}_{Q_j} \circ F_{P_j}\right), \quad j=1,\dots, d. $$
\begin{description}
\item[(i)] If $P$ is an empirical measure (i.e., $X$ represents a dataset), then $X$ and the perturbed dataset $\widetilde{X}$ have the same empirical copula. Moreover, the empirical measure of every perturbed marginal sample $\widetilde{X}_i$ converges towards $Q_i$, $i=1,\dots, d$.
\item[(ii)] If $P$ is atomless, and assuming additionally that $\mathcal{V}$ is such that every $F^{\leftarrow}_{Q_i}$, $i=1,\dots,d$ is strictly increasing, then the random vectors $X$ and $\widetilde{X}$ have the same copula. Moreover, each perturbed marginal $\widetilde{X}_i \sim Q_i$.
\end{description}
\label{prop:copula-results}
\end{rmk}
In other words, applying the perturbation map (\ref{eq:pert_prob}) to the inputs allows for preserving their copula and hence their dependence structure. If only an initial dataset is observed, applying $T$ to every observation results in a perturbed dataset with, for instance, the same Spearman correlation matrix. Moreover, these transportation maps achieve optimality for various univariate transportation costs (see the proof). Moreover, in the UQ framework, sampling from the perturbed inputs is as simple as applying these maps to simulated samples of $P$, which can naturally benefit from the large literature available on copula. However, it requires that the marginal gqfs $\{F^{\leftarrow}_{Q_j}\}_{1\leq j \leq d}$ are accessible, which obviously depend on the projection problem (\ref{eq:pert_prob}). As shown in the next sections, and among many other practical and theoretical motivations, the particular choice of the $2$-Wasserstein distance as a projection metric allows for characterizing the optimally perturbed marginal probability measures through their quantile functions and thus simplifies their accessibility.

%%%%%%%%%%%%%%%%%%%%%%%%%%%%%%%%%%%%%%%%%%%%%%%%%
% SECTION 4 : Wasserstein as a Discrepancy
\section{Wasserstein projections}\label{sec:wassProj}
%%%%%%%%%%%%%%%%%%%%%%%%%%%%%%%%%%%%%%%%%%%%%%%%%%%%%%%%%%%%%%%%%%%%%%%%%%%%%%%%%%%%%%%%%%%%%%%%%%%%%%%%%%%%%%%%%%%%%%%%%%%%%%%%%%%%%%%%%%%%%%%%%%%%%%%%%%%%%%%%%%%%%%%%%%%%%%%%%%%%%%%%%%%%%%%%%%%%
% Intro de section
This section motivates the choice of the $2$-Wasserstein distance as a projection metric for the perturbation problem (\ref{eq:pert_prob}). This distance is deeply rooted in optimal transportation theory \citep{Villani2003} and has been used successfully in many ML and deep learning applications \citep{Frogner2015, Arjovsky2017}. It has also been extensively studied as a tool for guaranteeing distributional robustness to adversarial attacks in ML \cite{Duchi2021}. In SA, it has been used to produce novel sensitivity indices \citep{Fort2021, Borgonovo2022}.

Using the Wasserstein distance imposes to work on the subset $\mathcal{P}_p(\mathbb{R}^d) \subset \mathcal{P}(\mathbb{R}^d)$ of probability measures with finite $p$-th moment. The Wasserstein distance between multi-dimensional probability measures can be computationally expensive to evaluate \cite{Peyre2019}. However, as stated in Section~\ref{sec:qPert_copula}, the fact that one wishes to preserve the copula between $P$ and its optimally perturbed counterpart $Q$ greatly simplifies the projection problem. Let $P, Q \in \mathcal{P}_p(\mathbb{R}^d)$ be two multi-dimensional probability measures, with marginals $P_1, \dots, P_d$ and $Q_1, \dots, Q_d$ respectively. Leveraging the work in \cite{Alfonsi2014}, if $P$ and $Q$ share the same copula, one can rewrite their $2$-Wasserstein distance as:
\begin{equation}
    W_p^p(P,Q) = \sum_{i=1}^d W_p^p(P_i, Q_i).
    \label{eq:w2_sameCop}
\end{equation}
Hence, finding $Q$ that minimizes $W_p^p(P,Q)$ under constraints on the marginals $Q_1, \dots, Q_d$ is equivalent to solving $d$ independent univariate projection problems with the $p$-Wasserstein distance as a projection metric. Furthermore, the transportation map defined in (\ref{eq:pertMap}) is indeed optimal \cite{Alfonsi2014} (but not unique), in the sense that it minimizes (\ref{eq:w2_sameCop}).

In other words, finding $Q$ that minimizes (\ref{eq:w2_sameCop}) such that $C_P = C_Q$ is equivalent to projecting each $P_i$ under its relevant constraints, and applying the transportation map (\ref{eq:pertMap}). As the problem reduces to perturbations of univariate marginal distributions, the $p-$Wasserstein distance may be easily computed, as recalled in the following definition.

\begin{dfi}[Wasserstein distance on the real line]
Let $p \in \mathbb{N}^*$ and $P,Q \in \mathcal{P}_p(\mathbb{R})$ be two probability measures on $\mathbb{R}$ admitting $F_P$ and $F_Q$ as probability distribution functions, respectively. Then, the $p$-Wasserstein distance between $P$ and $Q$ is:
\begin{equation*}
    W_p(P,Q) = \left(\int_0^1\left\lvert F^{\rightarrow}_P(x) - F^{\rightarrow}_Q(x) \right\rvert^p dx\right)^{1/p}
\end{equation*}
where $F^{\rightarrow}_P$ (resp. $F^{\rightarrow}_Q$) has been defined in Definition~\ref{def:genQuant}.
\label{def:wass_1D}
\end{dfi}

The following subsections argue on the specific choice of the $2$-Wasserstein distance. First, we highlight its attractive properties for conducting robust interpretability analyses. Then we investigate the solution of the perturbation problem (\ref{eq:pert_prob}), with and without regularity constraints, the latter enforced using isotonic, piece-wise polynomial approximations.

%%%%%%%%%%%%%%%%%%%%%%%%%%%%%%%%%%%%%%%%%%%%%%%%%%%%%%%%%%%%
% 4.1 Why 2-Wasserstein is cool
\subsection{The 2-Wasserstein  distance as a suitable perturbation discrepancy}\label{sec:w2_suitable}

The special choice of the $2$-Wasserstein  distance 
$$ W_2(P,Q) = \sqrt{\int_0^1\left( F^{\rightarrow}_P(x) - F^{\rightarrow}_Q(x) \right)^2 dx}, \quad P,Q\in \mathcal{P}_2(\mathbb{R})$$
to instantiate the perturbation problem (\ref{eq:pert_prob}) is based on the following rationale.

First, $W_2$ metricizes weak convergence on $\mathcal{P}_2(\mathbb{R})$ (\citep{Villani2003}, Section~7.2). It means that $W_2$ is a measure of proximity on a broad set of probability measures. In other words,  for any $P \in \mathcal{P}_2(\mathbb{R})$, and a sequence of probability measures $(Q_n)_{n\in \mathbb{N}^*}\in \mathcal{P}_2(\mathbb{R})$:
$$W_2(P,Q_n) \underset{n \rightarrow \infty}{\rightarrow} 0 \quad \Rightarrow \quad Q_n \overset{\text{d}}{\longrightarrow} P$$
where $\overset{\text{d}}{\rightarrow}$ denotes the convergence in distribution (or weak convergence). That is, $W_2$ allows for assessing the point-wise proximity between two probability measures, as long as both admit finite second-order moments (a current assumption in both SA and ML fields). Contrary to the KL divergence, no additional conditions on the absolute continuity of $Q_n$ w.r.t. $P$ are needed. When it comes to the perturbation problem (\ref{eq:pert_prob}), two practical advantages in favor of the $2$-Wasserstein distance can be drawn, compared to entropic projections (i.e., using the KL divergence): if $P$ is an empirical measure (i.e., purely atomic), then $Q$ is not restricted to be purely atomic; conversely, if $P$ admits a density, then it does not restrict $Q$ to admit a density. 

These benefits are key in unifying the frameworks of SA and ML interpretability: the flexibility of $W_2$ allows for greater explicit control (e.g., through smoothing restriction) on the nature of $Q$, independently of that of $P$. Results of entropic projections entail a re-weighting of the atoms when $P$ is empirical \citep{Bachoc2020}, or a perturbed density of $Q$ proportional to the one of $P$ when it is absolutely continuous w.r.t. the Lebesgue measure \citep{Lemaitre2015}. Using $W_2$ allows, for instance, to add additional atoms, or to allow $Q$ to admit a density, independently of the regularity of $P$.

Second, $W_2$ facilitates the projection of $P$ onto a quantile perturbation class  $C_{\mathcal{V}}(\theta)$. Indeed, the next proposition shows that the optimization problem is equivalent to a projection in $L^2$ of its gqf onto $\mathcal{V}$ under interpolation constraints 

\begin{propo}
Let $P\in \mathcal{P}_2(\mathbb{R})$, and $\mathcal{C}_{\mathcal{V}}(\theta)$ be a non-empty perturbation class defined by a subset $\mathcal{V}\subseteq\mathcal{F}^{\leftarrow}$. Consider for $\mathcal{V}$ the constraints system defined in $\S$ \ref{sec:qPert_def}-\ref{sec:qPert_shiftDila} associated with couples $(\alpha_i,b_i)_{1\leq i \leq K}$. In this frame, the solution $Q$ of the perturbation problem  (\ref{eq:pert_prob}), i.e.,
\begin{align}
    \begin{split}
        Q = \underset{G \in \mathcal{P}_2(\mathbb{R})}{\text{argmin }}& W_2(P,G) \ \ \ 
        \text{s.t. } \ \  G \in \mathcal{C}_{\mathcal{V}}(\theta)
        \label{eq:init_pb}
    \end{split}
\end{align}
is characterized as the unique Lebesgue-Stieljes measure induced by the cdf $F_Q$ with gqf $F_Q^{\leftarrow} \in \mathcal{F}^{\leftarrow}$:
\begin{align}
    \begin{split}
    F_Q^{\leftarrow} = \underset{L \in L^2([0,1])}{\text{argmin }}& \left\{ \int_0^1 \left(L(x) - F^{\rightarrow}_P(x) \right)^2\right\} \\
    \text{s.t. }& \quad
     L(\alpha_i) \leq b_i \leq L\left(\alpha_i^+\right), \quad i=1, \dots, K, \\
    & \quad 
     L \in \mathcal{V}.
    \label{eq:equiv_pb}
    \end{split}
\end{align}
\label{propo:wass_l2}
\end{propo}

The equivalent projection problem in (\ref{eq:equiv_pb}) echoes with the result in Proposition~\ref{prop:uniqueQuantFunc}. Projecting a measure w.r.t. the $2$-Wasserstein distance is equivalent to projecting its gqf in $L^2$. One can then leverage the vast literature on function approximations in $L^2$, especially on monotonic approximations. Moreover, as eluded at the end of Section \ref{sec:qcPert}, the proposed perturbation scheme depends heavily on the knowledge of the gqf of $Q$. Solving (\ref{eq:equiv_pb}) grants direct access to the gqf of $Q$, and thus allows applying perturbations fairly easily.

%%%%%%%%%%%%%%%%%%%%%%%%%%%%%%%%%%%%%%%%%%%%%%%%%%%%%%%%%%%%%%%%%%%%%%%%%%%%%%%%%%%%%%%%%%%%%%%%%%%%%%%%%%%%%%%%%%%%%%%%%%%%%%%%%%%%%%%%%%%%%%%%%%%%%%%%%%%%%%%%%%%%%%%%%%%%%%%%%%%%%%%%%%%%%%%%%%%%
% Solving without smoothing
\subsection{Quantile constrained Wasserstein projections}\label{sec:qcW_proj}

This subsection presents and discusses the main results of this paper. If no smoothing constraints on $F_Q^{\leftarrow}\in \mathcal{F}^{\leftarrow}$ is enforced, the perturbation problem (\ref{eq:equiv_pb}) leads to a unique analytical solution. However, many studies require $F_Q^{\leftarrow}$ to be smooth (e.g., continuous). We propose enforce continuity by using isotonic interpolating piece-wise polynomials, which lead to a well-defined convex constrained quadratic program, easily solvable in practice.

%%%%%%%%%%%%%%%%%%%%%%%%%%%%%%%%%%%%%%%%%%%%%%%%%
% W2 is interpretable
\subsubsection{Analytical solution without smoothing restrictions}\label{sec:qcW_proj_analSol}

The following proposition provides a convenient way to solve the perturbation problem  (\ref{eq:equiv_pb}) in the case when no smoothing constraint on $F_Q^{\leftarrow}$ is enforced.

\begin{propo}
Let $P$ be a probability measure in $\mathcal{P}_2(\mathbb{R})$. Let $\mathcal{C}$ be a non-empty perturbation class characterized by a set of $K$ quantile constraints. Assume, without loss of generality, for $i=1,\dots, K$, that 
$\alpha_1 < \dots < \alpha_K$ along with  $b_1 < \dots < b_K$. Let $\beta_i = F_P(b_i)$ for $i=1,\dots, K$. Define the intervals $A_i = (c_i, d_i]$ for $i=1,\dots,K$, such that:
\begin{gather*}
    c_1 = \min(\beta_1, \alpha_1), \quad c_i = \min\Bigl[ \max(\alpha_{i-1},\beta_i), \alpha_i \Bigr], i=2, \dots, K, \\
    d_K=\max(\beta_K, \alpha_K),\quad  d_j = \max \Bigl[ \min(\beta_j,\alpha_{j+1}), \alpha_j \Bigr], j=1, \dots, K-1.
\end{gather*}
Let $A = \bigcup_{i=1}^K A_i$ and $\overline{A} = [0,1] \setminus A$. Then the problem (\ref{eq:equiv_pb}) has a unique solution which can be written as, for any $y \in [0,1]$:
\begin{align}
    F^{\leftarrow}_Q(y) = \begin{cases}
    F^{\rightarrow}_P(y) & \text{if } y \in \overline{A}, \\
    b_i & \text{if } y \in A_i, \quad i=1,\dots, K.
    \end{cases}
    \label{eq:solPert_vanilla}
\end{align}
\label{prop:analytRes}
\end{propo}

In order to interpret this result, illustrated in Figure~\ref{fig:sch_analyticalSol}, let us recall that when a quantile function is constant on an interval, it implies that its related probability measure admits an atom at the constant value taken by the gqf. Moreover, the mass allocated to this atom is equal to the length of the interval. Additionally, each jump of the quantile function induces an interval with no mass. The solution displayed in (\ref{eq:solPert_vanilla}) shows that on $\overline{A}$, both initial and perturbed quantile functions are equal. However, they differ on every interval $A_i$ in the following fashion: 
\begin{itemize}
    \item $Q$ have atoms at each constraint point $b_i$, $i=1,\dots, K$;
    \item Each of these atoms have mass $Q(\{b_i\}) = d_i - c_i$, for $i=1,\dots, K$;
    \item Each open interval $I_i \subset \mathbb{R}$ defined as
    \begin{equation}
    I_i = \begin{cases}
            \Bigl( \max(F^{\leftarrow}_P(\alpha_i), b_{i-1}), b_i \Bigr), & \text{when } b_i > F^{\leftarrow}_P(\alpha_i),\\
            \Bigl( b_i, \min \left(b_{i+1}, F^{\leftarrow}_P(\alpha_i)\right) \Bigr), & \text{when } b_i < F^{\leftarrow}_P(\alpha_i)
        \end{cases}
        \label{eq:interv_nomass}
    \end{equation}
    with, by convention, $b_0=-\infty$ and $b_{K+1}= \infty$, has no mass. To put it briefly, $Q(I_i) = 0$ for every $i=1,\dots, K$.
\end{itemize}

\begin{figure*}[b!]
    \centering
    \includegraphics[width=\linewidth, trim= 0.25cm 0cm 0cm 0cm]{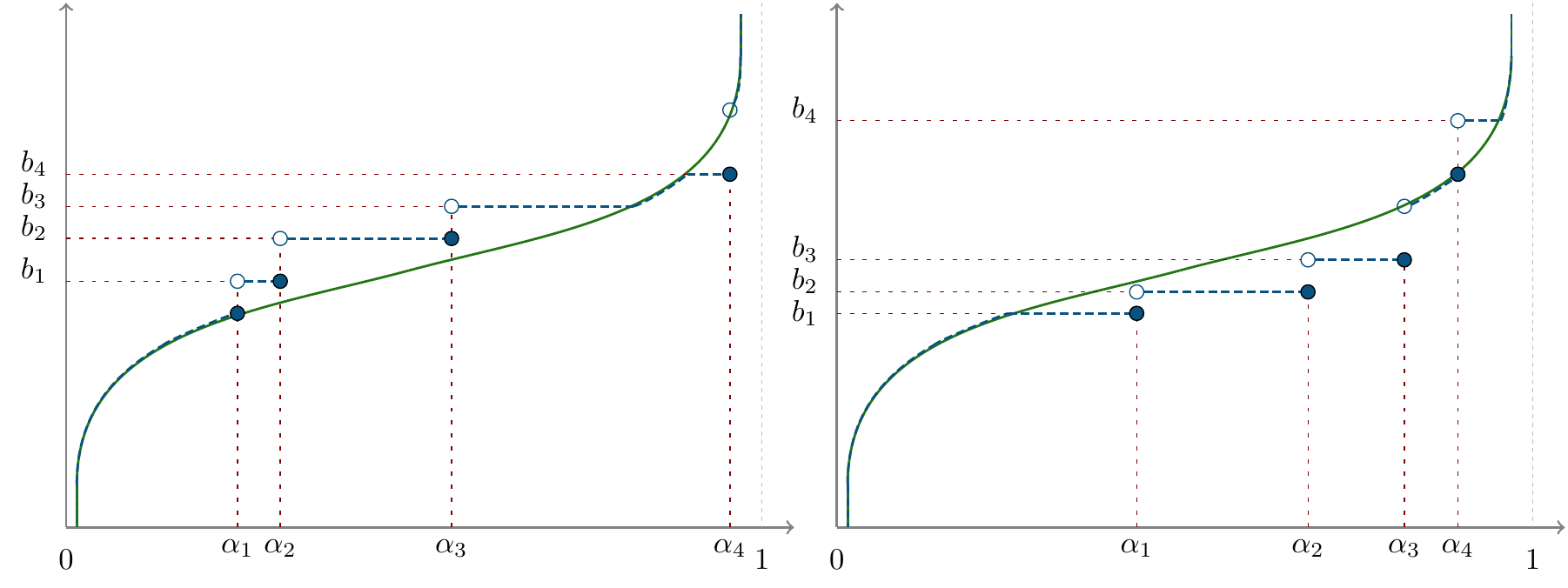}
    \caption{Characterizing quantile function of the solution of the perturbation problem (dashed blue). The initial quantile function (i.e., $F_P^{\leftarrow}$) is displayed in green, and dashed red lines identify the quantile constraints. (a.) and (b.) illustrate different possible perturbation configurations, increasing or decreasing several initial quantile values.}
    \label{fig:sch_analyticalSol}
\end{figure*}

In other words, whenever an $\alpha$-quantile $p_{\alpha}$ is shifted up to a value $b$, the perturbation entails sending every possible values in the range $(p_{\alpha}, b)$ to $b$. Hence, every value in $(p_{\alpha}, b)$ cannot be sampled according to $Q$. Moreover, the singleton $\{ b \}$ now admits a probability of being observed equal to the initial probability of this interval, i.e., $Q(\{b\}) = P\bigl( (p_{\alpha}, b) \bigr)$. When an $\alpha$-quantile is shifted down to $b$, the interval becomes $(b, p_{\alpha})$, and the same reasoning can be done.

The statement of Proposition \ref{prop:analytRes} is rather intuitive. Indeed, the  Wasserstein distance quantifies the amount of \emph{work} needed to transform a probability measure into another one \citep{Santambrogio2015}. When using $W_2$, the amount of work is quantified using the Euclidian distance, i.e., transporting a point $x_0$ to $x_1$ requires $(x_0 - x_1)^2$ units of work. This intrinsic \emph{point-wise} way of quantifying similarities can be sensed in the previous result: perturbing an $\alpha$-quantile entails giving the initial mass of an interval adjacent to $b$ to the singleton $\{ b \}$ in order to satisfy the constraint.

%%%%%%%%%%%%%%%%%%%%%%%%%%%%%%%%%%%%%%%%%%%%%%%%%
% Why smoothness
\subsubsection{Projection solution under smoothness restrictions}\label{sec:qcwProj_smooth}

The analytical solution provided in Proposition~\ref{prop:analytRes} presents a significant drawback: part of the application domain $\Omega_X$ of the perturbed input receives no mass. From a practical standpoint, it cannot be explored in a robustness study. It is because $\mathcal{F}^{\leftarrow}$ contains discontinuous functions. Ensuring a smoother solution relies on specifying a smooth perturbation class $\mathcal{Q}_{\mathcal{V}}$ where $\mathcal{V}$ is a set of continuous, non-decreasing functions. To solve the perturbation problem, we can take advantage of its $L^2([0,1])$ equivalent formulation \ref{eq:equiv_pb}) to project $F^{\leftarrow}_P$ onto $\mathcal{V}$. However, the main challenge is to ensure that $\mathcal{V}$ only contains monotonic functions.

We propose to projection $F^{\leftarrow}_P$ onto a space of piece-wise continuous polynomials. It implies that the support of $Q$ must be bounded. These bounds are made explicit using extremal quantile constraints (i.e., $F^{\leftarrow}_Q(0)$ and $F^{\leftarrow}_Q(1)$ are constrained to take finite values).

Formally, we aim to find a piece-wise polynomial of the form
\begin{equation}
    G(x) = \begin{cases} 
G_0(x) & \text{if } \alpha_0 := 0 \leq x < \alpha_1, \\
\vdots\\
G_i(x) & \text{if } \alpha_i \leq x < \alpha_{i+1},\\
\vdots\\
G_K(x) & \text{if } \alpha_K \leq x \leq 1 =: \alpha_{K+1}.
\end{cases}
\label{eq:def_pws_poly}
\end{equation}
under the continuity constraints at each knot of the grid $\alpha_1 < \dots < \alpha_{K}$
$$G_i(\alpha_{i+1}) = G_{i+1}(\alpha_{i+1}),\quad  i=0,\dots, K-1. $$
Here,  each $G_j \in \mathbb{R}[x]_{\leq p}$, for $j=0, \dots, K$ where we recall that $\mathbb{R}[x]_{\leq p}$ denotes the set of all  real polynomials of degree at most equal to $p$. Let $\mathcal{S}_p$ denote the space of functions defined by (\ref{eq:def_pws_poly}). Restricting the solution of the perturbation problem (\ref{eq:equiv_pb}) leads to the following optimization problem
\begin{align}
    \begin{split}
    F_Q^{\leftarrow} = \underset{L \in L^2([0,1])}{\text{argmin }}& \left\{ \int_0^1 \left(L(x) - F^{\rightarrow}_P(x) \right)^2 dx\right\} \\
    \text{s.t. }& \quad L(\alpha_i) = b_i, \quad i=1, \dots, K, \\
    &\quad L \in \mathcal{F}^{\leftarrow} \cap \mathcal{S}_p.
    \label{eq:pb_poly}
    \end{split}
\end{align}
Hence, the smooth perturbation class is defined by $\mathcal{V} =\mathcal{F}^{\leftarrow} \cap \mathcal{S}_p$. Since the design
enforced the polynomials in $\mathcal{S}_p$ to be defined on the grid $\alpha_0<\alpha_1 < \dots < \alpha_K< \alpha_{K+1}=1$, solving (\ref{eq:pb_poly}) reduces to solve several sub-problems on each sub-interval $[\alpha_i, \alpha_{i+1}], i=0,\dots, K$ of $[0,1]$. The optimization problem is indeed separable into $K+1$ independent optimization sub-problems. Each of them defines an optimal component $G_i$ of the piece-wise polynomial $G$ as defined in (\ref{eq:def_pws_poly}). 

Any of these problems can be formulated generically as follows. Let $[t_0, t_1] \subset [0,1]$, and $z_0, z_1 \in \mathbb{R}$ be interpolation values at $t_0$ and $t_1$ respectively. We aim to find the solution to the optimization sub-problem
\begin{align}
\begin{split}
    S = \underset{L \in \mathbb{R}[x]_{\leq p}}{\text{argmin }}& \left\{ \int_{t_0}^{t_1} (F^{\leftarrow}_P(x) - L(x))^2 dx \right\}\\
    \text{s.t. }& L(t_0) = z_0, L(t_1) = z_1,\\
    & L'(x) \geq 0, \quad \forall x \in [t_0, t_1].
    \label{eq:polyPb_interv}
\end{split}
\end{align}
This optimization sub-problem is nothing more than the $L^2$ isotonic (i.e., monotonic, in this case non-decreasing) polynomial approximation on a compact interval \citep{Murray2016}, with interpolation constraints at the boundaries. The interpolating polynomials have been extensively studied in the literature \citep{Fredenhagen1999}, as well as isotonic polynomial regression and approximation \citep{Schmidt1988,Wang2008}. However, to our knowledge, this specific optimization problem does not seem to have been particularly studied.

We propose to solve (\ref{eq:polyPb_interv}) using the \textit{sum-of-squares} (SOS) polynomials \citep{Lasserre2015} representation of non negative polynomials. Furthermore, we leverage in particular the representation of SOS polynomials using semi-definite positive (SDP) matrices \cite{Parrilo2010, Parrilo2012,Sobrie2018}. A similar characterization of isotonic polynomials has been proposed in \cite{Sobrie2018}. The following result, the second main outcome of this article, shows that the problem to solve falls into the category of strictly convex programs: the solution in (\ref{eq:equivCCQP}) is unique \citep{Bertsekas2016}. 

\begin{thm}
Let $[t_0, t_1] \subset [0,1]$. Let $M$ be the symmetric positive definite $\left((d+1) \times (d+1)\right)$ moment matrix of the Lebesgue measure on $[t_0, t_1]$, i.e. for $i,j=1,\dots, d+1$,
\begin{equation}
    M_{ij} = \int_{t_0}^{t_1} x^{i+j-2}dx = \frac{(t_1)^{i+j-1} - (t_0)^{i+j-1}}{i+j-1},
    \label{eq:compute_M}
\end{equation}
and denote $r \in \mathbb{R}^{d+1}$ the moment vector of $F^{\rightarrow}_P(x)$, i.e., for $i=0,\dots, d$
\begin{equation}
    r_i = \int_{t_0}^{t_1} x^iF^{\rightarrow}_P(x)dx.
    \label{eq:compute_r}
\end{equation}
Then, the vector $s^* = (s_0, \dots, s_d)^\top \in \mathbb{R}^{d+1}$ of coefficients characterizing the polynomial $S$ in (\ref{eq:polyPb_interv}) is the solution of the following convex constrained quadratic program
\begin{align}
\begin{split}
        s^* = \underset{s \in \mathbb{R}^{p+1}}{\text{argmin }} & s^\top M s -2s^\top r\\
    \text{s.t. } & s \in \mathcal{K},
    \label{eq:equivCCQP}
\end{split}
\end{align}
where $\mathcal{K}$ is an identifiable closed convex subset of $\mathbb{R}^{p+1}$ (for the sake of conciseness, $\mathcal{K}$ is characterized within the proof).
\label{thm:opti_pol_CCQP}
\end{thm}

As the computation of $s^*$ is a convex constrained quadratic program, it can be addressed efficiently using devoted solvers. The problem (\ref{eq:pb_poly}) can be addressed by solving $K+1$ optimization problems of the form (\ref{eq:equivCCQP}). Furthermore, computations can be done in parallel, leading to fast computational times. Notice that (\ref{eq:equivCCQP}) can be formulated and solved using \texttt{CVXR}. This is an \texttt{R} package for disciplined convex programming \citep{CVXR2020}. The pretty generic low-level logic behind the optimization scheme can be found in Algorithm~\ref{alg:solvePoly}.

%---------% Begin Algo %---------%
\begin{algorithm}
\caption{Isotonic interpolating piece-wise continuous polynomial optimization strategy}\label{alg:solvePoly}
\begin{algorithmic}[1]
\Require $\alpha$, $b$, $F^{\rightarrow}_P$, $p$
\For{$i=0, \dots, K$} (in parallel)
    \State Compute $M$ on $[\alpha_i, \alpha_{i+1}]$ (\ref{eq:compute_M}).
    \State Compute $r$ on $[\alpha_i, \alpha_{i+1}]$ (\ref{eq:compute_r}).
    \State Setup \texttt{CVXR} constraints.
    \State $s^{(i)} \leftarrow$ Solve (\ref{eq:equivCCQP}).
    \State $G_i(x) \leftarrow \sum_{j=0}^p s^{(i)}_j x^j$
\EndFor
\State \Return $G(x) \leftarrow \sum_{i=0}^{K} G_i(x) \mathds{1}_{[\alpha_i, \alpha_{i+1}]}(x)$
\end{algorithmic}
\label{alg:CVXR_poly}
\end{algorithm}
%---------% End Algo %---------%

While computing the Lebesgue moment matrix $M$ on each sub-interval of $[0,1]$ is straightforward, computing strategies for $r$, the moment vector of $F^{\leftarrow}_P$, can vary depending on the nature of $P$. Additional computational details are given in Appendix~\ref{apdx:compute_r}. The set-up of the \texttt{CVXR} constraints is detailed in the accompanying GitLab repository\footnote{\href{https://gitlab.com/milidris/qcWasserteinProj}{https://gitlab.com/milidris/qcWasserteinProj}}.

To provide a frame of reference for the practical usage of our method, the empirical computational time of solving one element of $G$, w.r.t. the polynomial degree is studied as follows. Values $t_0, t_1 \in [0,1]$, and $z_0, z_1 \in \Omega_X$ are randomly selected, and an isotonic interpolating piece-wise continuous polynomial is fitted (i.e., solving (\ref{eq:equivCCQP})). Polynomials of degrees ranging from 2 to 50 are fitted for each experiment, repeated 150 times. The execution time has been recorded and is displayed in Figure~\ref{fig:isoPoly_compTime}. One can notice that the mean computational time seems to be linear w.r.t. the polynomial degree. However, the higher the degree, the wider the $90\%$ time coverage seems to be, which may be caused by the complexity of the underlying optimization problem. In our limited testing, further numerical experiments showed that small polynomial degrees ($\leq 7$) often appear sufficient to obtain good approximations and that the approximation error tends to stabilize, w.r.t. the polynomial degree, rather rapidly.

\begin{figure}[b!]
    \centering
    \includegraphics[width=\linewidth]{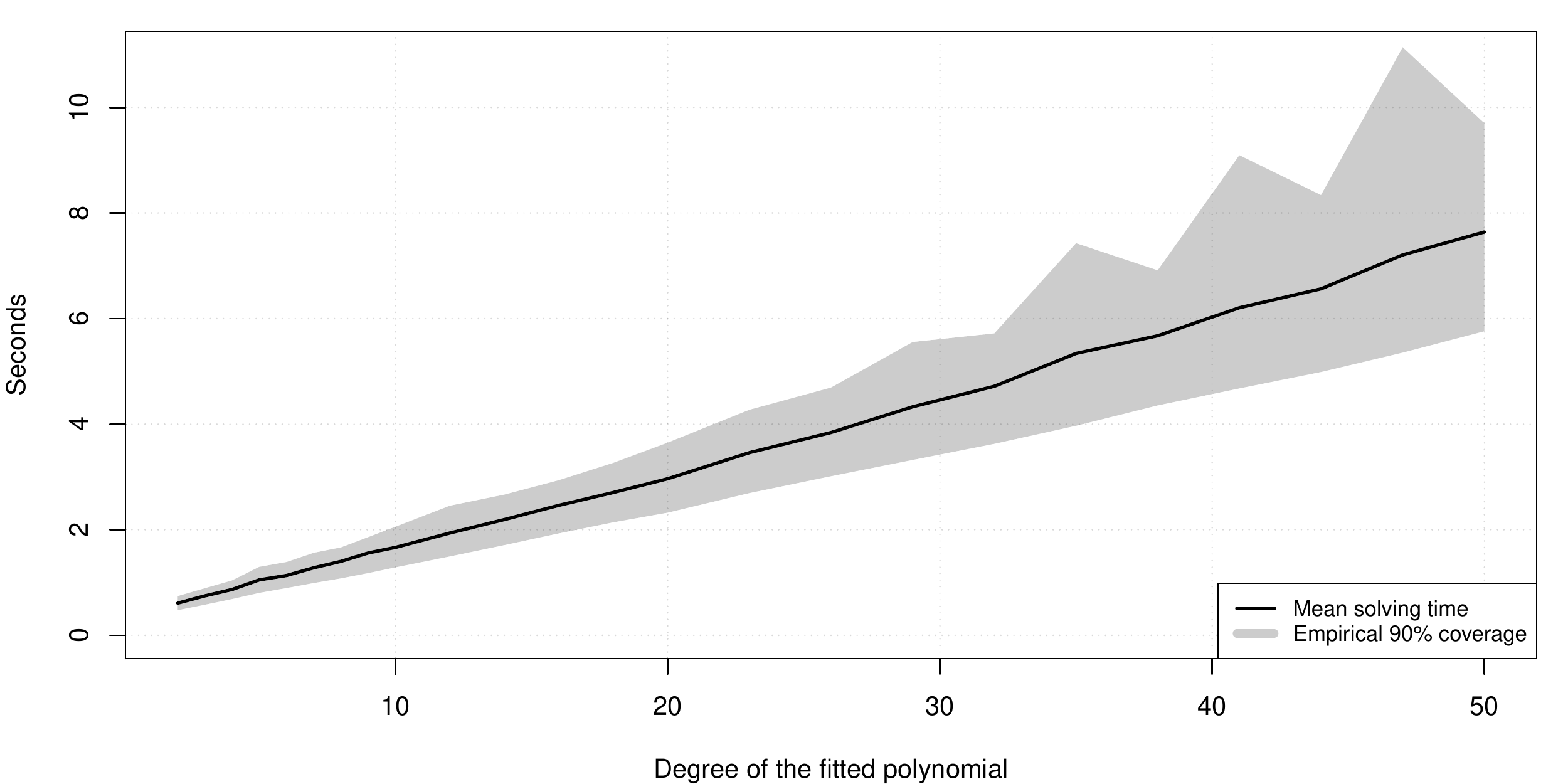}
    \caption{Computational solving time in seconds of the optimization problem~(\ref{eq:equivCCQP}) using \texttt{CVXR}, w.r.t. the chosen degree of the polynomial.}
    \label{fig:isoPoly_compTime}
\end{figure}

\begin{rmk}
The numerical solver used is \texttt{SCS V3.2.1} \cite{SCS2016}. To improve numerical stability, the quantile functions have been mapped to take values between $[-1,1]$. All the figures and all obtained optimal perturbations have been computed by performing this pre-processing step first. 
\end{rmk}

%%%%%%%%%%%%%%%%%%%%%%%%%%%%%%%%%%%%%%%%%%%%%%%%%
% SECTION 5 : Results on Use-Cases
\section{Robustness diagnostics to distributional perturbations}\label{sec:robust_meas}
We illustrate the previous method on two use cases. First, in an ML context, we assess the robustness to feature perturbations of a classification model (i.e., a one-layer neural network) trained on an acoustic fire extinguisher dataset. Then, in the UQ framework, we extend Example \ref{ex:riverWaterLevel}, studying the impact of input perturbation on the output of a numerical hydrological model predicting a river water level.

\begin{rmk}
In the following applications, isotonic polynomial smoothing is applied with an arbitrarily high degree. We chose the degree based on an empirical inspection of the solutions and by ensuring that the approximation error remains relatively the same w.r.t. higher degrees.
\end{rmk}

Our method is in line with the general SIPA (Sampling, Intervention, Prediction, Aggregation) framework for model-agnostic interpretation proposed in \cite{Scholbeck2020}. More precisely, the four step can be broken down as follows:
\begin{enumerate}
    \item {\bf Sampling}: In Section~\ref{sec:AFE}, we have access to an i.i.d. sample of the inputs. In Section~\ref{sec:waterLevel}, a sampling strategy is used to simulate data.
    \item {\bf Intervention}: In both use-cases, we define meaningful quantile perturbations, and solve the subsequent perturbation problem. Then we apply the transformation in (\ref{eq:pertMap}), resulting in perturbed samples.
    \item {\bf Prediction}: We predict using the available black-box model (a neural network in \ref{sec:AFE}, and a numerical model in Section~\ref{sec:waterLevel}), resulting in perturbed outputs.
    \item {\bf Aggregation}: The resulting perturbed outputs are aggregated w.r.t. the perturbation intensity, leading to global robustness metrics, or are simply plotted against the initial and perturbed of the perturbed input values, allowing for local robustness assessments.
\end{enumerate}

%%%%%%%%%%%%%%%%%%%%%%%%%%%%%%%%%%%%%%%%%%%%%%%%%%%%%%%%%%%%%%%%%%%%%%%%%%%%%%%%%%%%%%%%%%%%%%%%%%%%%%%%%%%%%%%%%%%%%
% Acoustic Fire Extinguisher
\subsection{ML application: Acoustic fire extinguisher dataset}\label{sec:AFE}
The acoustic fire extinguisher dataset is composed of 15390 experiments of fire extinguishing tests of three different liquid fire fuels. Amplified subwoofers are placed in a collimator with an opening. When activated at different frequencies, the acoustic waves produce an escape of air through the opening, which is used to extinguish fires. Three features are set using a design of experiment (DoE), and two are measured using appropriate equipment. For more details on the experiments settings, one can refer to the in-depth descriptions in \cite{Koklu2021, Taspinar2022}. Table~\ref{tab:AFE} gives additional details on the nature of the features.

\begin{table}[!h]
    \begin{tabularx}{\textwidth}{lccX}
    \hline
         Feature & Unit & Mode of measure & Description \\
         \hline \hline 
         &&& \\
         TankSize & cm & DoE & Discrete feature (5 levels) describing the size of the tank containing the fuel.\\
         &&& \\
         Fuel & & DoE & Type of fuel used (3 levels: Gasoline, Kerosene, Thinner). \\
         &&& \\
         Distance & cm & DoE & Distance of the flame to the collimator opening. \\
         &&& \\
         Frequency & Hz & DoE & Sound frequency range. \\
         &&& \\
         Decibel & dB & Measured & Sound pressure level. \\
         &&& \\
         Airflow & m/s & Measured & Airflow created by the sound waves.\\
         &&& \\
         \hline
    \end{tabularx}
    \caption{Description of the features of the acoustic fire extinguisher dataset.}
    \label{tab:AFE}
\end{table}
For each experiment, a binary output variable $Y$ is measured, representing the result of the experiment, i.e., whether the fire has been put out ($Y=1$) or not ($Y=0$). The two output classes are relatively balanced (i.e., $48.97\%$ of the observations describe effectively put out fires). The distribution, correlation structure, and relationship of the continuous features with the output are represented in Figure~\ref{fig:AFE_features}. Some variables seem fairly correlated (in Spearman's sense, i.e., the linear correlation of the rank-transformed data), such as Frequency and Decibel, as well as Distance and Airflow. 

\begin{figure}[b!]
    \centering
    \includegraphics[width=\linewidth]{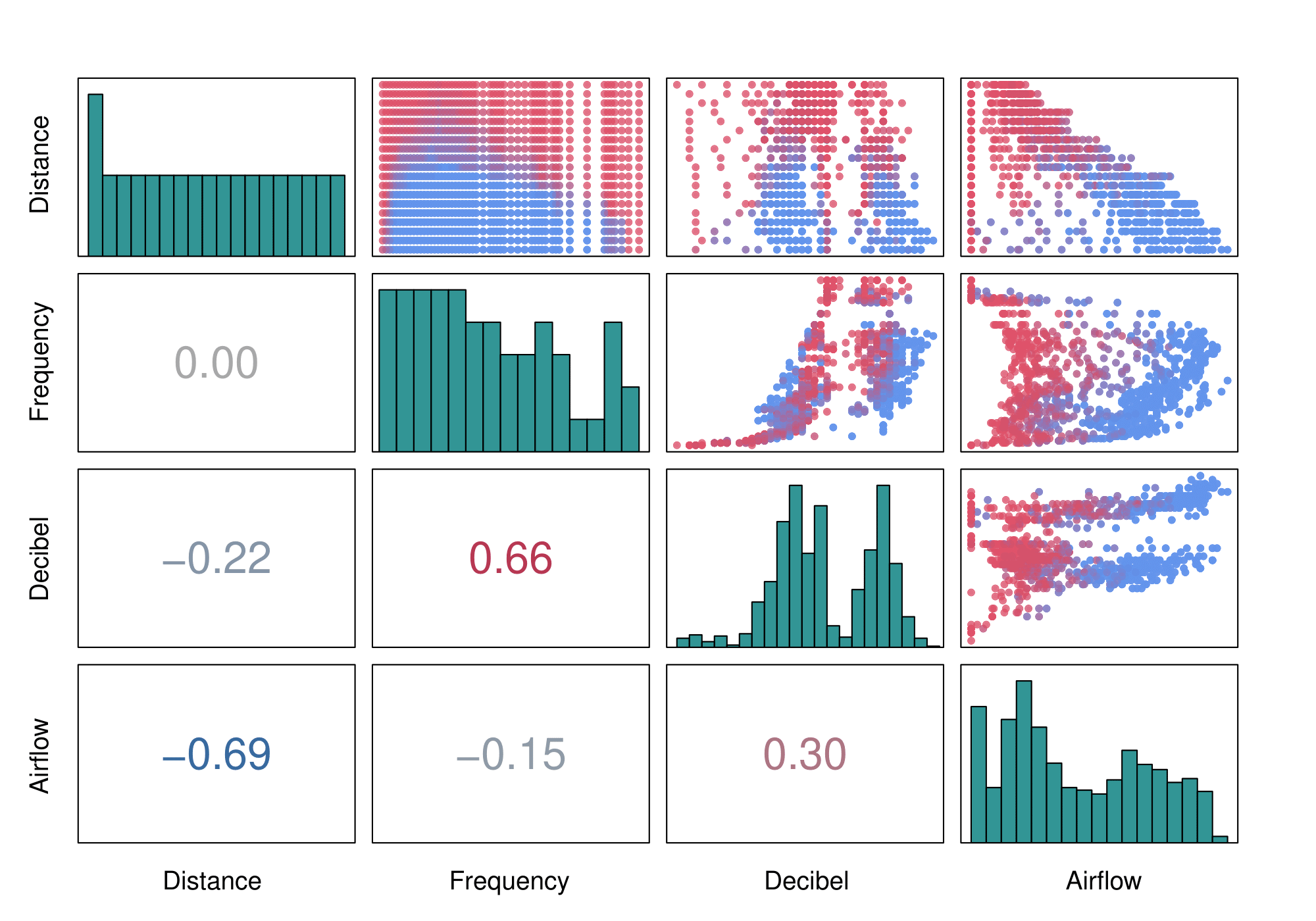}
    \caption{Histogram, cross-scatterplot, and Spearman's correlation coefficient of the input features. Red dots represent observations resulting in $Y=0$, and blue dots are observations resulting in $Y=1$.}
    \label{fig:AFE_features}
\end{figure}

The classification black-box model is a one-layer neural network (composed of 100 neurons), trained on 500 epochs, with a learning rate of $10^{-4}$, similar to the study conducted in \cite{Taspinar2021}. $5\%$ of the data has been randomly selected to serve as validation data. The model resulted in a good prediction accuracy: $95.15\%$ of the training data and $94.26\%$ of the validation data are correctly classified. Figure~\ref{fig:nn_valid} depicts the ROC curve and confusion matrix of the trained black-box model. The model's predictive performance can be validated globally with an AUC of $0.992$ and less than $3\%$ of type 1 and 2 prediction errors.

\begin{figure}[t!]
    \centering
    \includegraphics[width=0.49\textwidth]{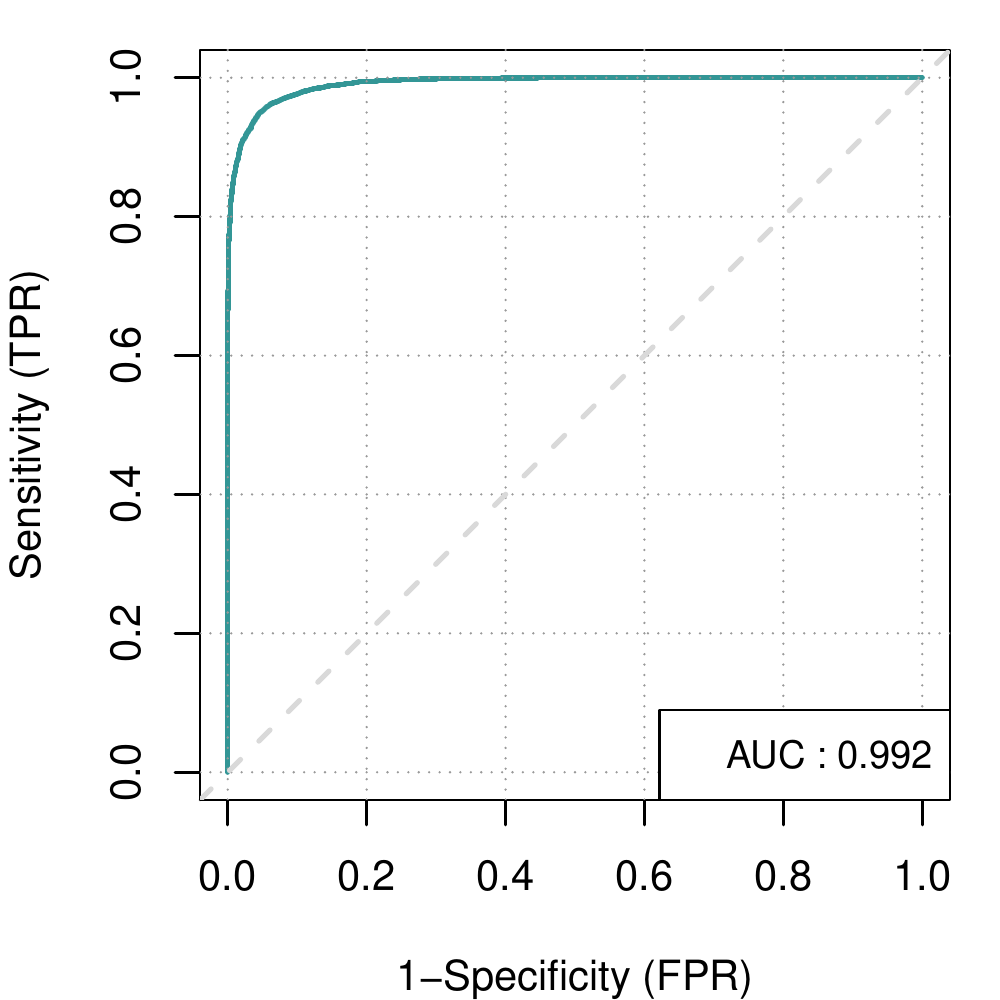}
    \includegraphics[width=0.49\textwidth]{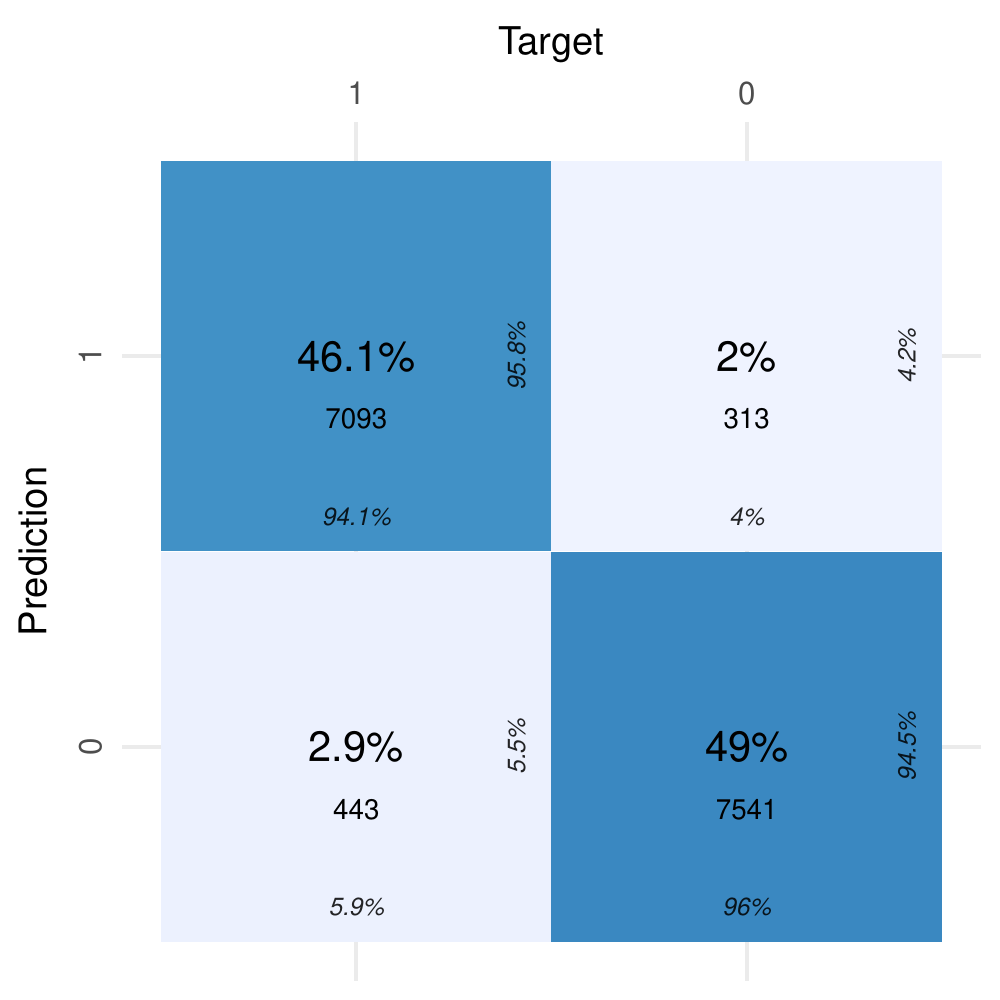}
    \caption{ROC curve (left) and confusion matrix (right) of the neural network model trained on the acoustic fire extinguisher dataset.}
    \label{fig:nn_valid}
\end{figure}

However, global predictive performance only focuses on effectively observed data points. It is mandatory to study the model's behavior on predictions outside of these points to improve confidence in its usage. Hence, one can be interested in the robustness of the model w.r.t. perturbations on its inputs. Note that ground truths cannot be observed for perturbed data. However, the impact, either globally or locally, of these perturbations on the predictive behavior of the model can still be assessed using predictions on the perturbed data. In the following, the feature perturbation scheme is detailed and motivated, and then the model's behavior is studied under these perturbations.

%%%%%%%%%%%%%%%%%%%%%%%%%%%%%%%%%
% AFE Perturbation scheme
\subsubsection{Perturbation strategy}\label{sec:AFE_pert}
We propose a  straightforward perturbation strategy. Only the Airflow feature is perturbed. The perturbation is composed of the $K=14$ constraints:
\begin{itemize}
    \item The application domain of the feature is preserved by setting both the $0$ and $1$-quantiles to the minimum and the maximum observed value of the dataset. 
    \item The left tail of the distribution is preserved by constraining every quantile of level $0.1$ to $0.6$ with a step of $0.05$ to interpolate the empirical quantile function of the feature.
    \item A quantile shift perturbation is put on the $0.8$-quantile of the feature, with an initial value of $F_P^{\leftarrow}(0.8)=12$, being shifted between 9.5 ($\theta=-1$) and 14.5 ($\theta = 1$).
\end{itemize}
Additionally to these perturbations, smoothness is enforced using piece-wise continuous isotonic polynomials, as described in Section~\ref{sec:qcwProj_smooth}. The degree of each monotone polynomial has been arbitrarily chosen to be up to $9$. The constraints and the resulting quantile-constrained Wasserstein projections are illustrated in Figure~\ref{fig:Airflow_qShift} for intensity values $-1$, $0$, and $1$. 
\begin{figure}[t!]
    \centering
    \includegraphics[width=\linewidth]{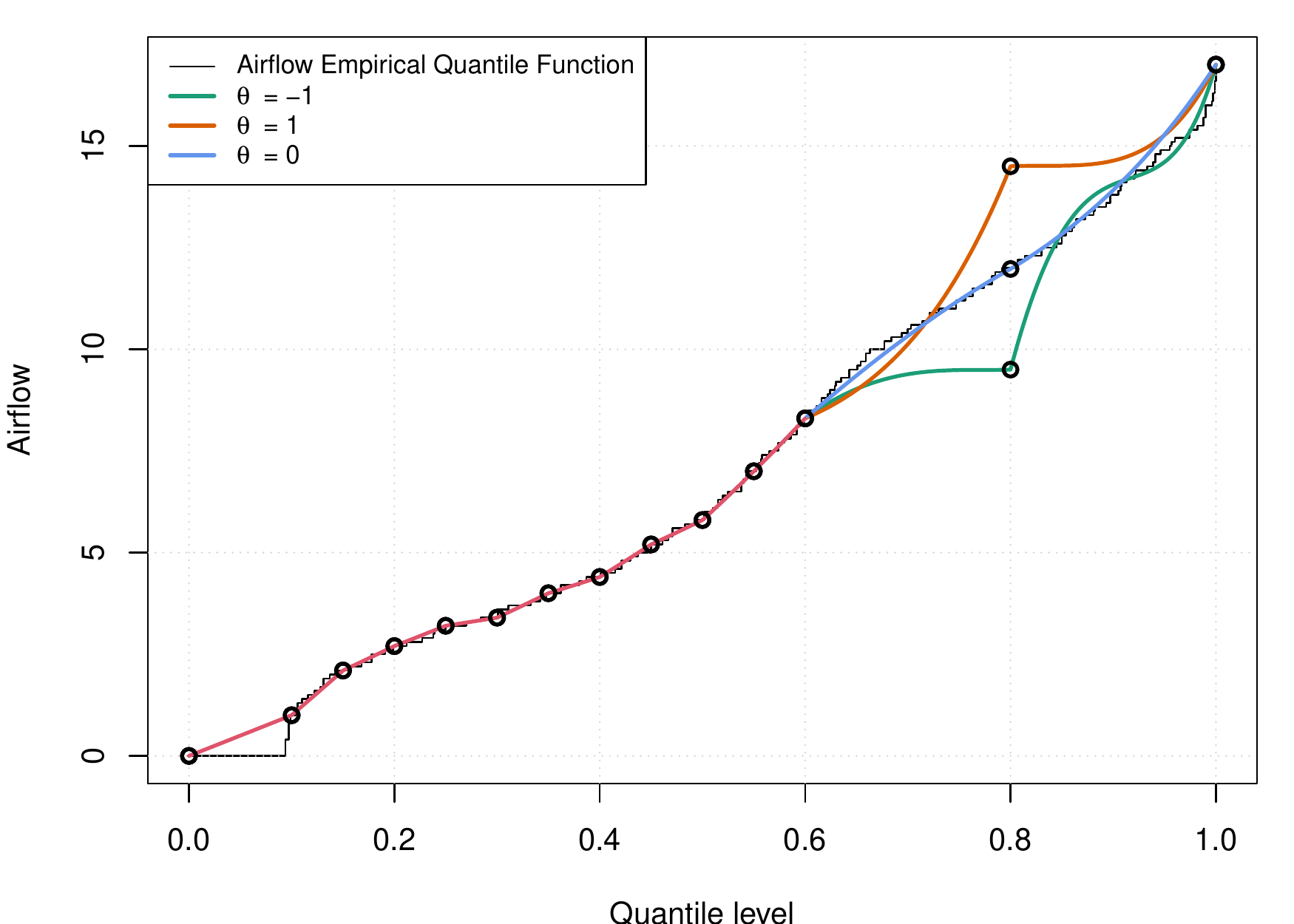}
    \caption{Quantile functions of the optimally perturbed Airflow feature, with a chosen polynomial degree equal to $9$. The red line represents the preserved tail; meanwhile, the green, blue and yellow lines represent various quantile shift intensity levels ($\theta=-1$, $\theta=0$, and $\theta=1$, respectively).}
    \label{fig:Airflow_qShift}
\end{figure}

The perturbed quantile level has been chosen in relation to the model's decision boundary: no observation in the initial dataset with an Airflow value exceeding $12.3$m/s is classified by the model as not extinguishing the fire, regardless of the values taken by the other features. Perturbing the $0.8$-quantile of the Airflow variable allows for exploring the model's behavior in regions close to this decision boundary. More importantly, it allows assessing the predictive robustness of the neural network in this region under perturbations of varying magnitude. Generally, this quantile shift regime can be understood as a perturbation on the right tail of the initial distribution, i.e., on values higher than the $0.6$-quantile.

%%%%%%%%%%%%%%%%%%%%%%%%%%%%%%%%%
% AFE NN Robustness
\subsubsection{Model robustness assessment}
First, we are interested in assessing the robustness of the neural network model in a global fashion. The left plot of Figure~\ref{fig:globRob_props} presents the proportion of perturbed observations with predictions of $1$ w.r.t. to the intensity of the perturbation. Notice that the proportion is increasing, along with $\theta$. Hence, decreasing the value of the initial $0.8$-quantile tend to result in a lower number of predicted put-out fires, and increasing its value results in an increasing number of predicted put-out fires. This interpretation is rather intuitive: all other things being equal, a higher Airflow value entails a higher chance of predicting $Y=1$. The right plot of Figure~\ref{fig:globRob_props} presents the proportion of prediction shift with respect to $\theta$. Notice that the higher the magnitude of the perturbation (either positively or negatively), the more predictions tend to change, and the closest $\theta$ is to $0$, the fewer predictions shift. This observation informs on the predictive stability in the vicinity of the decision boundary of the model: small perturbations tend to result in less prediction shift than bigger perturbations.

\begin{figure}[b!]
    \centering
    \includegraphics[width=\linewidth]{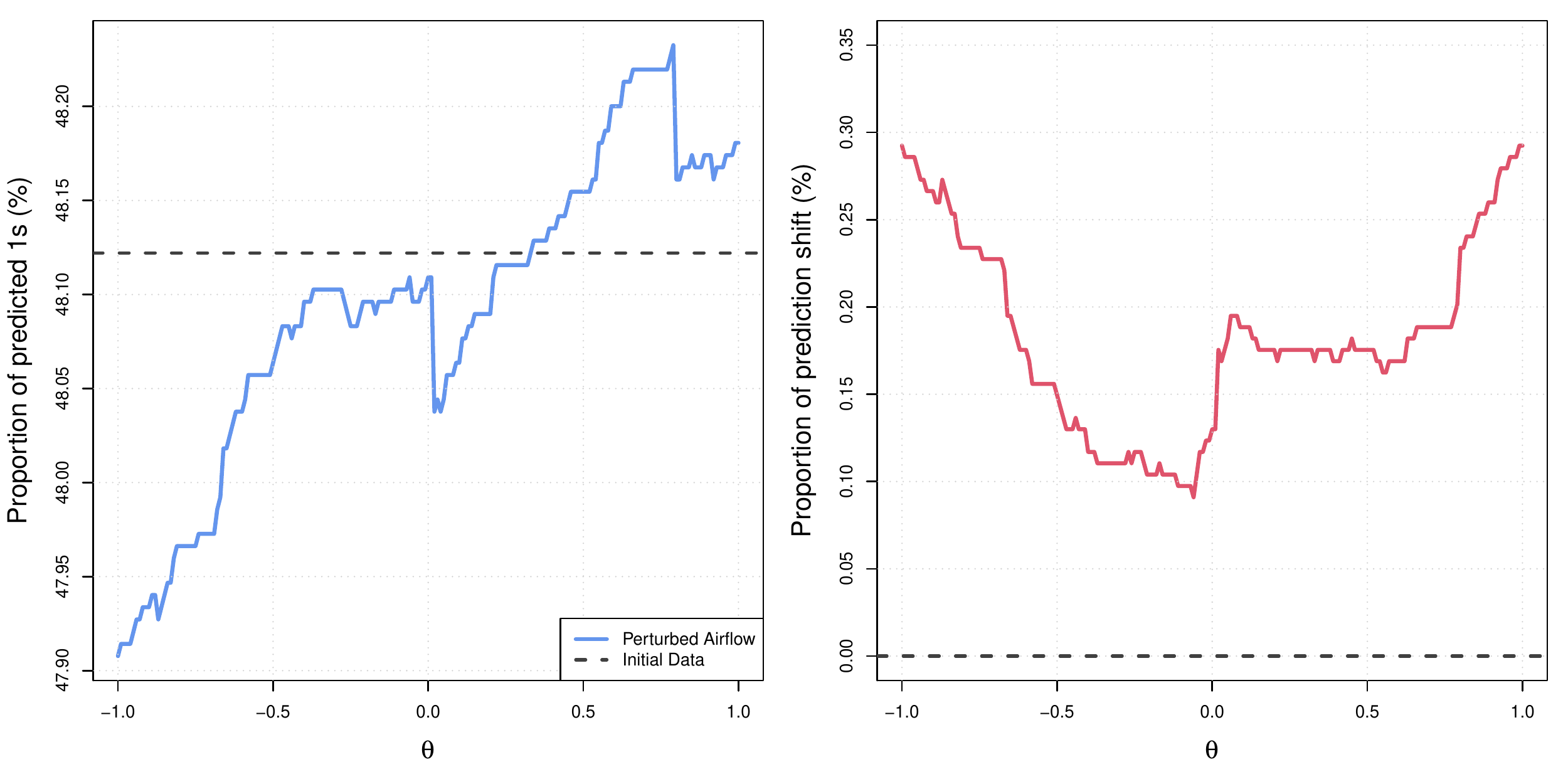}
    \caption{Proportion of predictions $Y=1$ (left) and proportion of classification prediction shift (right) compared to the initial data, w.r.t. the perturbation intensity parameter $\theta$.}
    \label{fig:globRob_props}
\end{figure}

Second, we study the robustness of global SA results. Figure~\ref{fig:globRob_TSE} presents the target Shapley effects \cite{IliCha2021}, a global SA input importance measure for binary black-box model outputs with dependent inputs, w.r.t. the perturbation intensity parameter $\theta$. These indices have been computed using the nearest-neighbor (KNN) approach proposed in \cite{Broto2020} (with an arbitrarily chosen number of neighbors equal to $6$). Recall that our perturbation method allows preserving the empirical copula between the features, justifying the use of the KNN approach. Studying the behavior of importance measures allows to verify if the feature importance order shifts due to the perturbations, i.e., if the importance hierarchy between the inputs changes due to perturbations around the model's decision boundary. The left barplot presents the initial target Shapley effects, computed on the model's prediction on the observed data, and the right plot presents their behavior under the airflow perturbation. One can notice that the importance indices remain stable w.r.t. $\theta$. This result indicates that the global SA of the neural network is robust to the distributional perturbations driven by $\theta$. Hence, we can be confident in the importance measures under uncertainties in the region near the model's decision boundary.

\begin{figure}[t!]
    \centering
    \includegraphics[width=\linewidth]{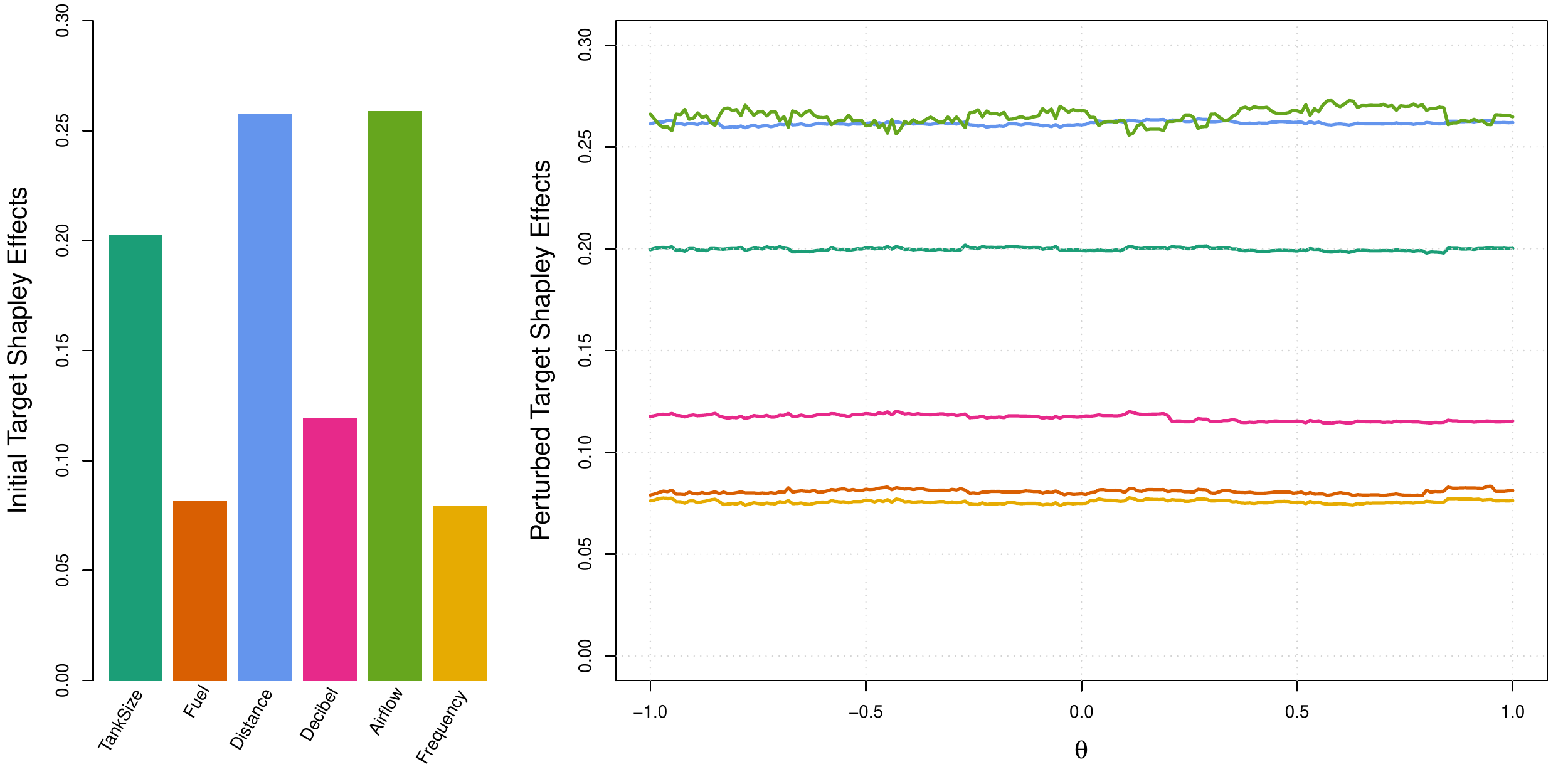}
    \caption{Initial (left) and perturbed (right) target Shapley effects, w.r.t. the intensity parameter $\theta$, using the same color panel.}
    \label{fig:globRob_TSE}
\end{figure}

Finally, the robustness of the neural network can also be assessed locally. Figure~\ref{fig:locRob_pertMagni} allow visualizing whether a prediction has shifted w.r.t. to the effective magnitude of the perturbation. The black line indicates no perturbation change: the airflow value of an observation has been mapped to itself. For a fixed initial airflow datapoint, its vertical distance to the black line indicates the (signed) magnitude of the applied perturbation. Red points indicate that the prediction has shifted w.r.t. the initial dataset, and blue points indicate no predictive change. One can note the presence of red dots close to the black line around the prediction boundary of the model. Small perturbations for observations with airflow values around $12$, all other features being equal, can lead to a prediction change. Hence, the confidence in predictions on observations in this region can be questioned. However, notice the lack of red dots near the black line for airflow values on the interval $[13,17]$ and on the interval $[7,10]$. Hence, we can be confident in the model's predictions for Airflow values on these intervals, which seem to be robust w.r.t. the quantile shift.

\begin{figure}[t!]
    \centering
    \includegraphics[width=\linewidth]{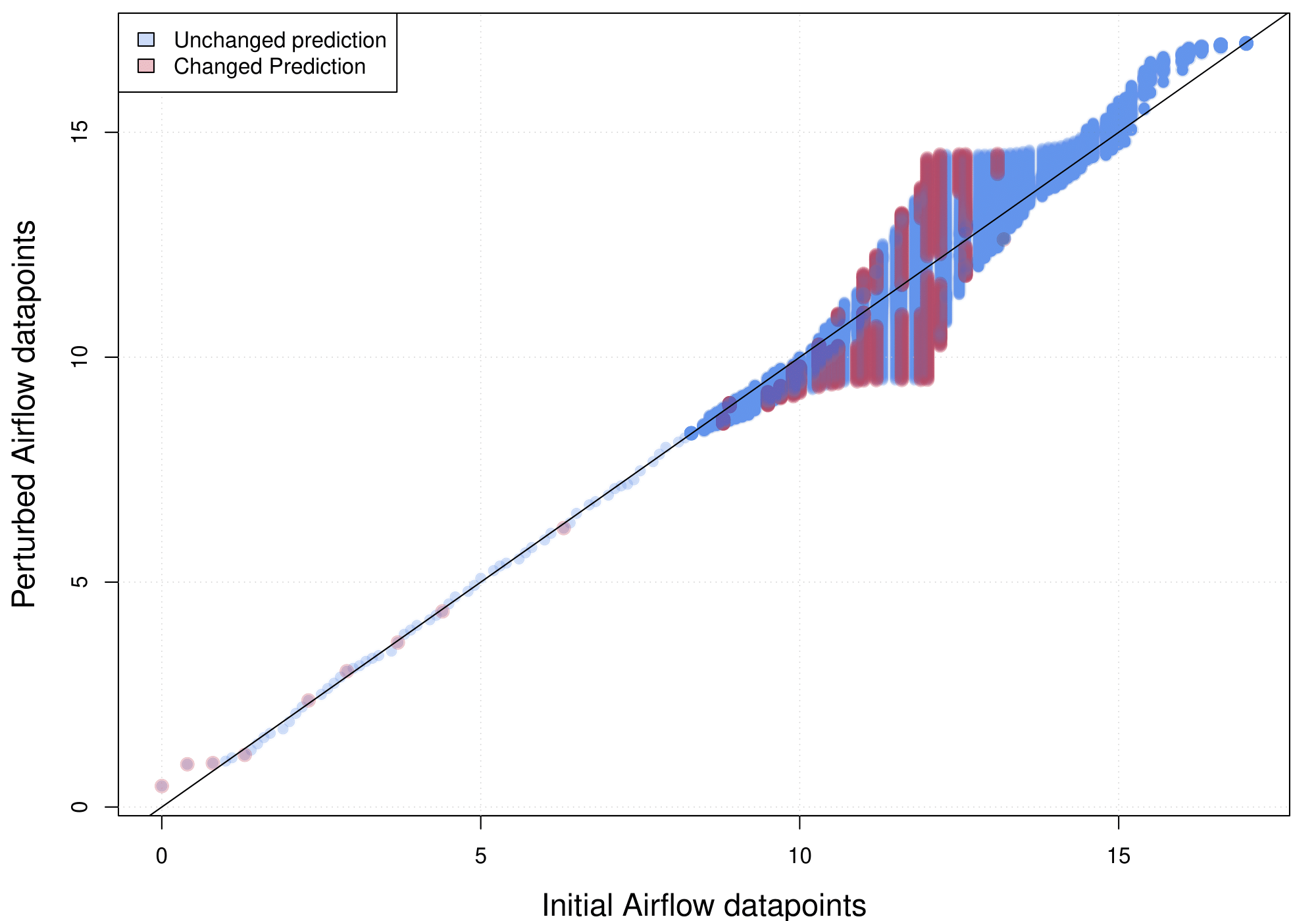}
    \caption{Perturbed datapoints w.r.t. their initial values. The black line represents no perturbation. The red and blue dots represent either a classification shift due to the perturbation or no classification shift, respectively.}
    \label{fig:locRob_pertMagni}
\end{figure}

One may notice the presence of small perturbation resulting in prediction changes for small airflow values. However, since the perturbation scheme focuses on exploring the model's behavior around its Airflow decision boundary, their interpretation is voluntarily omitted: a different perturbation scheme involving perturbing the left tail of the airflow distribution would be advised.

In summary, besides the model's good prediction accuracy, it also seems globally robust to distributional perturbation focused around its decision boundary. Moreover, we can be confident in the feature importance indices since they remain relatively similar under perturbation. Locally, the model prediction seems stable w.r.t. small perturbations, except on a small interval around its decision boundary (a behavior generally expected in ML applications). In conclusion, this robust interpretability analysis further assesses the model's behavior beyond classical accuracy metrics and provides additional arguments for its validation.

%%%%%%%%%%%%%%%%%%%%%%%%%%%%%%%%%%%%%%%%%%%%%%%%%%%%%%%%%%%%%%%%%%%%%%%%%%%%%%%%%%%%%%%%%%%%%%%%%%%%%%%%%%%%%%%%%%%%%
% Simplified flood model
\subsection{SA application: Simplified hydrological  model}\label{sec:waterLevel}
This use-case focuses on a simplified model of the water level of a river. This model has been extensively used in the safety and reliability of industrial sites, where the occurrence of a flood can lead to dramatic human and ecological consequences. It consists of a substantial simplification of the one-dimensional Saint-Venant equation, with a uniform and constant flow rate, inspired from  \cite{Iooss2011,Fu2017}. The maximal annual water level from sea level is modeled as:
$$Y = Z_v + \left(\frac{Q}{B K_s \sqrt{\frac{Z_m - Z_v}{L}}}\right)^{3/5}$$
where the description of each input variable and their explicit marginal probabilistic structure is detailed in Table~\ref{tab:RivWatLev_distribs}. 
\begin{table}[ht!]
    \centering
    \begin{tabularx}{\textwidth}{lcccX}
        \hline
        Input & Unit & Distribution & Application Domain & Description\\
        \hline
        \hline
        &&&&\\
        $Q$ & m$^3$/sec& $\mathcal{G}(1013,558)$ trunc. & $[500,3000]$ & River maximum annual water flow rate. \\
        &&&&\\
        $K_s$ &  & $\mathcal{N}(35,5)$ trunc. &$[20, 50]$ & Strickler riverbed roughness coefficient.\\
        &&&&\\
        $Z_v$ & m & $\mathcal{T}(49,50,51)$ & $[49,51]$ & Downstream river level.\\
        &&&&\\
        $Z_m$ & m & $\mathcal{T}(54,55,56)$ & $[54,56]$ & Upstream river level.\\
        &&&&\\
        $L$ & m & $\mathcal{T}(4990,5000,5010)$ & $[4990,5010]$ & River length.\\
        &&&&\\
        $B$ & m & $\mathcal{T}(295,300,305)$& $[295,305]$ & River width.\\
        &&&&\\
        \hline
    \end{tabularx}
    \caption{Inputs of the simplified river water level model and their explicit marginal distributions. $\mathcal{G}, \mathcal{N}, \mathcal{T}$ denote Gumbel, Normal and Triangular distributions, respectively (trunc means truncated).}
    \label{tab:RivWatLev_distribs}
\end{table}

Additionally, similarly to \cite{Chastaing2012}, a dependence structure is modeled using a Gaussian copula, with the covariance matrix
$$R_P = \begin{pmatrix}
1&0.5&0&0&0&0 \\
0.5&1&0&0&0&0 \\
0&0&1&0.3&0&0 \\
0&0&0.3&1&0&0 \\
0&0&0&0&1&0.3\\
0&0&0&0&0.3&1
\end{pmatrix}, \quad \text{where} \quad \begin{pmatrix}Q \\ K_s \\ Z_v \\ Z_m \\ L \\ B \end{pmatrix} \sim P.$$

Echoing Example~\ref{ex:riverWaterLevel}, we are interested in uncertainties on the application domain of the $K_s$ input, i.e., the Strickler riverbed roughness coefficient (which is the inverse of the Manning coefficient). Its value can range from around $3$ (proliferating algae) to around $90$ (smooth concrete). We refer the interested reader to the in-depth study in \cite{Fu2017} for more details on the determination and inference of the Strickler coefficient for realistic rivers. In this use-case, initially, the application domain $\Omega_X$ of the Strickler coefficient is set between the values of $20$ and $50$, corresponding to situations from very cluttered riverbeds to earthen channels. However, to illustrate our robustness method, epistemic uncertainties are assumed to affect this application domain.

%%%%%%%%%%%%%%%%%%%%%%%%%%%%%%%%%
% River Model Input Perturbations
\subsubsection{Perturbation strategy}\label{sec:waterLevel_pert}
In this use case, the three following inputs are perturbed. The river maximum annual water flow rate $Q$, the river length $L$, and the upstream river level $Z_m$ are subject to the following quantile constraints:
\begin{itemize}
    \item Quantile perturbations on $Q$:
    \begin{itemize}
        \item Shift of the application domain from $[500,3000]$ to $[500,3200]$;
        \item Preserve the median of the distribution;
        \item Increase the initial $0.15$-quantile by $75$;
        \item Decrease the initial $0.75$-quantile by $125$;
    \end{itemize}
    \item Quantile perturbations on $L$:
    \begin{itemize}
        \item Shift the application domain from $[4990,5010]$ to $[4988,5012]$;
        \item Preserve the median of the distribution;
    \end{itemize}
    \item Quantile perturbations on $Z_m$:
    \begin{itemize}
        \item Preserve the application domain and the median of the initial distribution;
        \item Increase the $0.8$ and $0.9$-quantiles by $0.1$;
        \item Decrease the $0.25$-quantile by $0.05$.
    \end{itemize}
\end{itemize}
The initial input distributions, their application domain, and the optimally perturbed results are illustrated in Figure~\ref{fig:LQZM_perts}. These constraints are mainly enforced to illustrate that multiple inputs can be perturbed simultaneously while preserving their dependence structure.
\begin{figure}[b!]
    \centering
    \includegraphics[width=\linewidth, trim={0cm 2cm 0cm 0cm}]{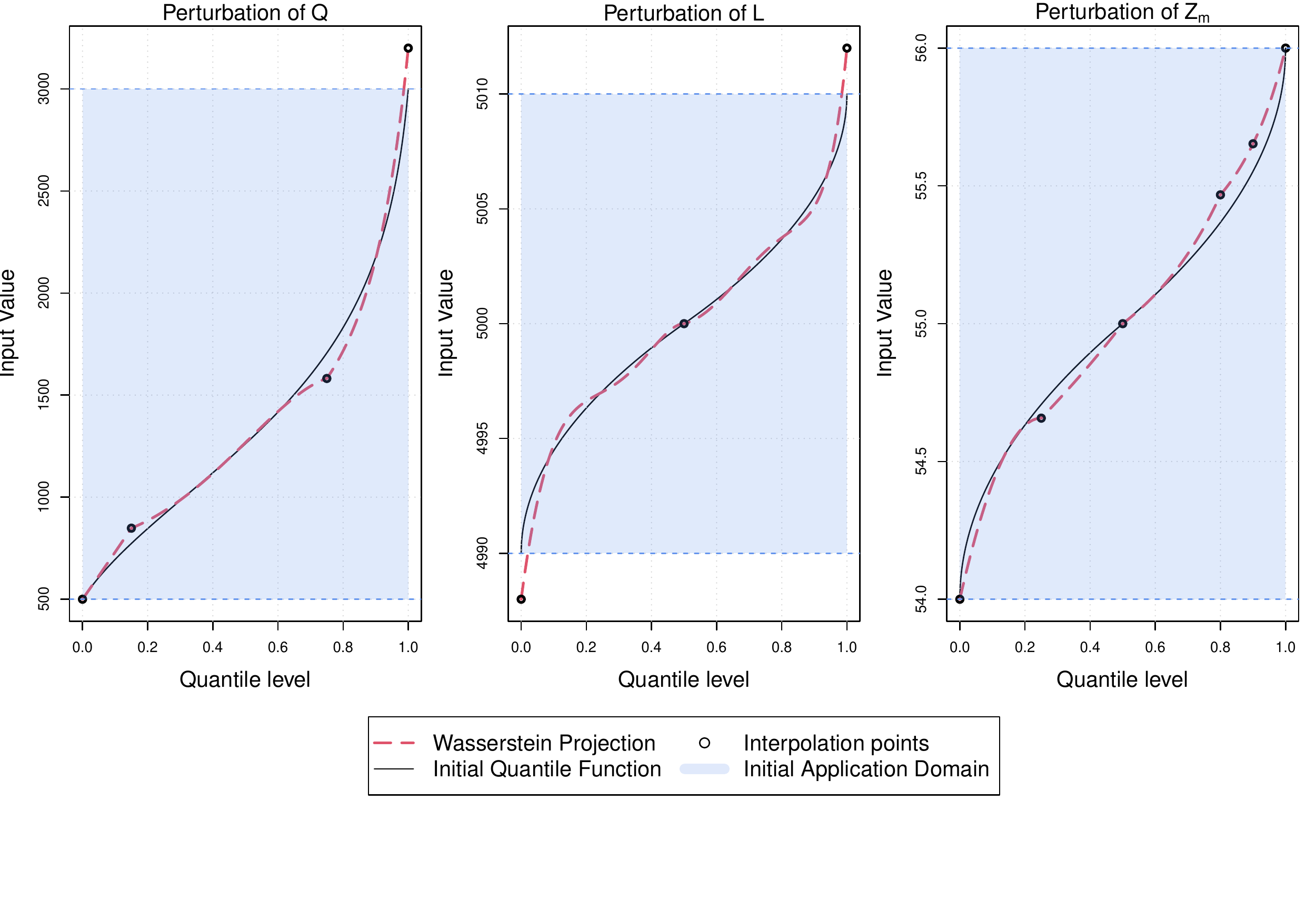}
    \caption{Initial quantile functions, application domains, and corresponding optimally perturbed quantile functions of the $Q$, $L$, and $Z_m$ inputs.}
    \label{fig:LQZM_perts}
\end{figure}

In addition to these constraints, the Strickler coefficient $K_s$ is subject to an application domain dilatation perturbation, with a scaling parameter $\eta=2$. Each perturbation intensity represents a degree of uncertainty on the type of riverbed roughness. When $\theta = -1$, the width of the initial application domain is halved, i.e., from $[20,50]$ to $[27.5, 42.5]$, which can be interpreted in a situation where the epistemic uncertainty on the riverbed roughness is narrower, between a slow winding natural river, up to a plain river without shrub vegetation. When $\theta=1$, the epistemic uncertainty on the riverbed is much wider, with an application domain equal to $[5, 65]$, which depicts a range of riverbed roughness from proliferating algae up to smooth concrete. Figure~\ref{fig:Ks_perts} illustrates the initial $K_s$ distribution, along with the optimally perturbed quantile functions for $\theta$ being equal to $-1$ and $1$.

\begin{figure}[b!]
    \centering
    \includegraphics[width=\linewidth, trim={0cm 2cm 0cm 0cm}]{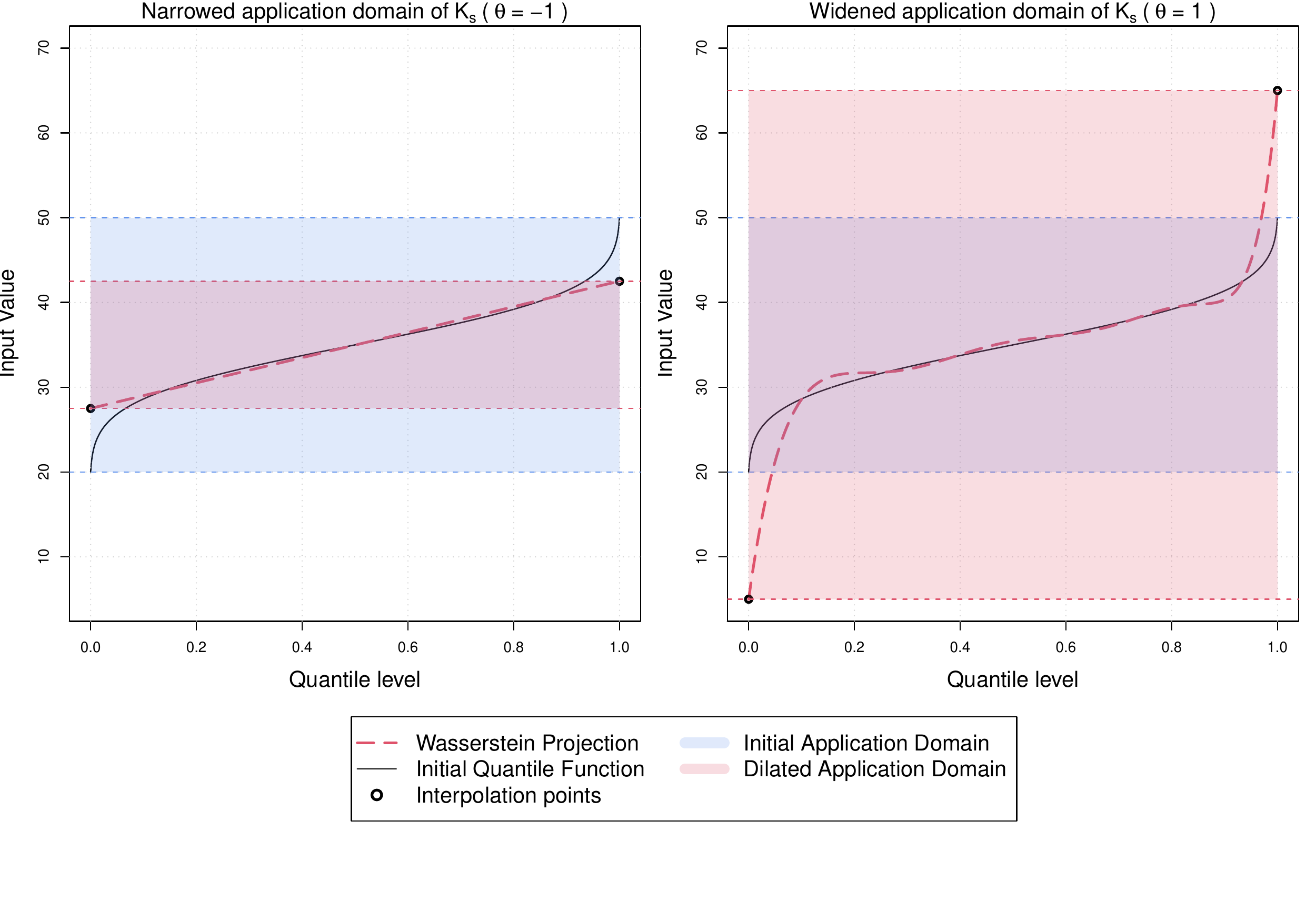}
    \caption{Initial quantile function, application domain and corresponding optimally perturbed quantile functions for $K_s$, for $\theta$ being equal to $-1$ (left) and $1$ (right), for a scaling parameter $\eta=2$.}
    \label{fig:Ks_perts}
\end{figure}
Additionally, the perturbations' smoothness is enforced using piece-wise continuous isotonic polynomials of degree up to $12$, chosen arbitrarily.

%%%%%%%%%%%%%%%%%%%%%%%%%%%%%%%%%
% River Model SA robustness
\subsubsection{Robustness of the sensitivity analysis}
From a global standpoint, one can be interested in the impact of the distributional perturbations on key statistics of the random output of the river water level model. Figure~\ref{fig:statsY} presents estimated values for the mean, standard deviation, $0.025$ and $0.975$-quantiles (shown by the $95\%$ coverage), and minimum and maximum values of the random output, computed on $10^5$ Monte Carlo samples, w.r.t. the dilatation intensity $\theta$. These values are compared to the reference ones according to the initial distribution of the inputs, estimated on a $2 \times 10^5$ Monte Carlo sample.

\begin{figure}[t!]
    \centering
    \includegraphics[width=\linewidth, trim={0cm 2cm 0cm 0cm}]{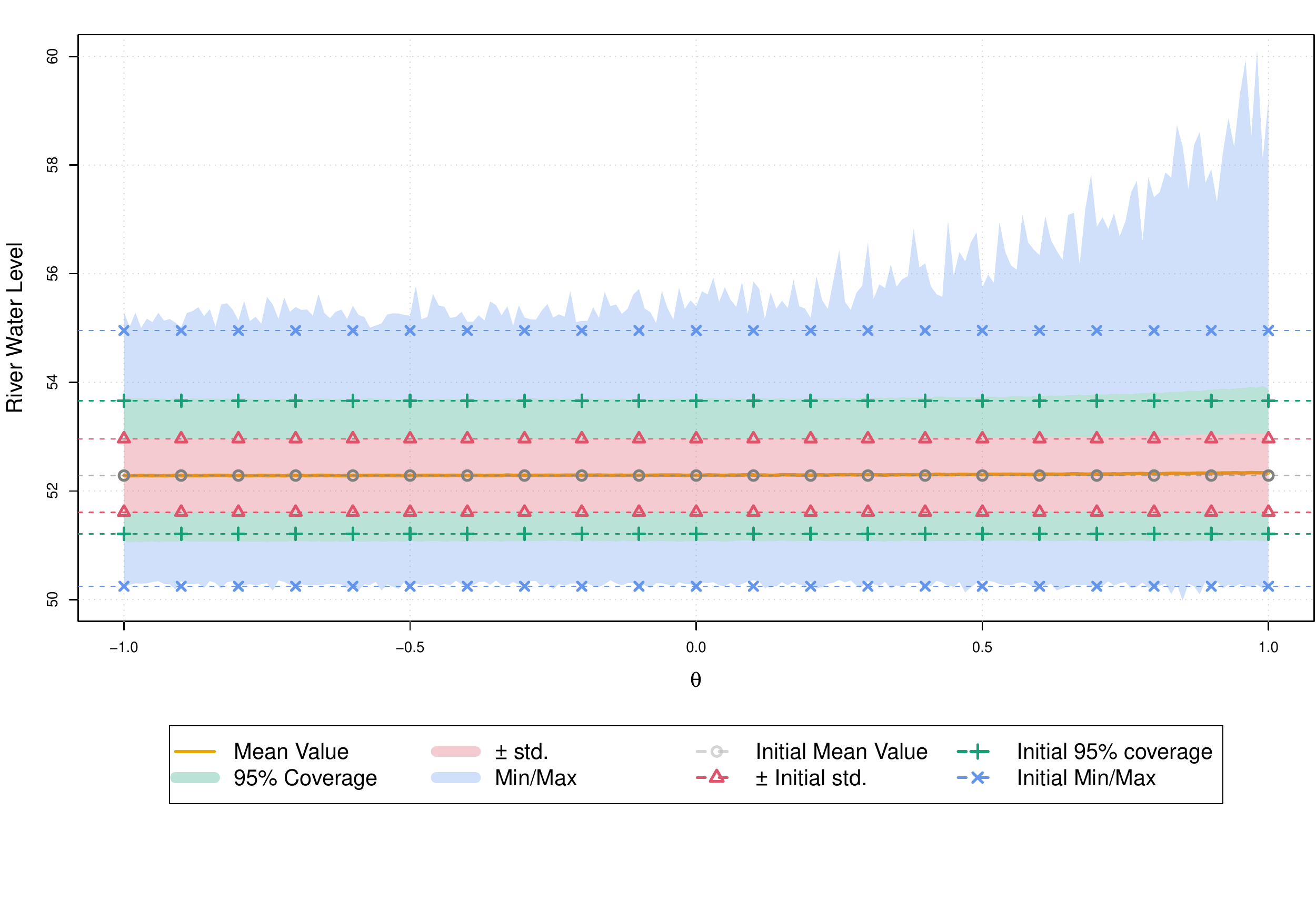}
    \caption{Expectation, standard deviation, 95\% coverage, minimum and maximum estimators of the river water level, w.r.t. the application domain dilatation intensity $\theta$.}
    \label{fig:statsY}
\end{figure}

Notice that the expectation, standard deviation, 95\% coverage quantiles, and minimum value of the model output remain stable under the distributional perturbations on the application domain of the Strickler coefficient. However, the estimated upper bound of the output support increases exponentially for positive values of $\theta$. Widening the uncertainty on the type of riverbed allows for relatively rare events of high river water levels since the $0.975$-quantile does not seem to be dramatically affected by the distributional perturbations.

Figure~\ref{fig:river_shapEffects} presents the Shapley effects \cite{Owen2014}, which are global SA importance measure for real-valued model outputs with dependent inputs. These indices have been computed using a double Monte Carlo scheme as depicted in \cite{Song2016}, with fixed simulated sample sizes, for each perturbed distribution $Q$ driven by a value of $\theta$, $N_v= 10^4$ for estimating $\mbox{Var}_Q(Y)$, as well as $N_o= 10^3$ and $N_i= 100$ to estimate $\mathbb{E}_Q\left[ \mbox{Var}_Q\left(Y \mid X_A\right) \right]$ for every subset $X_A, A \subseteq \{1,\dots,d\}$ of variables. Additionally, the reference Shapley effects have been computed under the initial distribution with sample sizes $N_v = 10^5$, $N_o = 3 \times 10^3$ and $N_i = 300$.

\begin{figure}[t!]
    \centering
    \includegraphics[width=\linewidth]{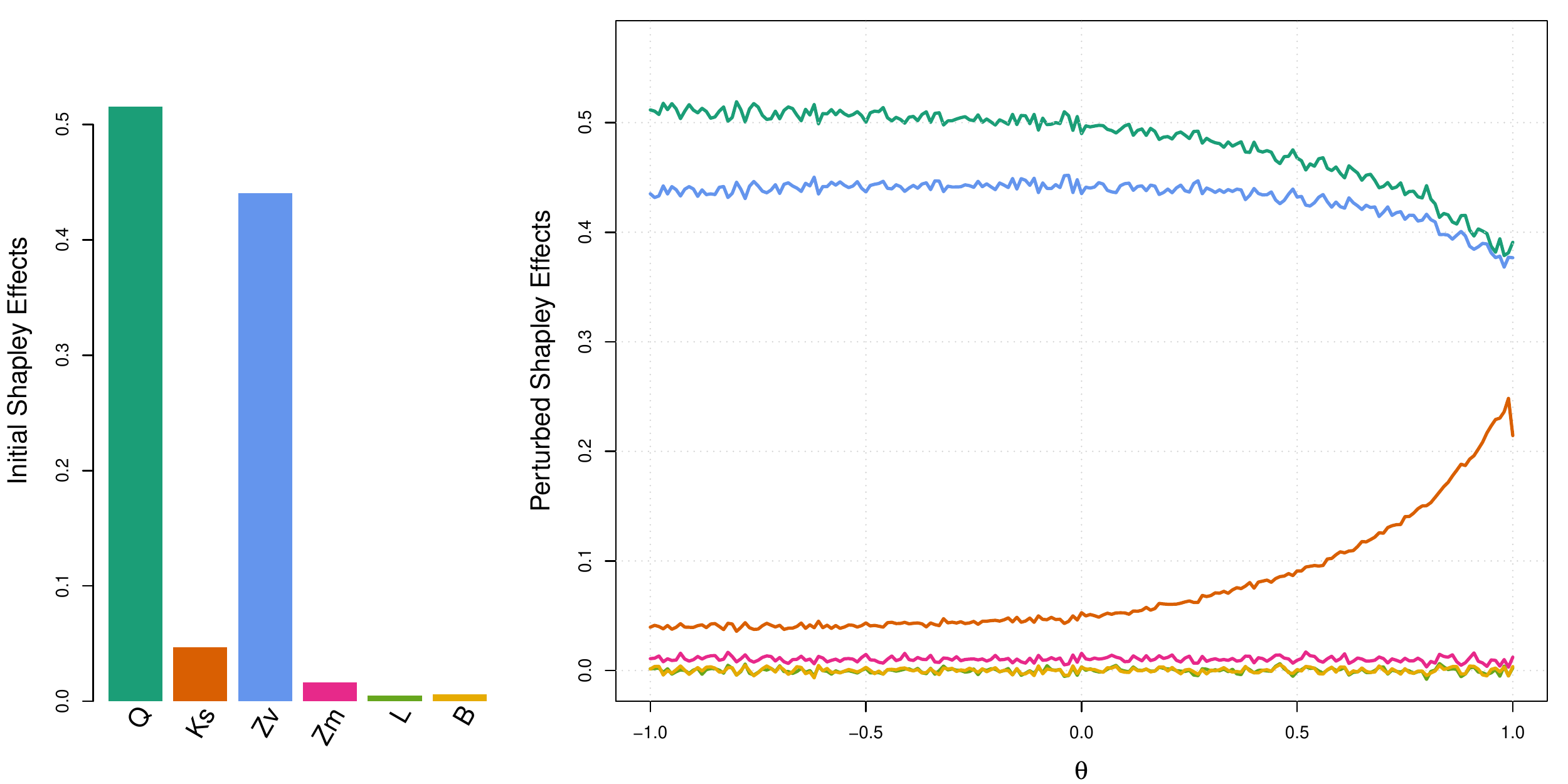}
    \caption{Reference Shapley effects (left) and Shapley effects of the river water level model under optimally dilated application domain w.r.t. $\theta$ (right), using the same color panel.}
    \label{fig:river_shapEffects}
\end{figure}

Note that the distributional perturbations have an impact on the importance measures. More precisely, increasing the range of the uncertainty of the riverbed roughness increases its importance for positive values of $\theta$. Conversely, the importance of both $Q$ and $Z_v$ decreases accordingly. However, the variable importance hierarchy induced by the Shapley effects is preserved. It is also essential to notice that both $Q$ and $Z_v$ tend to be considered equally important as $\theta$ gets large. Hence, this SA does not seem robust to distributional perturbations and, more precisely, to a widening of the support of the Strickler coefficient in combination with the quantile perturbations put on $Q$, $L$, and $Z_m$.

\subsection{Conclusions}

These two use cases illustrate the usefulness of our method in both UQ and ML studies. On the ML side, for classification tasks, it allows assessing the global behavior of black-box models under input perturbations. This assessment is quantified either through studying the prediction shifts due to the perturbation, or through the behavior of feature importance metrics. Locally, it allows the detection of low-stability regions of interest (regions where small perturbations induce a classification change). Overall, in addition to classical accuracy metrics, our method can be used to assess confidence in a predictive model. On the UQ side, it allows for studying the impact of distributional perturbations (whose intensity can be tuned to represent epistemic uncertainties) on the model output, even in situations where inputs are correlated. Furthermore, in a SA context, the behavior of classical sensitivity indices under those perturbations can also be studied, and their robustness (for instance, the preservation of the input importance hierarchy) w.r.t. the probabilistic modeling on the inputs can be assessed.

%%%%%%%%%%%%%%%%%%%%%%%%%%%%%%%%%%%%%%%%
% SECTION 6 : Discussion & perspectives
\section{Discussion and perspective}\label{sec:conclusion}
Obtaining a robustness diagnosis on the influence of input variables and the behavior of a model considered a black box is essential for its acceptance and use. Such models can either implement solutions of mechanistic equations or be learned from data. In both cases, the sensitivity of key indicators to a misspecification of the probabilistic input model must be evaluated. It is essential to have specific intelligible rules and well-defined computational tools to achieve this goal. This paper provides an answer to this question by proposing to modify the distributions of the input variables, seen as features. These perturbations modify the quantile of marginal distributions while preserving the dependence structure. We project the initial distribution under a Wasserstein cost to design optimal perturbations. The proposed approach allows the inclusion of regularity conditions through piece-wise isotonic polynomials. The robustness analyses conducted on real case studies illustrate its potential flexibility and speed, which are essential for scaling up (increasing dimension and size of data). These methods and examples have been implemented within a dedicated openly accessible GitLab repository.

We will continue this work by examining the following points, which offer specific avenues for further research. The first technical problem is determining rules for the degree of the nonnegative polynomials in smooth approximations. It can be understood as the injection of prior information on the order of differentiability of the sought-after perturbed gqf. If $F^{\leftarrow}_P$ is supposedly regular (on all or some parts of $[0,1]$) and its differentiability can be estimated, a rule-of-thumb heuristic could be to impose the same differentiability on the resulting perturbed gqf $F^{\leftarrow}_Q$. In an ML framework, nonparametric approaches to isotonic regression of the marginal gqfs of $P$ can provide answers through statistical testing \citep{Durot2001,DaVeiga2012} or criteria enforcing a trade-off between approximation error and sparsity (e.g., inspired from AIC or BIC).

The second problem to be addressed is providing rules for modifying the dependence structure between features. In an UQ context, classes of parametric high-dimensional (vine) copulas have been recently explored by \citep{TORRE20191,Benou2020,PLISCHKE20191046,Bai2021} for several tasks, including SA. In the light of the dedicated literature, association and concordance measures appear as the most interpretable tools for these multivariate distributions (and therefore frequently used to incorporate expert opinion) \citep{Clemen1999,Zondervan2017,Benou2020}. Relevant classes of parametric vine copulas could be defined objectively by the Fréchet-Hoeffding bounds on the conditional bivariate dependence structures of such copulas or by computing extreme dependence structures relevant for the phenomenon of interest, as proposed by \citep{Benou2020}. The latter authors use a sparsity hypothesis to avoid prohibitive computational costs related to the curse of dimensionality. It echoes the strong need for sparsity assumptions about feature dependence in an ML framework, to which preliminary dimensionality reduction steps usually respond. In this context, many ML methods pay particular attention to the search for decorrelated, if not independent, features or small groups of features. It ensures the good behavior of the learning models and the relevance of the indicators allowing their interpretability \citep{Gilpin2018ExplainingEA}. For instance, it is well known that permutation-based indicators for random forests provide relevant information on the importance of input features solely when they are independent \cite{gregorutti2017}. New solutions have recently been developed to overcome this limitation \cite{Benard_shaff_2022}. In addition to computational limitations, it seems difficult, if not impossible, to produce understandable indicators when dealing with a large number of features. 

An alternative approach to account for feature dependence could be based on multivariate quantile extensions. Among the many approaches to defining such a notion, the most theoretically accomplished today is the one resulting from the concept of \emph{center-outward distribution function}, based on optimal transportation ideas. It was first proposed and studied by \citep{Chernozhukov2017} and \citep{Hallin2020DistributionAQ}. The resulting \emph{center-outward quantile function} offers desirable continuity and invertibility properties, which ensures the existence of closed and nested quantile contours \citep{FIGALLI2018413}. Suggesting that characteristics of this multivariate quantile function can be used to define variation classes seems natural to generalize our approach to multidimensional perturbations instead of focusing on marginal perturbations. Therefore we suggest that studying the interpretability of these features and the computational work required to handle quantile contours can be useful to improve our approach. 

On the subject of isotonic polynomials, one can notice that the solution of the proposed approximation scheme can not be differentiable at the points of the grid. To circumvent this drawback and enforce the derivability of the smoothed gqf on $[0,1]$, one could enforce additional constraints on the derivate of the polynomials. It would result in smoother perturbed gqf. Doing so would bear resemblance with splines \cite{Hastie2009}, and in particular isotonic splines, which have been extensively studied in the literature \cite{Schmidt1988, Fredenhagen1999, Wang2008}.

As stated in the proof of Theorem~\ref{thm:opti_pol_CCQP}, we chose to use the SOS representation of nonnegative polynomials and their subsequent characterization using semi-definite positive matrices. Works in \cite{Dette1997, Murray2016} could allow for less convoluted representations, leading to faster solving times. However, our developments allow for more apparent extensions to the definition of multivariate nonnegative polynomials. This decision has been made to account for using multivariate polynomials for smooth perturbation classes of quantile contours, which generalizes our approach to multidimensional perturbations.

Other spaces of functions can also be used for smoothing purposes. Following the work of \cite{Bagnell2015LearningPF}, abstract reproducing kernel Hilbert space of nonnegative functions can be reached through particular kernels. Hence, it would allow accessing different sets of nonnegative functions, whose regularities can be assessed through a thorough study of these kernels.

Finally, one of the primary motivations for using the $2$-Wasserstein distance as a projection metric is that it metricizes weak convergence on a broad set of probability measures. Other distances between probability measures are endowed with similar properties, such as the Prokhorov-Levy distance. An exciting improvement of our method would be to leverage the different relationships between such distances (see \cite{Gibbs2002}) to assess their use's relevancy for robustness studies.

%%%%%%%%%%%%%%%%%%%%%%%%%%%%%%%%%%%%%%%%
\section*{Acknowledgements}

Support from the ANR-3IA Artificial and Natural Intelligence Toulouse Institute is gratefully acknowledged. 

The authors warmly thank Jean-Bernard Lasserre (Institut de Math\'ematique de Toulouse) and Guillaume Dalle (CERMICS) for their help in solving the optimization problem at the heart of this work, as well as Cl\'ement B\'enesse (Institut de Math\'ematiques de Toulouse) and Antoine Paolini (UVSQ Universit\'e Paris Saclay) for their support on some mathematical aspects of this study.

%%%%%%%%%%%%%%%%%%%%%%%%%%%%%%%%%%%%%%%%%%%%%%%%%
%% References
\bibliographystyle{plain}
\bibliography{bib/references}

%%%%%%%%%%%%%%%%%%%%%%%%%%%%%%%%%%%%%%%%%%%%%%%%%
% Appendices
\appendix
%%%%%%%%%%%%%%%%%%%%%%%%%%%%%%%%%%%%%%%%%%%%%%%%%%%%%%%%%%%%%%%%%%%%%%%%%%%%%%%%
% Proofs
\section{Proofs}\label{sec:proof}

%%%%%%%%%%%%%%%%
% Prop 1

\begin{proof}[Proof of Remark~\ref{prop:uniqueQuantFunc}]
First, recall the following result from \citep{Dufour1995}:
\begin{thm}
If $G$ is a real-valued non-decreasing, left-continuous function with domain $(0,1)$, then there is a unique distribution function $F$ such that $G = F^{\leftarrow}$.
\end{thm}
Noticing that distribution functions (i.e., functions in $\mathcal{F}$) uniquely induce probability measures as their Lebegue-Stieltjes measures (see, \cite{Kallenberg2021}, Thm 2.14, p42) ultimately proves the proposition.
\end{proof}

%%%%%%%%%%%%%%%%
% Prop 2

\begin{proof}[Proof of Lemma~\ref{propo:nonemptyQ}]
Notice that if (\ref{eq:condNonempty}) is respected, then the constraints are non-decreasing. Then, there exists at least a function $F^{\leftarrow}$ in $\mathcal{F}^{\leftarrow}$ such that the constraints are respected (e.g., the linear interpolant of the constraints). So that, using Proposition~\ref{prop:uniqueQuantFunc}, there exists a probability measure with $F^{\leftarrow}$ as a generalized quantile function.
\end{proof}

%%%%%%%%%%%%%%%%
% Lemma 1
\begin{proof}[Proof of Lemma~\ref{lemma:quantile.shift.lemma}]
Since  $[\eta_0, \eta_1]$ is bounded, one can define a standardized intensity parameter $\theta \in \Theta=[-1, 1]$ as:
\begin{align*}
    \theta(b) &= \frac{p_{\alpha}-b}{p_{\alpha} - \eta_1} \mathds{1}_{\{ b > p_{\alpha}\}}(b)  +\frac{b-p_{\alpha} }{p_{\alpha}- \eta_0}\mathds{1}_{\{b < p_{\alpha}\}}(b).
\end{align*}
This intensity level can be interpreted as follows: 
\begin{itemize}
    \item $-1 \leq \theta < 0 \Leftrightarrow b< p_{\alpha}$: the (perturbed) distribution $Q$ such that $F^{\leftarrow}_{Q}(\alpha) =b$ is constrained to have a lower $\alpha$-quantile than $P$, down to $\eta_0$;
    \item $\theta=0 \Leftrightarrow b=p_{\alpha}$: $Q$ and $P$ share the same $\alpha$-quantile;
    \item $0<\theta \leq 1 \Leftrightarrow b>p_{\alpha}$: $Q$ is constrained to have a higher $\alpha$-quantile than $P$, up to $\eta_1$.
\end{itemize}
Equivalently, one can express $b$ in terms of $\theta$, which directly provides the expression of $b_{\alpha} (\bm\eta, \theta)$.
\end{proof}

%%%%%%%%%%%%%%%%
% Lemma 2
\begin{proof}[Proof of Lemma~\ref{lemma:dilatation.lemma}]
Preserving the midpoint of $\Omega_X$ while perturbing its width requires  that, for any couple $(b_0, b_1) \in \mathbb{R}^2$, that
$$\begin{cases}
{\displaystyle\frac{b_0+b_1}{2} =  \frac{\omega_0 + \omega_1}{2}} \\
{\displaystyle b_1-b_0 = \kappa (\omega_1 - \omega_0)}
\end{cases} \iff \begin{cases}
{\displaystyle b_1 = \frac{\omega_1(\kappa +1) - \omega_0(\kappa -1)}{2}} \\
{\displaystyle b_0 = \frac{\omega_0(\kappa +1) - \omega_1(\kappa-1)}{2}}
\end{cases}$$
where $\kappa \in [\frac{1}{\eta}, \eta]$. Using the transformation
\[ \theta(\kappa) = \left\{\begin{array}{lcr}
{\displaystyle - \frac{\kappa-1}{\frac{1}{\eta}-1}}   &  \text{if } & {\displaystyle  \frac{1}{\eta} \leq \kappa < 1} \\
{\displaystyle 0 }  & \text{if } &\kappa=1 \\
{\displaystyle \frac{\kappa-1}{\eta-1}}   &   \text{if } & 1 < \kappa < \eta
\end{array}\right. \] 
allows to define the formulas for $b_0$ and $b_1$ provided in the lemma statement.
The perturbation intensity $\theta$ can be interpreted as follows:
\begin{itemize}
    \item If $-1\leq \theta<0$, the application domain is narrowed, its width being divided by up to $\eta$;
    \item If $\theta =0$, the application domain boundaries are not perturbed;
    \item If $0< \theta \leq1$, the application domain is widened, its width being multiplied by up to $\eta$.
\end{itemize}
\end{proof}

%%%%%%%%%%%%%%%%
% Prop 3

\begin{proof}[Proof of Remark \ref{prop:copula-results}]
{\bf (i}) Consider the empirical situation, typically encountered in ML, where the marginal distributions of the inputs, respectively $P_1,\ldots,P_d$, are purely atomic. For $j\in\{1,\ldots,d\}$, denote $\{x_{j,i}\}_{1\leq i \leq n}$ the $j$th marginal sample of observations. From \citep{Nelsen2006} the dependence structure of the joint measure $P$ is known by its empirical copula $\hat{C}_P:[0,1]^d \to [0,1]$ defined as
\begin{eqnarray}
    \hat{C}_P(u_1, \dots, u_d) = {\displaystyle \frac{1}{n} \sum_{i=1}^n \prod_{j=1}^d \mathds{1}_{\left\{ \frac{R_{j,i}}{n} \leq u_j \right\}}(u_j)},
    \label{eq:empCop}
\end{eqnarray}
where $R_{j,k}$ denotes the rank of $x_{j,k}$ in $\{x_{j,i}\}_{1\leq i \leq n}$. 

Perturbing only the marginals of $P$ when solving the optimal projection problem (\ref{eq:pert_prob}) amounts apply to each marginal sample $\{x_{j,i}\}_{1\leq i \leq n}$ the transportation maps
\begin{eqnarray}
T_j = (F^{\leftarrow}_{Q_j} \circ F_{P_j}), \quad j=1,\dots,d. \label{marginal.perturbation.map}
\end{eqnarray}
and consequently the marginal empirical probability measure of the perturbed samples are defined as
$$\widetilde{Q}_j = \frac{1}{n} \sum_{i=1}^n \delta_{\tilde{x}_{j,i}}$$
where $\tilde{x}_{j,i} = T_j(x_{j,i})$ are the perturbed datapoints. Since the marginal cdf and gqf in (\ref{marginal.perturbation.map}) are non-decreasing (ie., monotonic) functions, their composition is non-decreasing as well. Since  monotonic functions preserve orders, the ranks of $x_{j,i}$ and  $\tilde{x}_{j,i}$ are invariant, for $(i,j)\in\{1,\ldots,n\}\times\{1,\ldots,d\}$. From (\ref{eq:empCop}), it implies that the empirical copula of $Q$ is $\hat{C}_Q=\hat{C}_P$. This traduces, for instance, the invariance of  Spearman correlation matrices between the initial and perturbed datasets. 

The transportation map $T_j$ defined by (\ref{eq:pertMap}) is not the unique non-decreasing transportation map between $P$ and $Q$ (\citep{Santambrogio2015}, Chap. 2). However, $T_j$ is indeed an optimal perturbation plan (in Monge's sense) for a wide variety of transportation costs between the marginals $P_j$ and $\widetilde{Q}_j$ (\citep{Peyre2019}, Remark 2.30). Moreover notice that for any $j=1,\dots,d$
$$\widetilde{Q}_j \underset{n \rightarrow \infty}{\longrightarrow} Q_j$$
since the distribution of $F_{P_j}(X_j)$ converges toward a uniform distribution on $[0,1]$.

\noindent {\bf (ii})  Now, for the sake of conciseness, let $P$ denote any univariate probability measure and $Q$ its optimally perturbed counterpart. If $F^{\leftarrow}_Q$ is strictly increasing then from \cite{Dufour1995}, for all $u \in [0,1]$
$$(F_Q \circ F^{\leftarrow}_Q) (u) = u$$
Now denote the transportation map 
\begin{eqnarray*}
    T: \mathcal{X} & \to & \mathcal{X} \nonumber \\
    x & \mapsto & (F^{\leftarrow}_{Q} \circ F_{P})(x)
    \label{eq:pertMap2}
\end{eqnarray*}
Let $X\sim P$ and define $U_P=F_P(X)$. Then, one has, if $P$ is atomless, that $T(X)\sim Q$ and
$$U_Q = F_Q(T(X)) = (F_Q \circ F^{\leftarrow}_Q \circ F_P)(X) = F_P(X) = U_P \text{ a.s.}$$

Now, for a multivariate probability measure $P$ with marginals $P_1, \dots, P_d$, and respective optimally perturbed probability measures $Q_1, \dots, Q_d$, this entails that, for $i=1,\dots, d$:
$$U_{Q_i} = F_{Q_i}(T(X_i)) = F_{P_i}(X_i) = U_{P_i} \text{ a.s.}$$
and hence, the random vectors $U_Q$ and $U_P$ are equal almost surely. Subsequently, for any $u\in [0,1]^d$, $C_P(u) = C_Q(u)$. Hence the $X$ and $T(X)$ have the same copula.
\end{proof}

%%%%%%%%%%%%%%%%
% Prop 4

\begin{proof}[Proof of Proposition~\ref{propo:wass_l2}]
First, one can notice that, for any probability measure $G \in \mathcal{P}(\mathbb{R})$:
$$F^{\leftarrow}_G = F^{\rightarrow}_G \quad \mu-\text{almost everywhere (a.e.)} $$
where $\mu$ denotes the Lebesgue measure on $[0,1]$ (see, \cite{Fortelle2015}). This entails that:
\begin{align*}
    \int_0^1 \left(F^{\leftarrow}_G - F^{\rightarrow}_P\right)^2 d\mu &= \int_0^1\left(F^{\rightarrow}_G  - F^{\rightarrow}_P\right)^2 d\mu\\
    &= W^{2}_2(P,G)
\end{align*}
by a continuous composition of $\mu$-a.e. equal functions, and by identification with the definition of the $W_2$ distance for probability measures supported on $\mathbb{R}$.

In the following of this proof, one refers to the projection problem  (\ref{eq:init_pb}) as the ``Wasserstein projection'', and the projection problem (\ref{eq:equiv_pb}) as the ``$L^2$ projection''.

The proposition can be proven in two steps. First, if $Q$ is the solution of the Wasserstein projection, then its quantile function $F^{\leftarrow}_Q$ is necessarily the solution of the $L^2$ projection. This is due to the fact that $F^{\leftarrow}_Q$ is uniquely characterized by $Q$, as well as the particular case of the $2$-Wasserstein distance for measures supported on $\mathbb{R}$ (see  Definition~\ref{def:wass_1D}).

Second, let $Q$ be any probability measure in $\mathcal{P}_2(\mathbb{R})$, and denote $F^{\leftarrow}_Q$ its characterizing quantile function. Assume that $F^{\leftarrow}_Q$ is the solution of the $L^2$ projection. Since $F^{\leftarrow}_Q$ characterizes uniquely $Q$, thanks to Proposition~\ref{prop:uniqueQuantFunc}, one has that $Q$ is necessarily the solution of the Wasserstein projection.
\end{proof}

%%%%%%%%%%%%%%%%
% Prop 6

\begin{proof}[Proof of Proposition~\ref{prop:analytRes}]
First, note that the intervals $A_i, i=1,\dots,K$ are disjoint. Moreover for any $i=1,\dots,K-1$, consider the four cases:
\begin{enumerate}
    \item If $\alpha_i < \beta_i< \alpha_{i+1}$ and, then $A_i = (\alpha_i, \beta_i]$;
    \item If $\beta_i < \alpha_i<\beta_{i+1}$ and, then $A_i = (\beta_i, \alpha_i]$;
    \item If $\alpha_i < \beta_i$ and assume that $\alpha_{i+j} < \beta_{i+j-1}$ for $j=1,\dots, m$ where $m \leq K-i$ is some non-negative integer, then $A_i = (\alpha_i, \alpha_{i+1}]$, additionally for $j=i+1, \dots,i+m-1$, $A_j=(\alpha_j, \alpha_{j+1}]$ and finally $A_{i+m} = (\alpha_{i+m}, \beta_{i+m}]$;
    \item If $\beta_i < \alpha_i$ and assume that $\alpha_{i+j} < \beta_{i+j+1}$ for $j=1,\dots, m$ where $m \leq K-i-1$ is some non-negative integer, then $A_i = (\beta_i, \alpha_{i}]$ and for $j=i+1, \dots, i+m$, $A_j=(\alpha_{j-1}, \alpha_{j}]$.
\end{enumerate}
The integral can be decomposed as follows:
$$\int_0^1 \left(L(x) - F^{\rightarrow}_P(x) \right)^2 dx = \int_{\overline{A}} \left(L(x) - F^{\rightarrow}_P(x) \right)^2 dx + \sum_{i=1}^K \int_{A_i} \left(L(x) - F^{\rightarrow}_P(x) \right)^2 dx$$
where
$$\int_{\overline{A}} \left(L(x) - F^{\rightarrow}_P(x) \right)^2 dx \geq 0.$$
Since the quantile constraints are of the form:
$$L(\alpha_i) \leq b_i \leq L\left(\alpha_{i}^+\right).$$
we can always write $L(y)=b_i + h(y)$ for $y \in A_i$, and where $h$ is an non-decreasing, left-continuous function. Moreover, note that:
\begin{itemize}
    \item $h(y)$ is non-negative, and $F^{\rightarrow}_P(y) - b_i \leq 0$ if $A_i$ falls in cases 2. and 4.
    \item $h(y)$ is non-positive, and $F^{\rightarrow}_P(y) - b_i \geq 0$ if $A_i$ falls in cases 1. and 3.
\end{itemize}
Then we have:
\begin{align*}
    \int_{A_i} \left(L(x) - F^{\rightarrow}_P(x) \right)^2 dx &= \int_{A_i} \left(L(x) - b_i - h(y) \right)^2 dx \\
    &= \int_{A_i} \left(F^{\rightarrow}_P(x) - b_i\right)^2 dx + \int_{A_i} h(x)^2 dx \\
    & \ \ \ \ -2\int_{A_i} h(x)\left(F^{\rightarrow}_P(x) - b_i\right)dx \\
    &\geq \int_{A_i} \left(F^{\rightarrow}_P(x) - b_i\right)^2 dx
\end{align*}
since $h(x)$ and $F^{\rightarrow}_P(x) - b_i$ have different sign. Due to the constraint and the left-continuous non-decreasing nature of $L$, this bound is tight and is attained if and only if $h(y)=0$ for all $y \in A_i$.
Globally, this entails that
$$\int_0^1 \left(L(x) - F^{\rightarrow}_P(x) \right)^2 dx \geq \sum_{i=1}^K \int_{A_i} \left(F^{\rightarrow}_P(x) - b_i\right)^2 dx$$
and this tight bound is uniquely attained by the left-continuous non-decreasing function defined as
\begin{align*}
    F^{\leftarrow}_Q(y) = \begin{cases}
    F^{\rightarrow}_P(y) & \text{if } y \in \overline{A} \\
    b_i & \text{if } y \in A_i, \quad i=1,\dots, K.
    \end{cases}
\end{align*}
\end{proof}

%%%%%%%%%%%%%%%%
% Prop 7
\paragraph{Proof of Theorem~\ref{thm:opti_pol_CCQP} (ingredients)}\label{proof:opti_pol_CCQP-partI}
This proof relies on the following results that can be found in \cite{Parrilo2010, Parrilo2012, Lasserre2015}, and further recalled in \cite{Sobrie2018}. They involve sum-of-squares (SOS) polynomials, which can be defined as follows.
\begin{dfi}[SOS polynomials]
A polynomial $S$ of even degree $p$ is said to be a SOS polynomial if, for $m \in \mathbb{N}^*$, there exists $s_1, \dots, s_m$ polynomials of degree at most equal to $\frac{d}{2}$, and such that, $\forall x \in \mathbb{R}$:
$$S(x) = \sum_{i=1}^m \bigl(s_i(x)\bigr)^2.$$
\label{def:sos_poly}
\end{dfi}

%%%%%%%%%%%%%%%%%%%%%%%%%%%%%%%%%%%%%%%%
\begin{thm}
Let $t_0, t_1 \in \mathbb{R}$ such that $t_0 < t_1$, and let $p \in \mathbb{N}^*$. 
\begin{description}
\item[(i)] A univariate polynomial $S$ of even degree $d=2p$ is non-negative on $[t_0, t_1]$ if and only if it can be written as, $\forall x \in [t_0, t_1]$
$$S(x) = Z(x) + (x-t_0)(t_1 -x) W(x) $$
where $Z$ is a SOS polynomial of degree at most equal to $d$, and $W$ is an SOS polynomial of degree at most equal to $d-2$.
\item[(i)] An univariate polynomial $S$ of odd degree $d=2p+1$ is non-negative on $[t_0, t_1]$ if and only if it can be written as, $\forall x \in [t_0, t_1]$
$$S(x) = (x-t_0)Z(x) + (t_1 -x)W(x)$$
where $Z, W$ are SOS polynomials of degree at most equal to $d$.
\end{description}
\label{thm:posPoly_interv}
\end{thm}

It is important to note that Theorem~\ref{thm:posPoly_interv} is quite general, in the sense that it allows for extensions to multivariate polynomials (i.e., polynomials taking values from $\mathbb{R}^d$). As pointed out in \cite{Dette1997} (Thm. 1.4.2), nonnegative polynomials on compact intervals can also be defined as a linear combination of squared polynomials. It may facilitate the identification of the nonnegative polynomials' coefficients, as done in \cite{Murray2016} in the context of statistical learning. However, for the sake of potential future genericity, we chose to leverage the direct powerful link between SOS polynomials and semi definite positive matrices, as expressed in the following theorem.

\begin{thm}
Let $S$ be an univariate polynomial of even degree $d = 2p$, with coefficients $s=(s_0, \dots, s_{d})$, and denote $x_p$ the usual monomial basis of polynomials of degree at most equal to $p$, i.e., $x_p=(1, x, x^2, \dots, x^{p-1}, x^p)^\top$. $S$ is an SOS polynomial if and only if there exists a $(p \times p)$ symmetric semi definite positive (SDP) matrix
$$\Gamma = \begin{bmatrix}\Gamma_{ij} \end{bmatrix}_{i,j=1, \dots, p} $$
that satisfies, $\forall x \in \mathbb{R}$,
$$S(x) = x_p^\top \Gamma x_p.$$
Moreover, for $k=0, \dots, d$, let $\mathbb{I}_k^p$ be the $(p \times p)$ matrix defined by, for $i,j=1,\dots,p$:
$$\begin{bmatrix} \mathbb{I}_k^p \end{bmatrix}_{i,j} = \mathds{1}_{\{i+j = k+2 \}}(i,j). $$
Then one additionally has that, for $i=0,\dots, d$
\begin{equation}
    s_i = \langle \mathbb{I}_i^p, \Gamma \rangle_F = \sum_{j+k = i+2} \Gamma_{j,k}
    \label{eq:coefToSDP}
\end{equation}
where, $\langle .,. \rangle_F$ denotes the Frobenius norm on matrices.
\label{thm:SOS_sdpRep}
\end{thm}

\begin{thm}
Let $\mathbb{S}_n$ the subspace of real-valued symmetric matrices, in the vector space of square matrices. The set of symmetric SDP matrices $\Sigma_N$ is  a proper cone in $\mathbb{S}_n$, and thus is a closed convex set.
\label{thm:symat_sdp}
\end{thm}

A few results on the preservation of convexity of sets under transformations are also required. These lemmas can be found in \cite{Bertsekas2016}.

\begin{lme}[Linear maps preserve convexity]
Let $V,W$ be two vector spaces over the same field $F$. Let $T : V \rightarrow W$ be a linear map, and let $\mathcal{C} \subset V$ be a convex set. Then the image of $\mathcal{C}$ under $T$, i.e., :
$$ T(C) = \{ T(x) \in W \mid x \in C \subset V \}$$
is also a convex set.
\label{lme:linmaps_conv}
\end{lme}

\begin{lme}[Cartesian product of convex sets is a convex set]
Let $C_1$ be a subset of $\mathbb{R}^m$ and $C_2$ be a convex subset of $\mathbb{R}^n$. Then, the Cartesian product
$C_1 \times C_2$ is a convex subset of  $\mathbb{R}^m \times \mathbb{R}^n.$
\label{lme:cart_conv}
\end{lme}
Two additional results, proven beneath, are required before proceeding to the proof of Theorem \ref{thm:opti_pol_CCQP}.
\begin{lme}
The mapping in (\ref{eq:coefToSDP}), $V : \mathbb{S}_p \rightarrow \mathbb{R}^{2p}$, defined, for any $\Gamma \in \mathbb{S}_p$, as:
$$V(\Gamma) = \left( \sum_{j+k = i+2} \Gamma_{j,k}\right)_{i=0,\dots,2p} $$
is linear.
\label{lme:map_linear}
\end{lme}
\begin{proof}[Proof of Lemma~\ref{lme:map_linear}]
We need to show that:
\begin{itemize}
    \item For $A,B \in \mathbb{S}_p$, $T(A + B) = T(A) + T(B)$;
    \item For $\alpha \in \mathbb{R}$, $\Gamma \in \mathbb{S}_p$, $T(\alpha\Gamma) = \alpha T(\Gamma)$. 
\end{itemize}
First, we have, for $i=0,\dots, 2p$:
\begin{align*}
    \Bigl[T(A+B)\Bigr]_i &= \sum_{j+k = 2p-i} \Bigl[A+B\Bigr]_{jk} \\
    &= \sum_{j+k = i+2} A_{jk}+B_{jk} \\
    &= \sum_{j+k = i+2} A_{jk} + \sum_{j+k = i+2} B_{jk} \\
    &= \Bigl[T(A)\Bigr]_i + \Bigl[T(B)\Bigr]_i
\end{align*}
since it holds for $i=0, \dots, 2p$, it entails:
$$T(A+B) = T(A) + T(B).$$
Moreover, we have, for $i=0,\dots, 2p$:
\begin{align*}
    \Bigl[T(\alpha \Gamma)\Bigr]_i &=  \sum_{j+k = i+2} \alpha \Gamma_{jk} \\
    &= \alpha  \Bigl[T(\Gamma)\Bigr]_i
\end{align*}
and since it holds for $i=0, \dots, 2p$, it entails:
$$T(\alpha \Gamma) = \alpha T(\Gamma).$$
Hence $T$ is a linear map between $\mathbb{S}_p$ and $\mathbb{R}^{2p}$.
\end{proof}

\begin{lme}
Let $S$ be a univariate polynomial of degree $d$ and $s=(s_0, \dots, s_d)^\top \in \mathbb{R}^{d+1}$ its coefficients. Let $S'$ be its derivative, i.e., a polynomial of degree $d-1$, with coefficients $\breve{s} = (s_1, \dots, s_d)^\top \in \mathbb{R}^{d}$. Let $Z$ and $W$ be SOS polynomials, with coefficients $z$ and $w$, and assume that $S'$ is non-negative on $[t_0, t_1]$ as a combination of $Z$ and $W$ as in Theorem~\ref{thm:posPoly_interv}. Moreover, let 
$$D = \operatorname{diag}(1,2,\dots,d) $$
be the $(d \times d)$ diagonal matrix with $(1,\dots, d)$ as a diagonal elements and denote the bloc-matrices
\begin{equation*}
    \overI_{i,d} = \begin{pmatrix}
    I_d \\
    \mathbf{0}_{i,d}
    \end{pmatrix}, \quad
   \underI_{i,d} = \begin{pmatrix}
        \mathbf{0}_{i,d} \\
        I_d 
    \end{pmatrix}, \quad         
    \overunderI_{i,d} = \begin{pmatrix}
        \mathbf{0}_{i,d} \\
        I_d \\
        \mathbf{0}_{i,d}
    \end{pmatrix}
\end{equation*}
where $\mathbf{0}_{i,d}$ denotes the $(i \times d)$ matrix of zeros, and $I_d$ be the $(d \times d)$ identity matrix.
If $d$ is odd, then $z \in \mathbb{R}^d$ and $w \in \mathbb{R}^{d-2}$ and furthermore 
$$\breve{s} = Az + Bw$$
where $A$ and $B$ are $(d\times d)$ and $(d \times d-2)$ matrices, respectively.
If the degree $d$ of $S$ is even, one has that $z,w \in \mathbb{R}^{d-1}$ and furthermore:
$$\breve{s} = Cz + Dw.$$
where $C$ and $D$ are $(d \times d-1)$ matrices.
More specifically,
\begin{align*}
    A &= \mathcal{D}_d^{-1}, & B &= \mathcal{D}_d^{-1} \left( (t_0 + t_1) \overunderI_{1,d-2} - \underI_{2,d-2} - t_0t_1 \overI_{2,d-2} \right), \\
    C &= \mathcal{D}_d^{-1} \left(\underI_{1,d-1}-t_0\overI_{1,d-1} \right), & D &= \mathcal{D}_d^{-1} \left(t_1\overI_{1, d-1} - \underI_{1,d-1} \right).
\end{align*}
\label{lme:SOStoDeriv}
\end{lme}

\begin{proof}[Proof of Lemma~\ref{lme:SOStoDeriv}]
First, assume that $S$ is a polynomial of odd degree $d=2p+1$, meaning that its derivative, $S'$, is a polynomial of even degree $2p$. From Theorem~\ref{thm:posPoly_interv}, one has that $S'(x)$ is positive on an interval $[t_0, t_1]$ if and only if it can be expressed as :
\begin{equation*}
    S'(x) = Z(x) + (x-t_0)(t_1 - x)W(x)
\end{equation*}
where $Z$ is an SOS polynomial of degree at most equal to $d-1$ and $W$ is an SOS polynomial of degree at most equal to $d-3$. Denote $\breve{s} = (s_1, \dots, s_d) \in R^d$ the coefficients of $S'$ and $z = (z_1, \dots, z_d) \in \mathbb{R}^d$ and $w = (w_1, \dots, w_{d-2)} \in \mathbb{R}^{d-2}$ the coefficients of $Z$ and $W$ respectively. One has that :
\begin{align*}
    S'(x) &= \sum_{i=1}^d is_ix^{i-1} \\
    &= \sum_{j=0}^{d-1} (j+1) s_{i+1}x^i
\end{align*}
and if $S'$ is assumed to be non-negative on $[t_0, t_1]$
\begin{align*}
    S'(x) &=  Z(x) + (x-t_0)(t_1 - x)W(x) \\
    &= \sum_{j=0}^{d-1} z_{j+1}x^j + (-x^2 + (t_0 + t_1)x - t_0t_1)\sum_{j=0}^{d-3}w_{j+1}x^j
\end{align*}
leading to the following identification :
\begin{equation*}
\begin{cases}
    s_1 = z_1 - t_0t_1w_1 \\
    s_2 = \frac{1}{2} \left(z_2 -t_0t_1w_2 +(t_0+t_1)w_1 \right) \\
    s_i = \frac{1}{i} \left(z_i - t_0t_1w_i + (t_0+t_1)w_{i-1} - w_{i-2} \right), & \text{for } i=3,\dots,d-2 \\
    s_{d-1} = \frac{1}{d-1} \left(z_{d-1} + (t_0 + t_1)w_{d-2} - w_{d-3} \right) \\
    s_{d} = \frac{1}{d} \left(z_{d-1}  - w_{d-2} \right),
\end{cases}
\end{equation*}
or, written in a matrix form:
$$\breve{s} =  \mathcal{D}_d^{-1}\left( z + \left((t_0 + t_1) \overunderI_{1,d-2} - \underI_{2,d-2} - t_0t_1 \overI_{2,d-2}\right)w\right).$$

If $S$ is assumed to be a polynomial of even degree $d=2p$, $S'$ is necessarily odd degree. From Theorem~\ref{thm:posPoly_interv}, one has that $S'(x)$ is positive on an interval $[t_0, t_1]$ if and only if it can be expressed as :
\begin{equation*}
    S'(x) = (x-t_0)Z(x) + (t_1 -x)W(x)
\end{equation*}
where $Z$ and $W$ are SOS polynomials of degree at most equal to $d-2$ with $z = (z_1, \dots, z_{d-1}) \in \mathbb{R}^{d-1}$ and $w = (w_1, \dots, w_{d-1)} \in \mathbb{R}^{d-1}$ as coefficients, respectively. It leads to the following identification:
\begin{equation*}
\begin{cases}
    s_1 = - t_0z_1 + t_1w_1 \\
    s_i = \frac{1}{i} \left(z_{i-1} - t_0z_{i} + t_1w_{i} - w_{i-1} \right) & \text{for } i=2,\dots,d-1\\
    s_{d} = \frac{1}{d} \left(z_{d-1} - w_{d-1} \right),
\end{cases}
\end{equation*}
which can be written in matrix form as 
$$\breve{s} =  \mathcal{D}_d^{-1}\left(\left(\underI_{1,d-1}-t_0\overI_{1,d-1} \right)z + \left(t_1\overI_{1, d-1} - \underI_{1,d-1} \right)w \right).$$
\end{proof}

%%%%%%%%%%%%%%%%%%%%%%%%%
We can now proceed to prove Theorem~\ref{thm:opti_pol_CCQP}.
\begin{proof}[\textbf{Proof of Theorem~\ref{thm:opti_pol_CCQP}} \textbf{(rationale)}]
This rationale can be broken down in two steps: {\bf (a)} proving that the objective function (\ref{eq:polyPb_interv}) can indeed be written in a quadratic form, and:{\bf (b)} proving that the problem  constraints form a feasible set in $\mathbb{R}^{d+1}$ which is closed and convex.\\

\noindent {\bf (a)} Notice first that the initial objective function
$$\int_{t_0}^{t_1} (L(x) - F^{\rightarrow}_P(x))^2 dx $$
where $L \in \mathbb{R}[x]_{\leq d}$ with coefficients $s \in \mathbb{R}^{d+1}$, can be rewritten as:
\begin{align*}
    \int_{t_0}^{t_1} (F^{\rightarrow}_P(x) - L(x))^2 dx &= \int_{t_0}^{t_1} (\sum_{i=0}^{d}s_ix^i - F^{\rightarrow}_P(x))^2 dx\\
    &= \int_{t_0}^{t_1} \left(\left(\sum_{i=0}^{d}s_ix^i\right)^2 +\left(F^{\rightarrow}_P(x)\right)^2 - 2\sum_{i=0}^{d}s_ix^iF^{\rightarrow}_P(x) \right) dx \\
    &= \int_{t_0}^{t_1} \left(\sum_{i=0}^{d}s_ix^i\right)^2 dx - 2\sum_{i=0}^{d}s_i \int_{t_0}^{t_1}x^iF^{\rightarrow}_P(x)dx \\ 
    & \ \ + \int_{t_0}^{t_1} \left(F^{\rightarrow}_P(x)\right)^2dx.
\end{align*}
Note that
\begin{align*}
    \int_{t_0}^{t_1} \left(\sum_{i=0}^{d}s_ix^i\right)^2 dx &= \sum_{i=0}^d \sum_{j=0}^d s_is_j \int_{t_0}^{t_1}x^{i+j}dx \\
    &= s^\top M s
\end{align*}
where $M$ is the moment matrix of the Lebesgue measure on $[t_0, t_1]$, i.e., defined entry-wise, for $i,j=1,\dots,d+1$ as
$$M_{ij} = \int_{t_0}^{t_1} x^{i+j-2}dx = \frac{(t_1)^{i+j-1} - (t_0)^{i+j-1}}{i+j-1}.$$
and further notice that $M$ is thus positive definite since, for any $u \in \mathbb{R}^{d+1}$,
$$u^\top M u = \int_{t_0}^{t_1} \left(\sum_{i=0}^d u_{i+1}x^i \right)^2 dx \geq 0$$
is always non-negative, and equal to $0$ if and only if $u_i = 0, i=1,\dots, d+1$. Moreover, note that:
\begin{align*}
    \sum_{i=0}^{d}s_i \int_{t_0}^{t_1}x^iF^{\rightarrow}_P(x)dx &= s^\top r
\end{align*}
where $r \in \mathbb{R}^{d+1}$ is the moment vector of $G$ with respect to the Lebesgue measure on $[t_0, t_1]$, defined for $i=0,\dots, d$ as:
$$r_i = \int_{t_0}^{t_1} x^iF^{\rightarrow}_P(x)dx$$
Since a polynomial is completely characterized by its coefficients, searching for:
$$S^* = \underset{L \in \mathbb{R}[x]_{{\leq d}}}{\text{argmin }}\int_{t_0}^{t_1} (L(x) - F^{\rightarrow}_P(x))^2 dx$$
is equivalent to finding the coefficients $s^*$ of $S^*$, i.e.,
$$ s^* = \underset{s \in \mathbb{R}^{p+1}}{\text{argmin }}  s^\top M s -2s^\top r$$
and thus proving the first part of the proposition. \\

\noindent {\bf (b)} Notice that the interpolation constraints
$$\begin{cases}
    S(t_0) = b_0 \\
    S(t_1) = b_1
\end{cases}$$
can be written as
$$\begin{cases}
    s^\top \mathbf{t_{0}}^d = b_0 \\
    s^\top \mathbf{t_{1}}^d  = b_1 
\end{cases}$$
where, for $a \in \mathbb{R}$, one denote $\mathbf{a}^d$ the vector of powers of $a$ up to $d$, i.e., $\mathbf{a}^d = (1, a, \dots, a^{d-1}, a^d) \in \mathbb{R}^{d+1}$. Moreover, by letting:
$$\mathbf{T} = \begin{pmatrix} \mathbf{t_{0}}^d \\ \mathbf{t_{1}}^d \end{pmatrix}, \quad b= \begin{pmatrix} b_0 \\ b_1 \end{pmatrix}, $$
where $\mathbf{T}$ is a $(2 \times d+1)$ bloc-matrix, the constraint can be written as:
$$Ts = b.$$
Furthermore, notice that 
$$\mathcal{C}_0 = \{s \in \mathbb{R}^{d+1} \mid Ts =b\}$$
is a convex subset of $\mathbb{R}^{d+1}$, since the equality constraints are linear.
Concerning the monotonicity constraint
$$S'(x) \geq 0, \quad \forall x \in [t_0, t_1],$$
from Lemma~\ref{lme:SOStoDeriv} we can quite generically write
\begin{equation*}
    \begin{pmatrix}s_d \\ \vdots \\ s_1\end{pmatrix} = T_0(z,w) := Az + Bw
\end{equation*}
where $z$ and $w$ are the coefficient of SOS polynomials of degrees depending on $d$. Additionally, notice that the mapping $T_0 : \mathbb{R}^d \times \mathbb{R}^{d-2} \rightarrow \mathbb{R}^d $ is linear. 
Next, let $V_1 : \mathbb{S}_p \rightarrow \mathbb{R}^{2p}$, and $V_2 : \mathbb{S}_{q} \rightarrow \mathbb{R}^{2q}$ be defined as in (\ref{eq:coefToSDP}), where $p = {d-1}/{2}$  and $q = {d-3}/{2}$ if $d$ is odd, or  $p={d-2}/{2}$ and $q={d-2}/{2}$ if $d$ is even, and note that both mappings are linear thanks to Lemma~\ref{lme:map_linear}. 

Moreover, denote the following sets:
$$\mathcal{Z} = \{V_1(E) \mid E \in \Sigma_{p} \}, \quad \mathcal{W} = \{V_2(E) \mid E \in \Sigma_{p-1} \}$$
and notice the polynomial $Z$ (resp. $W$) is SOS if and only its coefficients $z$ (resp. $w$) are in $\mathcal{Z}$ (resp. $\mathcal{W}$) thanks to Theorem~\ref{thm:symat_sdp}. In addition again, notice that, thanks to Lemma~\ref{lme:linmaps_conv}, and due to the fact that $\Sigma_p$ is a closed convex set in $\mathbb{S}_p$ as per Theorem~\ref{thm:symat_sdp}, both $\mathcal{Z}$ and $\mathcal{W}$ are convex subsets of $\mathbb{R}^{2p}$ and $\mathbb{R}^{2q}$ respectively. Besides, thanks to Lemma~\ref{lme:cart_conv}, the set $\mathcal{Z} \times \mathcal{W}$ is a convex subset of $\mathbb{R}^{2p} \times \mathbb{R}^{2q}$ as well. Moreover, let
$$\mathcal{C}_1 = \left\{ \begin{pmatrix}T_0(w,z) \\ x \end{pmatrix} \in \mathbb{R}^{d+1} \mid x \in \mathbb{R}, \quad (z,w) \in  \mathcal{Z} \times \mathcal{W}\right\} $$
and note that it is a convex subset of $\mathbb{R}^{d+1}$ due to the fact that $T_0$ is a linear map. 

Finally, since both $\mathcal{C}_0$ and $\mathcal{C}_1$ are convex sets, their intersection:
$$\mathcal{K} = \mathcal{C}_0 \cap \mathcal{C}_1$$
is as well, and note that any element $s \in \mathcal{K}$ are the coefficients of a polynomial respecting both equality and monotonicity constraints. In other words, $\mathcal{K}$ is the feasible set of coefficients of the initial optimization problem.
\end{proof}

%%%%%%%%%%%%%%%%%%%%%%%%%%%%%%%%%%%%%%%%%%%%%%%%%%%%%%%%%%%%%%%%%%%%%%%%%%%%%%%%%%%%%%%%%%%%
%%%%%%%%%%%%%%%%%%%%%%%%%%%%%%%%%%%%%%%%%%%%%%%%%%%%%%%%%%%%%%%%%%%%%%%%%%%%%%%%%%%%%%%%%%%%%%%%%%%%%%%%%%%%%%%%%%%%%%%%%%%%%%%
% Computing stuff
%\section{Computational details}

%%%%%%%%%%%%%%%%
% Lebesgue Moment Matrix
% \subsection{Moment matrix of the Lebesgue measure}\label{apdx:compute_M}
% The following \texttt{R} code snippet defines a function for computing the moment matrix of the Lebesgue measure on any interval $[a,b]$.
% \begin{lstlisting}[language=R, caption=Moment matrix of the Lebesgue measure on an interval.]
% ###############################################
% # Lebesgue Moment Matrix on an interval [a,b]
% mom_mat<-function(a,b,d){
% # a,b : Upper and lower bound of the interval
% # d : Degree up to which the moment matrix is computed
%     M<-matrix(NA, nrow=d+1, ncol=d+1)
%     for(i in 1:(d+1)){
%         for(j in 1:(d+1)){
%             M[i,j] = (b^{i+j-1} - a^{i+j-1})/(i+j-1)
%         }
%     }
%     return(M)
% }
% \end{lstlisting}

%%%%%%%%%%%%%%%%
% Quantile Func moment vector
\section{Computing moment vector of arbitrary quantile functions}\label{apdx:compute_r}
One wishes here at computing the vector described in (\ref{eq:compute_r}). In the case where $P$ is an empirical measure built on a $n$-sample, one has that for $[t_0, t_1] \in [0,1]$, $i=0,\dots,p$:
\begin{align*}
    r_i =& \frac{1}{i+1}\Biggl[\sum_{j \in J} \frac{X_{(j)}}{n^{i+1}} \left(\left(j+1\right)^{i+1} - j^{i+1} \right) \\
    &+ X_{\left(\overline{j}\right)} \left(t_1^{i+1} - \left(\frac{\overline{j}}{n}\right)^{i+1} \right) \\
    &+ X_{\left(\underline{j}-1\right)} \left(\left(\frac{\underline{j}}{n}\right)^{i+1} -t_0^{i+1}  \right)
    \Biggr]
\end{align*}
where $J = \left\{i \in \mathbb{N} \mid \lfloor n t_0 \rfloor < i < \lfloor n t_1 \rfloor \right\}$, $\overline{j}= \lfloor t_1 n \rfloor$, $\underline{j} = \lfloor t_0 n \rfloor +1$, and where $X_{(j)}$ denotes the $j$-th order statistic of the observe sample. In cases where $F^{\leftarrow}_P$ is continuous, it is possible to use numerical quadrature methods in order to evaluate each integral composing the elements $r_i$ of $r$.

\end{document}